\documentclass[a4paper,11pt]{article}

\usepackage{enumerate}
\usepackage{algpseudocode}
\usepackage{algorithm}
\usepackage{amsfonts}
\usepackage[leqno]{amsmath}
\usepackage{amssymb}
\usepackage{amstext}
\usepackage{amsthm}
\usepackage[toc,page]{appendix}
\usepackage{array}
\usepackage{authblk}
\usepackage{bm}
\usepackage{booktabs}
\usepackage[font=small]{caption}
\usepackage{cite}
\usepackage{datetime}
\usepackage{etoolbox}
\usepackage[T1]{fontenc}
\usepackage{framed}
\usepackage{fullpage}
\usepackage[margin=0.85in]{geometry}
\usepackage{graphics}
\usepackage{graphicx}
\usepackage{hyperref}
\usepackage{color}
\usepackage[utf8]{inputenc}
\usepackage{mathrsfs}
\usepackage[framemethod=TikZ]{mdframed}
\usepackage{microtype}
\usepackage{multirow}
\usepackage{nicefrac}
\usepackage{pdflscape}
\usepackage{pgfplots}
\usepackage{rotating}
\usepackage{setspace}
\usepackage{textcomp}
\usepackage{tikz}
\usepackage{times}
\usepackage{units}
\usepackage{url}
\usepackage{dsfont}
\usepackage{subfig}

\hypersetup{colorlinks,breaklinks,linkcolor=blue,urlcolor=blue,anchorcolor=blue,citecolor=blue}

\newcommand{\vt}{\vartheta}
\renewcommand{\L}{\mathcal{L}}
\renewcommand{\phi}{\varphi}

\renewcommand{\div}{\text{div}}
\newcommand{\one}{\mathds{1}}
\newcommand{\wO}{\widetilde{\Omega}}
\newcommand{\wc}{\widetilde{c}}

\newcommand{\R}{\mathbb{R}}

\newcommand{\Z}{\mathbb{Z}}
\newcommand{\e}{\mathbf{e}}
\renewcommand{\P}{\mathbb{P}}
\newcommand{\F}{\mathcal{F}}
\newcommand{\E}{\mathbb{E}}
\renewcommand{\bar}[1]{\overline{#1}}
\renewcommand{\tilde}[1]{\widetilde{#1}}

\numberwithin{equation}{section}

\usetikzlibrary{shapes,arrows,plotmarks,matrix,positioning,fit,calc,3d,patterns,decorations.pathreplacing}
\usepgfplotslibrary{groupplots,fillbetween,decorations.softclip}
\pgfplotsset{compat=newest}
\providecommand{\keywords}[1]{\textbf{\textit{\noindent  Keywords and phrases. }} #1}
\providecommand{\subjclass}[1]{\textbf{\textit{\noindent Mathematics Subject Classification. }} #1}

\newmdenv[
roundcorner=10pt,
backgroundcolor=gray!10,
linecolor=gray!10,
tikzsetting={draw=black,line width=3pt,dashed,dash pattern= on 10pt off 3pt},
outerlinewidth=0pt,
innerlinewidth=0pt,
]{figbox}

\newlength\figureheight
\newlength\figurewidth

\DeclareMathOperator*{\argmin}{\mathrm{argmin}}

\def\dd{\mathrm{d}}

\newcommand{\dist}{\mathrm{dist}}
\newcommand{\Div}{\mathrm{div}}

\newcommand{\Id}{\mathrm{Id}}
\newcommand{\iid}{\stackrel{\mathrm{iid}}{\sim}}

\newcommand{\Lip}{\mathrm{Lip}}

\newcommand{\Vol}{\mathrm{Vol}}

\def\eps{\varepsilon}

\DeclareMathOperator*{\Glim}{\Gamma\text{-}\lim}

\DeclareMathOperator*{\weakstarto}{\stackrel{*}{\rightharpoonup}}

\def\l{\left(}
\def\r{\right)}
\def\la{\left|}
\def\ra{\right|}
\def\lda{\left\|}
\def\rda{\right\|}
\def\lb{\left\{}
\def\rb{\right\}}
\def\rd{\right.}
\def\ls{\left[}
\def\rs{\right]}

\def\XXint#1#2#3{{\setbox0=\hbox{$#1{#2#3}{\int}$ }
\vcenter{\hbox{$#2#3$ }}\kern-.6\wd0}}

\newcommand{\bbE}{\mathbb{E}}

\newcommand{\bbN}{\mathbb{N}}

\newcommand{\bbP}{\mathbb{P}}

\newcommand{\bbR}{\mathbb{R}}

\newcommand{\cE}{\mathcal{E}}

\newcommand{\cL}{\mathcal{L}}

\newcommand{\cP}{\mathcal{P}}

\newcommand{\cX}{\mathcal{X}}

\newcommand{\numtype}{section}
\relax

\newtheorem{theorem}{Theorem}[\numtype]
\newtheorem{lemma}[theorem]{Lemma}
\newtheorem{proposition}[theorem]{Proposition}
\newtheorem{corollary}[theorem]{Corollary}
\newtheorem{mydef}[theorem]{Definition}
\theoremstyle{definition}
\newtheorem{remark}[theorem]{Remark}

\makeatletter
\newcommand{\leqnomode}{\tagsleft@true}
\newcommand{\reqnomode}{\tagsleft@false}

\newcounter{num}
\newcommand{\dWinfty}{\mathrm{d}_{\mathrm{W}^{\infty}}}

\newcommand{\ty}{\tilde{y}}

\newcommand{\cEn}{\cE_{n}} %!
\newcommand{\cEpn}{\cEn^{(p)}} %!
\newcommand{\cEncon}{\cE_{n,\mathrm{con}}} %!
\newcommand{\cEpncon}{\cEncon^{(p)}} %!
\newcommand{\cEinftycon}{\cE_{\infty,\mathrm{con}}} %!
\newcommand{\cEpinftycon}{\cEinftycon^{(p)}} %!

\newcommand{\cEpenneps}{\cE^{(\mathrm{pen})}_{n,\eps}}

\newcommand{\cEpneps}{\cE_{n,\eps}^{(p)}}
\newcommand{\cEpnepsn}{\cE_{n,\eps_n}^{(p)}}
\newcommand{\cEpinfty}{\cE^{(p)}_\infty}
\newcommand{\cEppenneps}{\cE^{(p,\mathrm{pen})}_{n,\eps}}
\newcommand{\cEppennepsn}{\cE^{(p,\mathrm{pen})}_{n,\eps_n}}
\newcommand{\cEpconneps}{\cE^{(p,\mathrm{con})}_{n,\eps}}
\newcommand{\cEpconnepsn}{\cE^{(p,\mathrm{con})}_{n,\eps_n}}

\def\cENLp#1{\cE^{(p,\mathrm{NL})}_{#1}}
\newcommand{\TL}{\mathrm{TL}}
\newcommand{\TLp}{\TL^p}
\newcommand{\dTLp}{d_{\TLp}}
\newcommand{\toTLp}{\stackrel{\TLp}{\to}}
\newcommand{\rmL}{\mathrm{L}}
\newcommand{\Lp}{\rmL^p}
\newcommand{\Lzero}{\rmL^0}
\newcommand{\Lone}{\rmL^1}
\newcommand{\Ltwo}{\rmL^2}
\newcommand{\Linfty}{\rmL^\infty}

\def\Ck#1{\mathrm{C}^{#1}}
\def\Wkp#1#2{\mathrm{W}^{#1,#2}}

\title{Rates of Convergence for Laplacian Semi-Supervised Learning with Low Labeling Rates}
\author{Jeff Calder\thanks{School of Mathematics, University of Minnesota (\url{jwcalder@umn.edu})}, Dejan Slep\v{c}ev\thanks{Department of Mathematical Sciences, Carnegie Mellon University (\url{slepcev@math.cmu.edu})}, and Matthew Thorpe\thanks{Department of Mathematics, University of Manchester, (\url{matthew.thorpe-2@manchester.ac.uk})}}
\date{June 2020}

\begin{document}

\maketitle

\begin{abstract}
We study graph-based Laplacian semi-supervised learning at low labeling rates. Laplacian learning uses harmonic extension on a graph to propagate labels. At very low label rates, Laplacian learning becomes degenerate and the solution is roughly constant with spikes at each labeled data point. Previous work has shown that this degeneracy occurs when the number of labeled data points is finite while the number of unlabeled data points tends to infinity. In this work we allow the number of labeled data points to grow to infinity with the number of labels. Our results show that for a random geometric graph with length scale $\eps>0$ and labeling rate $\beta>0$, if $\beta \ll\eps^2$ then the solution becomes degenerate and spikes form, and if $\beta\gg \eps^2$ then  Laplacian learning is well-posed and consistent with a continuum Laplace equation. Furthermore, in the well-posed setting we prove quantitative error estimates of $O(\eps\beta^{-1/2})$ for the difference between the solutions of the discrete problem and continuum PDE, up to logarithmic factors. We also study $p$-Laplacian regularization and show the same degeneracy result when $\beta \ll \eps^p$. The proofs of our well-posedness results use the random walk interpretation of Laplacian learning and PDE arguments, while the proofs of the ill-posedness results use $\Gamma$-convergence tools from the calculus of variations.  We also present numerical results on synthetic and real data to illustrate our results.
\end{abstract}

\noindent

\noindent
% 	49J55  	Problems involving randomness [See also 93E20]
% 	49J45  	Methods involving semicontinuity and convergence; relaxation
% 	62G20  	Nonparametric inference: Asymptotic properties
%  	35J20   Variational methods for second-order elliptic equations
%   60G50  	Probability theory and stochastic processes: Sums of independent random variables; random walks
%   60J20  	Probability theory and stochastic processes: Applications of Markov chains and discrete-time Markov processes on general state spaces (social mobility, learning theory, industrial processes, etc.) [See also 90B30, 91D10, 91D35, 91E40]

\section{Introduction} \label{sec:Intro}

Semi-supervised learning algorithms use both labeled and unlabeled data in learning tasks. In many situations
the labeling rate (i.e. the ratio between the number of labeled and the total number of data points) is low. The aim of
 semi-supervised learning is to utilize the geometric or topological properties of unlabeled data to improve the performance of classification or regression algorithms.
Given the ubiquity of data in the modern world, much of which is \emph{unlabeled}, and the relative expense of labeling data, the problem of semi-supervised learning has attracted significant interest. 

A typical semi-supervised learning problem is posed as follows: Given a data set $\Omega_n = \{x_i\}_{i=1}^n\subset \bbR^d$, and labels $\{y_i\}_{i\in Z_n}\subset \bbR$ for $\Gamma_n=\{x_i\}_{i\in Z_n}$ for a subset $Z_n\subset \{1,\dots, n\}$ of the dataset, learn a function $u:\Omega_n\to \bbR$ that extends the labels to the whole dataset $\Omega_n$ in some meaningful way.
There are many possible functions $u:\Omega_n\to \bbR$ that interpolate the labeled data $\{(x_i,y_i)\}_{i\in Z_n}$, the \emph{semi-supervised} smoothness assumption stipulates that the learned labels should vary smoothly in high density regions of the data \cite{ssl}.
Different mathematical interpretations of the semi-supervised smoothness assumption lead to different algorithms.
To encode the geometry of the data
distribution and measure the smoothness of the functions on the data, it is common  to construct a graph over the data $\Omega_n$. We consider weighted graphs with 
 edge-weight matrix $W=(w_{ij})_{i,j=1}^n$, where $w_{ij}\geq 0$ and the edge weights are symmetric: $w_{ij}=w_{ji}$, for all $i,j$.
If $w_{ij}=0$ then we say there is no edge between $i$ and $j$.
Two standard graph constructions are kNN graphs, where each data point to its $k$ nearest neighbors, or 
random geometric graphs where  all pairs of points that are within some specified distance $\eps>0$ are connected. 
Here we take the weighted variant of the  latter approach and assume $w_{ij} = \eta_\eps(|x_i-x_j|)$ for some kernel $\eta_\eps$.

This paper is concerned with Laplacian semi-supervised learning, originally proposed in \cite{zhu03}, which is a graph-based learning algorithm that propagates labels by solving the graph Laplace equation
\begin{equation}\label{eq:GLE}
\left\{\begin{aligned}
\L_{n,\eps} u(x_i) &=0, &&\text{if } x_i\not\in \Gamma_n\\ 
u(x_i) &=y_i,&&\text{if } x_i \in \Gamma_n,
\end{aligned}\right.
\end{equation}
where $\L_{n,\eps}$ is the graph Laplacian given by
\begin{equation}\label{eq:GL_intro}
\L_{n,\eps} u(x_i) = \sum_{j=1}^nw_{ij}(u(x_i) - u(x_j)).
\end{equation}
When the graph is connected, the solution of \eqref{eq:GLE} is unique.
Laplace learning has been widely used since its introduction in \cite{zhu03}, and several authors have considered other normalizations of the graph Laplacian~\cite{zhou2004semi,zhou2005learning,zhou2004learning} (the graph Laplacian defined in~\eqref{eq:GL_intro} is   called the unnormalized graph Laplacian), and soft label constraints~\cite{belkin2006manifold,belkin2004semi} (where labels are penalized rather than enforced).
Laplace learning can also be formulated in terms of random walks on graphs.
Let $X_0,X_1,X_2,\dots,$ be a random walk on $\Omega_n$ with transition probabilities $p_{ij} = w_{ij}/d_i$ of transitioning from $i$ to $j$, where $d_i=\sum_{j=1}^n w_{ij}$.
Then, the \emph{random walk} Laplacian $\L_r$, defined by  $\L_r u(x_i) = \frac{1}{d_i}\L_{n,\eps} u(x_i)$, is the generator for the random walk, and we have the representation formula
\begin{equation}\label{eq:RepForm}
u(x_i) = \E[y_{\tau} \, | \, X_0=x_i]
\end{equation}
for the solution $u$ of \eqref{eq:GLE}, where  $\tau\in Z_n$ is the index of the label point first reached by the walk. In this sense, we can think of Laplacian learning as computing $u(x_i)$ by taking a weighted average of labels that are \emph{close to} $x_i$ in the metric inherited from the random walk. Finally, Laplacian learning can also be viewed as an optimization, or variational, problem. The solution $u$ of \eqref{eq:GLE} is the unique minimizer of the graph Dirichlet energy
\begin{equation}\label{eq:GDir}
\cE^{(2)}_{n,\eps}(u) = \frac{1}{n^2\eps^2} \sum_{i,j=1}^n w_{ij}(u(x_i)-u(x_j))^2,
\end{equation}
subject to the constraints $u(x_i)=y_i$ for all $x_i \in \Gamma_n$. 

When the label rate $\beta :=\frac{1}{n}|Z_n|$ is very low, Laplacian learning becomes degenerate and the label function $u$ becomes nearly constant with sharp spikes at the labeled points~\cite{nadler09,elalaoui16}. See Figure \ref{fig:spikes} for a simulation showing the formation of spikes at low label rates. This degeneracy can be explained easily both from the variational and the random walk perspectives. Indeed, from the random walk perspective, if the label rate is very low then the random walk will explore the graph for too long before hitting a label, and the distribution of the walker will approach the invariant distribution \emph{before} hitting a label. The invariant distribution for the walker is $\pi(x_i) = \frac{d_i}{\sum_{j=1}^n d_j}$, and so by \eqref{eq:RepForm} we see that for very low label rates we have
\[ u(x_i) =\E[y_{\tau} \, | \, X_0=x_i]  \approx \frac{\sum_{j\in Z_n} d_jy_j }{\sum_{j\in Z_n} d_j}. \]
Thus, when the label rate is low, $u(x_i)$ is approximately the average of all labels weighted only by degree. Since the labels are imposed as hard constraints, the function $u$ must attain the labels, and it does so with spikes in a localized neighborhood of each label. This degeneracy leads to very poor results for classification at low label rates, since most datapoints get assigned to the same class \cite{flores2019algorithms,shi17}.

\begin{figure}[t!]
\centering
\includegraphics[trim=350 200 320 180, clip=true,width=0.23\textheight]{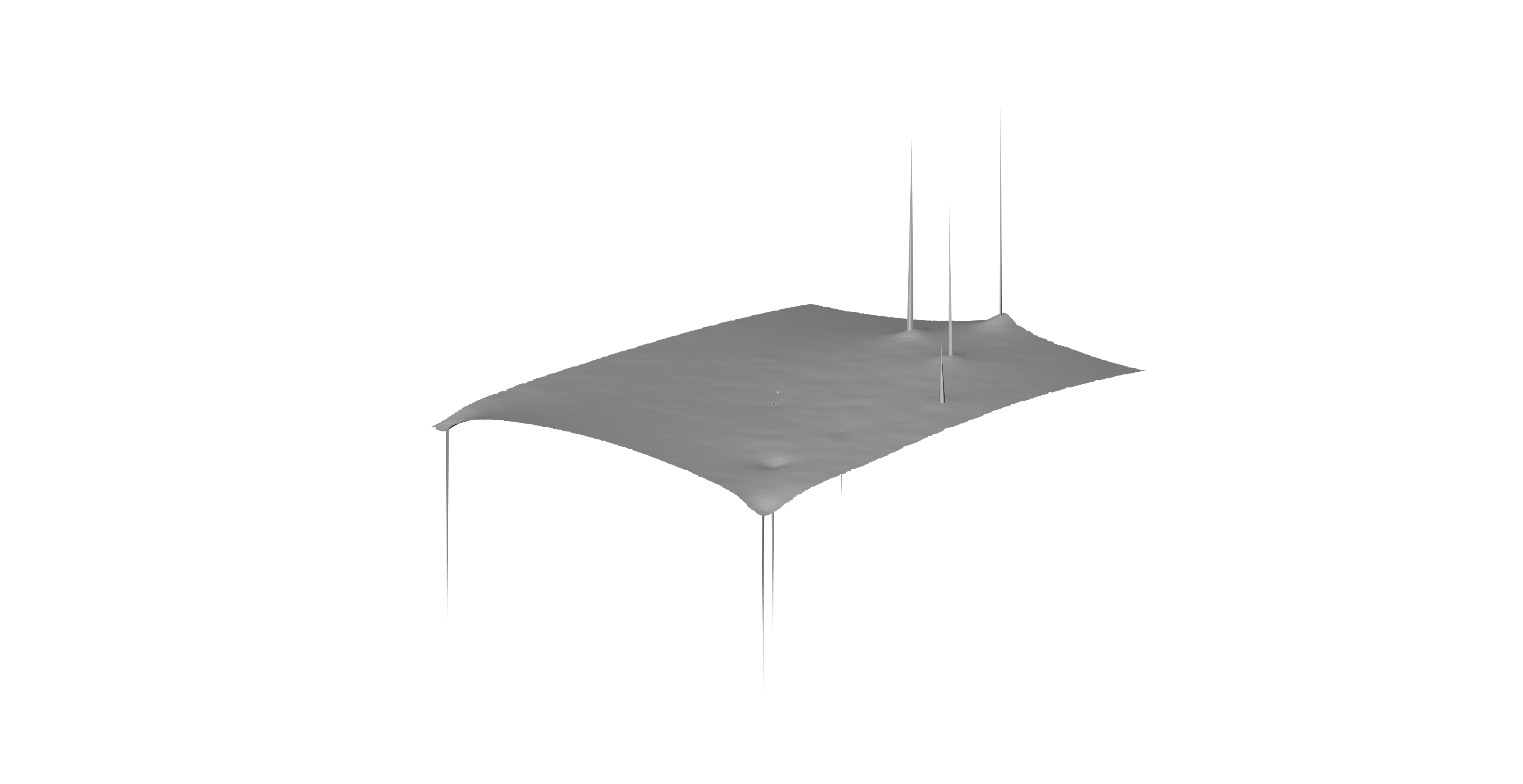}
\includegraphics[trim=350 130 320 120, clip=true,width=0.2\textheight]{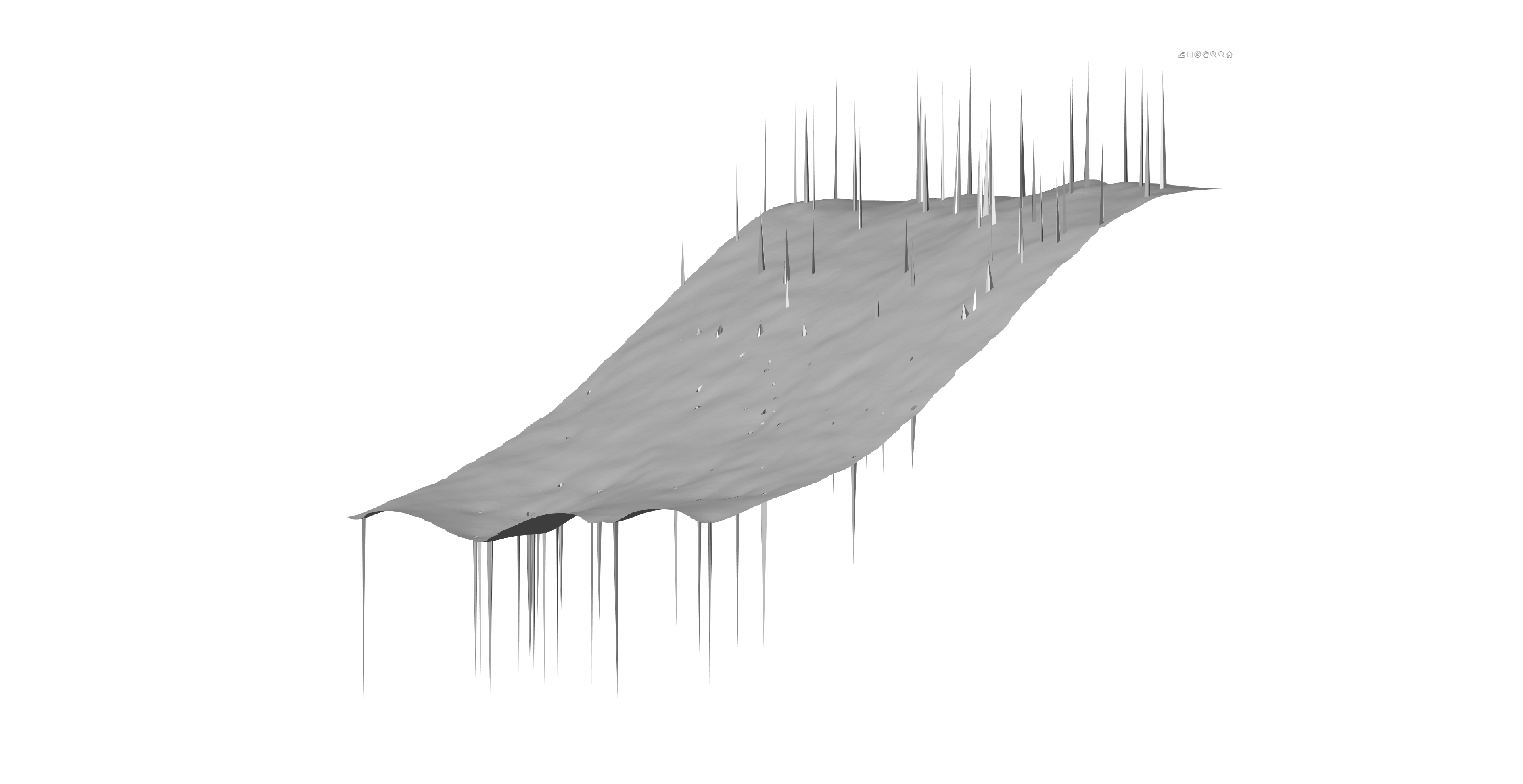}
\includegraphics[trim=350 100 320 70, clip=true,width=0.18\textheight]{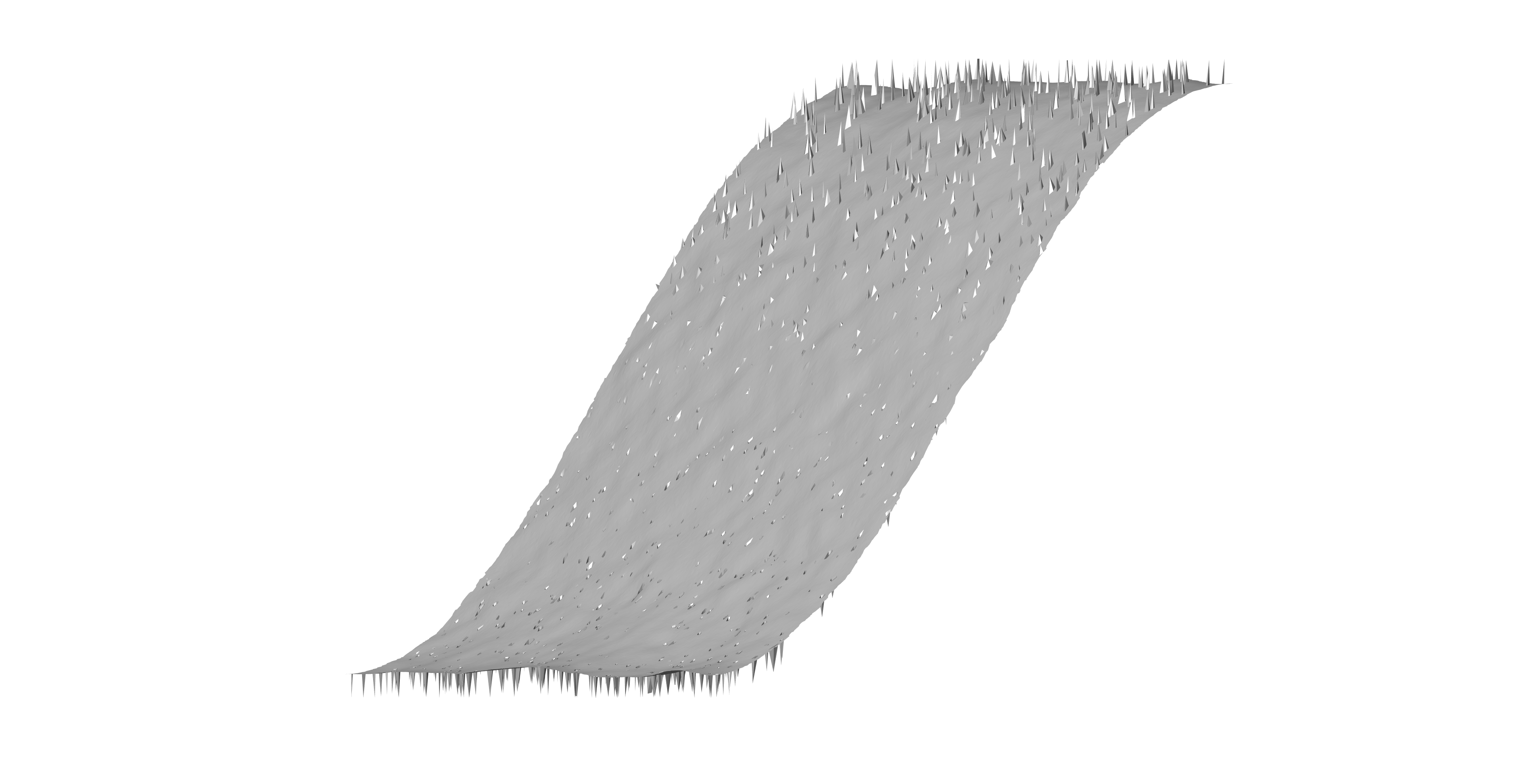}
\caption{Demonstration of spikes in Laplacian learning. The graph consists of $n=10^5$ independent uniform random variables on $[0,1]^2$ and from left to right we provide $10, 100$, and $1000$ labeled points according to the label function $\cos(x\cdot e_1)$ where $e_1=(1,0)$. Even at higher label rates where the solution is not constant, the spikes remain and the label values are not attained continuously.}
\label{fig:spikes}
\end{figure}

A significant amount of recent research has aimed to address this degeneracy by introducing new graph-based semi-supervised learning models at low label rates.
It was suggested in \cite{elalaoui16} to replace Laplacian learning with $p$-Laplacian regularization, which minimizes the $p$-Dirichlet energy
\[\cE^{(p)}_{n,\eps}(u) = \frac{1}{n^2\eps^p}\sum_{i,j=1}^n w_{ij}|u(x_i)-u(x_j)|^p.\]
If $p>d$ then the Sobolev embedding in the continuum suggests that all labels should be attained continuously and spikes cannot form. Using this intuition, \cite{elalaoui16} proposed to use $p$-Laplacian regularization with $p>d$ at very low label rates. This problem was studied rigorously in \cite{slepcev19} where it was shown that $p>d$ is actually insufficient for continuity at the labels if the length scale $\eps$ of the random geometric graph is too large (more precisely $n\eps^p \ll 1$).
However, when $\eps$ is small enough ($n\eps^p \ll 1$), then the continuum intuition is correct and one obtains a non-constant limit that this the solution to a well-posed $p$-Laplace equation. 
The authors of \cite{slepcev19} go on to show that the length scale restriction can be eliminated if the labels are first extended to nearest neighbors on the graph before solving the graph $p$-Laplace equation.
On a graph there are several different ways to formulate the $p$-Laplacian, and in \cite{calder18a,calder18AAA} the game-theoretic $p$-Laplacian was introduced for semi-supervised learning.
It was shown in \cite{calder18a} that solutions of the game-theoretic $p$-Laplacian on the graph are approximately H\"older continuous when $p>d$ and thus the label values are attained continuously, without any additional length scale restriction.
Algorithms for solving the variational and game-theoretic graph $p$-Laplace equations were developed in \cite{flores2019algorithms}.

Another way to address low label rates is to reweight the graph more heavily near labels to discourage spikes from forming. In~\cite{shi17} the authors propose a reweighting that amplifies the edge weights connected to labeled nodes by the ratio of unlabeled to labeled data. While the reweighting is useful in practice, it was shown in~\cite{calder18bAAA} that the method still degenerates to a constant at very low label rates. The authors of \cite{calder18bAAA} proposed a new way to reweight the graph that is sufficient to establish continuity at labels in the continuum limit by considering singularly weighted Sobolev spaces.

In this paper we address the question of precisely how many labels are required in Laplacian semi-supervised learning to avoid the formation of spikes and ensure the labels are propagated robustly. Previous work \cite{slepcev19} showed that if the number of labels is held fixed as the number of unlabeled data points tends to infinity, then spikes will form and the solution of Laplacian learning will cluster around constant solutions. Here, we allow the number of labeled data points to grow to infinity with the total size $n$ of the dataset, and we determine how the label rate $\beta$ should scale to ensure a well-posed continuum limit. Roughly speaking, our results show that for a random geometric graph with length scale $\eps>0$ (where $\eps\to0^+$ as $n\to\infty$), if $\beta \gg \eps^2$ then the solution of Laplacian learning is consistent in the continuum with Laplace's equation and the labels are attained continuously, and if $\beta \ll \eps^2$ then spikes form and the solution degenerates to a constant label function. In the well-posed regime $\beta \gg \eps^2$ we obtain a quantitative error estimate of $O(\eps\beta^{-1/2}\log(\eps^{-1}\beta^{1/2}))$ between the solution of the discrete Laplace learning problem and the continuum PDE. We also study $p$-Laplacian regularization and are able to show the ill-posedness result when $\beta \ll \eps^p$. The analogous well-posedness result for $p\neq 2$ is an interesting open problem for future research. We provide some numerical experiments at the end of the paper to illustrate the results on real and synthetic data.

The proofs in our paper use a blend of PDE, random walk, and variational techniques. The PDE techniques involve pointwise consistency of graph Laplacians and the maximum principle, which have been used for discrete to continuum convergence in graph-based settings recently \cite{calder18a,calder18AAA,yuan2020continuum,trillos2019maximum,shi18AAA}. The variational techniques use $\Gamma$-convergence and the $\TLp$ topology, which was originally developed in \cite{garciatrillos16} and has been used numerous times since for studying discrete to continuum convergence~\cite{slepcev19,garciatrillos18,garciatrillos18bAAA,garciatrillos16a,garciatrillos17,dunlop18AAA,garciatrillos17cAAA,garciatrillos18a,davis18,cristoferi18AAA,thorpe19,caroccia19,thorpe18,osting17} and for other applications, such as distances between signals for image processing~\cite{fitschen17,thorpe17a}. In all of our well-posedness results, the random walk arguments are essential for establishing convergence rates and correct scalings for $\beta$. At a high level, we use the random walk interpretation of Laplace learning \eqref{eq:RepForm} and show that when $\beta\gg \eps^2$ with high probability the random walk travels a distance no greater than $O(\eps\beta^{-1/2})$ before encountering a label, and thus $u$ is a local average of a Lipschitz label function. The random walk methods use martingale techniques to control stopping times and establish discrete to continuum convergence. In fact, one of our main results (Theorem \ref{thm:wellposed2}) is proved exclusively using random walk techniques, which to our knowledge have not been explored before in the discrete to continuum literature on graph-based learning. 

In the following section we introduce the notation, give the assumptions and state our results. We then prove the well-posedness results and ill-posedness results in Sections~\ref{sec:WellPosed} and~\ref{sec:IllPosed} respectively. Some numerical experiments are given in Section~\ref{sec:NumExp} and we conclude in Section~\ref{sec:Conc}. For convenience we include some background material in Appendix~\ref{sec:azuma} and~\ref{sec:app:TLpConv}.

\section{Setting and Main Results} \label{sec:MainRes}

We make the following assumption on the set of data points $\{x_i\}_{i=1}^n$.

\begin{enumerate}
\item[\bf (A1)] Let $d \geq 2$ and let $\Omega\subset \R^d$ be open, bounded, and connected with $\Ck{3}$ boundary. Let $x_1,x_2\dots,x_n$ be \emph{i.i.d.}~with density $\rho\in \Ck{2}(\bar{\Omega})$ satisfying $0<\rho_{\min}:=\inf_{x\in \Omega} \rho(x) \leq \sup_{x\in \Omega} \rho(x)=:\rho_{\max} < \infty$. We write $\Omega_n=\{x_1,x_2,\dots,x_n\}$.
\end{enumerate}
Let $\beta\in (0,1]$ and $\delta\in (0,\eps]$ and set
\begin{equation} \label{eq:MainRes:Boundary:NearBoundary}
\partial_\delta \Omega = \lb x\in \Omega \,:\, \dist(x,\partial \Omega) <\delta \rb.
\end{equation}
We consider two models for the distribution of training data.
\begin{description}
\item[Model 1.] Let $\wO$ be open with $\Ck{3}$ boundary and $\wO\subset\subset\Omega$. Each $x_i\in \wO$ is selected as training data independently with probability $\beta$.
\item[Model 2.] Each $x_i\in \partial_\delta \Omega$ is selected as training data independently with probability $\beta$.
\end{description}
In each case, we let $\Gamma_n\subset \Omega_n$ denote the collection of training data points and $Z_n \subset \{1,\dots,n\}$ the \emph{indices} of the training data points, so that
\[\Gamma_n = \{x_i \, : \, i \in Z_n\}.\]

We make the following assumption on the labels for training data.
\begin{enumerate}
\item[\bf (A2)] Let $g\in \Ck{3}(\bar{\Omega})$. Each training data point $x\in \Gamma_n$ is assigned the label $g(x)$.
\end{enumerate}

Given the data $\Omega_n$, an interaction potential $\eta:[0,\infty)\to[0,\infty)$, and a length scale $\eps>0$, we define a graph with nodes $\Omega_n$ and edge weights $\eta_{\eps}(|y-x|)$ between points $x,y\in \Omega_n$, where $\eta_\eps(t)=\frac{1}{\eps^d}\eta(\tfrac t \eps)$. 
In our results we allow for general choices of the interaction potential $\eta$, and only  assume that $\eta$ satisfies the following assumption.
\begin{enumerate}
\item[\bf (A3)] The interaction potential $\eta:[0,\infty)\to[0,\infty)$ is non-increasing, positive and continuous at $t=0$, $\eta(t) \geq 1$ for $t\leq 1$, and $\eta(t) = 0$ for $t\geq 2$.
We define $\eta_\eps = \frac{1}{\eps^d} \eta(\cdot/\eps)$ and
\[ \sigma_\eta:= \int_{\bbR^d} \eta(|x|) |x_1|^2 \, \dd x<\infty. \]
\end{enumerate}

The Laplacian semisupervised learning problem \cite{zhu03} is to minimize over all $u:\Omega_n\to \R$
\begin{align}\label{eq:L2}
\begin{split}
\cE_{n,\eps}(u) :=&\frac{1}{n^2\eps^2}\sum_{x,y\in \Omega_n} \eta_{\eps}(|x-y|)(u(x) - u(y))^2 \\
\hspace*{-30pt} \text{subject to the constraint } \quad  & u(x)=g(x) \text{ for  all }  x\in \Gamma_n.
\end{split}
\end{align}
Minimizers of the Laplacian learning problem \eqref{eq:L2} satisfy the boundary-value problem
\begin{equation}\label{eq:BVP}
\left\{\begin{aligned}
\cL_{n,\eps}u(x) &= 0,&&\text{if }x\in \Omega_n\setminus \Gamma_n\\ 
u(x)&=g(x),&&\text{if }x\in \Gamma_n,
\end{aligned}\right.
\end{equation}
where $\cL_{n,\eps}$ is the graph-Laplacian, which is given for $u:\Omega_n\to \R$ by
\begin{equation}\label{eq:GL}
\cL_{n,\eps}u(x)=\frac{1}{n\eps^2}\sum_{y\in \Omega_n}\eta_\eps(|x-y|)(u(x) -u(y)).
\end{equation}

The semi-supervised learning paradigm is that by using structure in the unlabeled data, we may improve the learning algorithm and reduce the amount of labeled data required, which reduces the cost of labeling data. The geometry and topology of the unlabeled data enters the process through the graph Laplacian $\cL_{n,\eps}$. Thus, the key question is how small we can take the labeling rate $\beta>0$ while ensuring the learning algorithm is stable and has a well-posed continuum limit. We organize our main results into two sections. Section \ref{sec:wellposed} presents our results on well-posedness, which roughly speaking show that when $\beta \gg \eps^2$ the learned function attains the labels continuously in the limit.
Then, in Section \ref{sec:illposed}, we present our ill-posedness results, that show that when $\beta \ll \eps^2$, the label information is lost in the limit.

\subsection{Well-posedness results}\label{sec:wellposed}

We now present our well-posedness results. We first consider {\bf Model 1}, where the continuum version of the boundary-value problem \eqref{eq:BVP} is
\begin{equation}\label{eq:model1pde}
\left\{\begin{aligned}
\div (\rho^2 \nabla u) &= 0&&\text{in }\Omega\setminus \wO\\ 
u &=g&&\text{in } \wO\\
\nabla u\cdot \bm{n}&=0&&\text{on }\partial \Omega.
\end{aligned}\right.
\end{equation}

We have the following well-posedness result.
\begin{theorem}[{\bf Model 1, Well-posed}]
\label{thm:wellposed1}
Assume {\bf (A1-3)} and {\bf Model 1}, let $u_n:\Omega_n\to \R$ be the solution of \eqref{eq:BVP}, and let $u\in \Ck{3}(\bar{\Omega})$ be the solution of \eqref{eq:model1pde}.  Then the following hold.
\begin{enumerate}[(i)]
\item There exists $C>c>0$ such that if $\beta \geq \eps^2$ and $\eps<\vartheta<c$ then 
\begin{equation}\label{eq:model1rate}
\max_{x\in \Omega_n} |u_n(x) - u(x)| \leq C \l\frac{\eps}{\sqrt{\beta}}\log\l\frac{\sqrt{\beta}}{\eps}\r + \vartheta\r
\end{equation}
holds with probability at least $1-Cn\exp\l-cn\eps^{d+2}\vartheta^2\r$.
\item If $\eps=\eps_n\to 0$ and $\beta=\beta_n\to 0$ as $n\to \infty$ such that $\beta_n\gg \eps_n^2$ and $n\beta_n \eps^d_n \gg \log(n)$ then
\begin{equation}\label{eq:model1limit}
\lim_{n\to \infty}\frac{1}{n}\sum_{x\in \Omega_n}|u_n(x)-u(x)|^2  = 0 \ \ \ \text{a.s.}
\end{equation}
\end{enumerate}
\end{theorem}

\begin{remark}
Theorem \ref{thm:wellposed1} says that the Laplacian learning with training data selected by {\bf Model 1} is well-posed when $\beta \gg \eps^2$ and $n\beta \eps^d \gg \log(n)$, which is equivalent to
\[ \sqrt{\beta} \gg \eps \gg \left( \frac{\log(n)}{n\beta} \right)^{\frac{1}{d}}.\]
By choosing $\vartheta = \frac{\eps}{\sqrt{\beta}}$ we obtain a convergence rate of $O(\frac{\eps}{\sqrt{\beta}})$ up to log factors for larger choices of $\eps$ satisfying $n\beta^{-1}\eps^{d+4}\gg \log(n)$, which is equivalent to
\[\sqrt{\beta} \gg \eps \gg \left( \frac{\beta\log(n)}{n} \right)^{\frac{1}{d+4}}.\]
We expect the well-posedness result to extend to the regime
\[\left( \frac{\log(n)}{n} \right)^{\frac{1}{d}}\ll \eps \ll \left( \frac{\log(n)}{n\beta} \right)^{\frac{1}{d}},\]
provided $\beta \gg \eps^2$, but are unable to prove this on random geometric graphs due to local irregularities in the graphs that lead to a large drift in a random walk. We discuss the smaller length scale regime further in Remark~\ref{rem:lattice}.
\end{remark}

\begin{remark}
For Theorem~\ref{thm:wellposed1}(ii) it is enough that $g$ is Lipschitz continuous and $\Omega$, $\wO$ have Lipschitz boundaries (i.e. we do not need $g\in \Ck{3}(\bar{\Omega})$ or $\Omega$, $\wO$ to have $\Ck{3}$ boundaries).
This is because the proof of Theorem~\ref{thm:wellposed1}(ii) does not use elliptic regularity results. 
\end{remark}

We now turn our attention to {\bf Model 2}. Here, the continuum version of the boundary-value problem \eqref{eq:BVP} is
\begin{equation}\label{eq:model2pde}
\left\{\begin{aligned}
\div (\rho^2 \nabla u) &= 0&&\text{in }\Omega\\ 
u &=g&&\text{on }\partial \Omega.
\end{aligned}\right.
\end{equation}

We have the following well-posedness result, analagous to Theorem \ref{thm:wellposed1}.
\begin{theorem}[{\bf Model 2, Well-posed}]
\label{thm:wellposed2}
Assume {\bf (A1-3)} and {\bf Model 2}, let $u_n:\Omega_n\to \R$ be the solution of \eqref{eq:BVP}, and let $u\in \Ck{3}(\bar{\Omega})$ be the solution of \eqref{eq:model2pde}. There exists $C>c>0$ such that if $\delta\leq \eps$,  $\beta\delta\geq \eps\vartheta$ and $\eps \leq \vt \leq c$ then we have that
\begin{equation}\label{eq:model2rate}
\max_{x\in \Omega_n} |u_n(x) - u(x)| \leq C\frac{\vt \eps}{\beta\delta}
\end{equation}
holds with probability at least $1 - Cn\exp\left( -cn\eps^{d+2} \vt^2 \right)$.
\end{theorem}
\begin{remark}
Theorem \ref{thm:wellposed2} says that the Laplacian semisupervised learning with {\bf Model 2} is well-posed when $\beta\delta \gg \eps^2$, $n\eps^{d+2}\gg \log(n)$, and $n\eps^{d-1}\beta\delta \gg \log(n)$, which is equivalent to the length scale restriction
\[\sqrt{\beta\delta} \gg \eps \gg \left( \frac{\log(n)}{n} \right)^{\frac{1}{d+2}}.\]
In the case that $n\eps^{d+4}\gg \log(n)$, we have a convergence rate of $O(\frac{\eps^2}{\beta\delta})$.  As with {\bf Model 1}, we expect Theorem \ref{thm:wellposed2} to extend to length scales $\eps$ in the regime
\[\left( \frac{\log(n)}{n} \right)^{\frac{1}{d}}\ll \eps \ll \left( \frac{\log(n)}{n} \right)^{\frac{1}{d+2}},\]
provided $\beta\delta \gg \eps^2$. We defer discussion of this to Remark \ref{rem:lattice}.
\end{remark}

\begin{remark}
\label{rem:MainRes:wellposed:Soft}
Similar results to Theorems \ref{thm:wellposed1} and \ref{thm:wellposed2} can be obtained for the soft constrained version of the Laplacian learning problem \eqref{eq:L2}:
\begin{equation}\label{eq:softL2}
\min_{w:\Omega_n\to \R} \cEpenneps(w):= \frac{1}{n^2\eps^2}\sum_{x,y\in \Omega_n} \eta_{\eps}(|x-y|)(w(x) - w(y))^2 + \frac{\lambda}{|\Gamma_n|}\sum_{y\in \Gamma_n} |w(y) - g(y)|^2.
\end{equation}
We sketch the proof here.  Let $w_n$ be the solution of \eqref{eq:softL2} and note that
\[\max_{y\in \Gamma_n}|w_n(y)-g(y)|^2 \leq \sum_{y\in \Gamma_n} |w_n(y)-g(y)|^2 \leq \cEpenneps(w_n)\frac{|\Gamma_n|}{\lambda}.\]
The maximum principle yields
\[\max_{y\in \Omega_n}|w_n(y) - u_n(y)|^2 \leq \max_{y\in \Gamma_n}|w_n(y) - u_n(y)|^2 \leq \cEpenneps(w_n)\frac{|\Gamma_n|}{\lambda} \leq \frac{Cn\beta}{\lambda},\]
under the assumptions of {\bf Model 1} with probability at least $1-e^{-cn\beta}$, where $u_n$ is the solution of the hard constrained problem \eqref{eq:BVP}. Thus, under the same assumptions as Theorem \ref{thm:wellposed1}(i) we have
\begin{equation}\label{eq:model1ratesoft}
\max_{x\in \Omega_n} |w_n(x) - u(x)| \leq C\left(\frac{\eps}{\sqrt{\beta}}\log\l\frac{\sqrt{\beta}}{\eps}\r +\vartheta + \frac{n\beta}{\lambda}\right)
\end{equation}
with probability at least $1 - Cn\exp\left( -cn\eps^{d+2}\vartheta^2\right)$, where $u$ is the solution of \eqref{eq:model1pde}.
Under the assumptions of Theorem \ref{thm:wellposed1}(ii), with the additional assumption that $\lambda=\lambda_n\to \infty$ as $n\to \infty$ so that $\lambda_n \gg n\beta_n$, we have
\begin{equation}\label{eq:model1limitsoft}
\lim_{n\to \infty}\frac{1}{n}\sum_{x\in \Omega_n}|w_n(x)-u(x)|^2  = 0 \ \ \ \text{a.s.}
\end{equation}
A similar argument holds for {\bf Model 2} under the assumptions of Theorem~\ref{thm:wellposed2}.

We remark, however, that soft constraints are normally used for problems with noisy labels, or regression problems. Our model assumes clean labels with a high degree of confidence, and here it is natural to use hard constraints. It would be interesting to extend our results to a noisy label model of the form $y = g(x) + \eta$ where $\eta$ is noise. In this case, we expect the argument above to be suboptimal and a more nuanced approach is necessary.
\end{remark}

\begin{remark}
\label{rem:MainRes:wellposed:GTLap}
The proofs of Theorems \ref{thm:wellposed1} and \ref{thm:wellposed2} use the random walk interpretation of Laplace learning~\eqref{eq:BVP} given in \eqref{eq:RepForm}. Since the variational graph $p$-Laplacian does not have a random walk interpretation, our arguments to not directly extend to this case.  However, there are other variants of the graph Laplacian that have natural random walk interpretations. For example, the game-theoretic $p$-Laplacian \cite{calder18a} has an interpretation in terms of two-player stochastoc tug-of-war games, which are the natural extensions of the random walk interpretation of Laplace's equation to the $p$-Laplacian. The weighted graph Laplacian~\cite{shi17} also has a random walk interpretation, since it is a $p=2$ graph Laplacian on a reweighted graph. We expect it to be possible to adapt the arguments used in this paper to these settings in order to obtain similar results. We note that there has already been some analysis of convergence rates for the weighted graph Laplacian in~\cite{shi18AAA}.
\end{remark}

\begin{remark}\label{rem:lattice}
Theorems \ref{thm:wellposed1} and \ref{thm:wellposed2} require lower bounds on the connectivity length scale $\eps>0$ of the graph of $n\eps^{d+2}\gg \log(n)$, except for Theorem \ref{thm:wellposed1} (ii) which requires the weaker condition $n\eps^d\beta \gg \log(n)$. The lower bounds are required to control the randomness in the graph with concentration inequalities. It is well-known that the graph is connected with high probability when $n\eps^d \gg \log(n)$. It is thus natural to ask whether Theorems \ref{thm:wellposed1} and \ref{thm:wellposed2}, in particular their convergence rates, hold in the intermediate length scale regime
\begin{equation}\label{eq:lengthscale}
 \left( \frac{\log(n)}{n} \right)^{1/d }\ll \eps \ll \left( \frac{\log(n)}{n} \right)^{1/(d+2)}.
\end{equation}
In this length scale regime, we lose control of all estimates on the random walk due to a large random drift arising from local irregularities in the random geometric graph. In particular, we also lose pointwise consistency of graph Laplacians in this regime. To extend our results to this setting, one would need to show that the random walk becomes far more regular over many steps, and establish some form of a central limit theorm for the walk on a random geometric graph. To our knoweldge, such results are not known in the length scale regime \eqref{eq:lengthscale}.

Nevertheless, based on energy scaling arguments we expect Theorems \ref{thm:wellposed1} and \ref{thm:wellposed2} to extend to the length scale regime \eqref{eq:lengthscale} provided $\beta \gg \eps^2$ and $\beta\delta\gg \eps^2$, respectively, although the exact form of the rates may change. To see how we may expect the random walk arguments to extend to this length scale regime, we momentarily consider the case of a lattice, where the graph is highly regular and the random walk is well-understood.  Let $\Z_\eps = \eps\Z= \{\eps z\, : \, z \in\Z\}$ and consider the lattice $\Z_\eps^d  = (\Z_\eps)^d$. Define the graph Laplacian
\begin{equation}\label{eq:latticeGL}
\Delta_\eps u(x) = \sum_{i=1}^d\sum_{b=\pm 1}(u(x+b \eps e_i) - u(x))
\end{equation}
where $u:\Z_\eps^d \to \R$.  For an integer $m\geq 1$ and a function $g:\Z^d_\eps\to \R$, consider the Laplace equation
\begin{equation}\label{eq:latticeLE}
\left\{\begin{aligned}
\Delta_\eps u(x) &= 0,&&\text{if }x \in \Z^d_\eps\setminus \Z^d_{m\eps}\\ 
u(x) &=g(x),&&\text{if }x\in \Z^d_{m\eps}.
\end{aligned}\right.
\end{equation}
The fraction of lattice points with Dirichlet conditions, or the label rate, is $\beta := m^{-d}$. Let $u_\eps$ be the solution of \eqref{eq:latticeLE}. We claim that
\begin{equation}\label{eq:latticerate}
\sup_{x\in \Z^d_\eps}|u_\eps(x) - g(x)| \leq \frac{C\eps}{\beta^{1/d + 1/2}},
\end{equation}
provided $g:\R^d\to \R$ is Lipschitz and bounded. Note that the rate in \eqref{eq:latticerate} requires $\beta \gg \eps^{2d/(d+2)}$  for the error to converge to zero as $\eps \to 0$, which is worse than the expected rate of $\beta \gg \eps^2$. We believe that this is due to our method of proof using random walks, and that some more precise tools are needed to establish the tight $\beta\gg \eps^2$ rate. We note that  in high dimensions $2d/(d+2)\approx 2$ and the rate is very close to optimal.

We briefly sketch a random walk proof of \eqref{eq:latticerate}. The ideas in the proof are similar to ones used in proving Theorems \ref{thm:wellposed1} and \ref{thm:wellposed2}. Let $X_0,X_1,X_2,\dots$ be a lazy simple random walk on $\Z^d_\eps$. The lazy simple random walk has transition probabilities 
\[\P(X_{k+1} = y \, | \, X_k=x) = \frac{1}{2d+1}\one_{|x-y| \leq \eps}\]
for $x,y\in \Z^d_\eps$.
The walk is \emph{lazy} because it has a positive probability of remaining at the current vertex at each step, and this makes the walk \emph{aperiodic}. We note that $\frac{1}{2d+1}\Delta_\eps$ is the generator for the lazy simple random walk, which is the key property that we use in the proof below.

Define the $k$-step probability transition function for the walk $p_k:\Z^d_\eps\times \Z^d_\eps\to [0,1]$ as
\begin{equation}\label{eq:heatkernel}
p_k(x,y) = \P(X_k=y \, | \, X_0=x).
\end{equation}
By the Local Central Limit Theorem for random walks (see \cite[Theorem 2.1.1]{lawler2010random}), there exists $N,C>0$, depending only on $d$, such that for all $x,y\in \Z^d_\eps$ we have
\begin{equation}\label{eq:heatkernelGauss}
p_k(x,y) \geq \frac{C}{k^{d/2}} \ \ \ \text{ provided }k \geq N \text{ and }|x-y|\leq \eps\sqrt{k}.
\end{equation}
Of course the Local Central Limit Theorem is stronger, and shows that $p_k(x,y)$ is approximately Gaussian, however, we state \eqref{eq:heatkernelGauss} to emphasize that the random walk methods only require lower bounds of the right order of magnitude. Provided $k \geq Cm^2$ we have
\begin{equation}\label{eq:stoppingtimebound}
\P(X_k\in \Z^d_{m\eps} \, | \, X_0=x)=\sum_{y\in \Z^d_{m\eps}}p_k(x,y)\geq\sum_{\substack{y\in \Z^d_{m\eps}\\ |x-y|\leq \eps\sqrt{k} }}\frac{C}{k^{d/2}}\geq c \left( \frac{\sqrt{k}}{m} \right)^d \frac{C}{k^{d/2}} = \frac{c}{m^d},
\end{equation}
where $C,c>0$ depend only on dimension $d$ and can change from line to line. Assume the random walk starts at $X_0=x\in \Z^d_\eps$ and define the stopping time
\[\tau = \inf\{k>0 \, |\, X_\tau \in \Z^d_{m\eps}\}.\]
 By \eqref{eq:stoppingtimebound} we obtain $\E[\tau]\leq Cm^{2+d}$. By a standard Martingale argument $Z_k := |X_k-x|^2 - Ck\eps^2$ is a supermartingale for large enough $C>0$, and so we have $\E[Z_\tau]\leq \E[Z_0]=0$, which yields
\[\E[|X_\tau - x|^2] \leq C\E[\tau]\eps^2 \leq Cm^{2+d}\eps^2.\]
Since $\Delta_\eps u(X_k)=0$ for $k<\tau$, $u(X_k)$ is a martingale up to the stopping time $\tau$. Thus, we may apply Doob's Optional Stopping Theorem to obtain $u_\eps(x) = \E[g(X_\tau)]$. Combining these observations yields
\[ \la u_\eps(x) - g(x)\ra =\la\E\ls u_\eps(X_\tau) - g(x)\rs\ra= \la\E\ls g(X_\tau) - g(x)\rs\ra\leq C\E\ls\la X_\tau - x\ra\rs\leq Cm^{1+d/2}\eps,\]
which establishes \eqref{eq:latticerate}.
\end{remark}

\subsection{Ill-posedness results}
\label{sec:illposed}

When $\eps_n$ satisfies $n\beta_n\eps_n^{d}\gg \log(n)$ then Theorem~\ref{thm:wellposed1} implies the problem is asymptotically well-posed if $\beta_n\gg\eps_n^2$.
We can show that this bound is tight in certain regimes; in particular we can show that if $\beta_n\ll\eps_n^2$ (and $\eps_n$ satisfies a lower bound) then constraints are forgotten in the limit.
This is the ill-posed regime.

The results in the previous section were proved using the random walk interpretation of solutions to~\eqref{eq:L2}.
This relationship is special to the 2-Dirichlet energy.
The results of this section are proved using variational methods which in particular allow one to treat the $p$-Laplacian semisupervised learning problem introduced in~\cite{zhou05}.
That is, for $p> 1$, we minimize over all $u:\Omega_n\to \bbR$ 
\begin{align}\label{eq:MainRes:illposed:Lp}
\begin{split}
\cE_{n,\eps}^{(p)}(u) :=&\frac{1}{n^2\eps^p}\sum_{x,y\in \Omega_n} \eta_{\eps}(|x-y|)|u(x) - u(y)|^p \\
\hspace*{-30pt} \text{subject to the constraint } \quad  & u(x)=g(x) \text{ for  all }  x\in \Gamma_n.
\end{split}
\end{align}
We have the following result corresponding to the ill-posed version of Theorem~\ref{thm:wellposed1}.

\begin{theorem}[{\bf Model 1, Ill-posed}]
\label{thm:MainRes:illposed:illposed1}
Assume $p>1$, {\bf (A1-3)} and {\bf Model 1}. Let $u_n:\Omega_n\to \bbR$ be the solution of~\eqref{eq:MainRes:illposed:Lp}.
Further, we assume that $\beta=\beta_n\to 0^+$ and $\eps=\eps_n\to 0^+$ satisfy
\begin{equation} \label{eq:MainRes:illposed:scaling1}
\beta_n\ll \eps_n^p, \quad n\eps_n^p\gg\log(n) \quad \text{and} \quad n \eps_n^d \gg \log(n).
\end{equation}
Then, with probability one, the set $\{u_n\}_{n\in\bbN}$ is pre-compact and any convergent subsequence converges to a constant.
\end{theorem}

\begin{remark}
Precompactness above is with respect to the $\TLp$ topology (see Appendix~\ref{sec:app:TLpConv}) which is a topology that allows us to define a discrete-to-continuum notion of convergence.
The convergence to a constant is both with respect to $\TLp$ and in $\Lp(\mu_n)$, where $\mu_n=\frac{1}{n}\sum_{x\in\Omega_n} \delta_x$ is the empirical measure and we say $u_n:\Omega_n\to \bbR$ converges to $u\in \Ck{0}(\Omega)$ in $\Lp(\mu_n)$ if $\frac{1}{n} \sum_{x\in \Omega_n} |u_n(x)-u(x)|^p \to 0$ (the $\Lp(\mu_n)$ convergence can only be defined when $u$ is continuous, the $\TLp$ convergence is defined for all $u\in \Lp(\mu)$). 
\end{remark}

In the assumptions we have an upper bound on $\beta_n$, which in light of the well-posedness result is natural, and lower bounds on $\eps_n$.
Let us remark on the latter conditions.

\begin{remark}
We make the following remarks on the scaling of $\eps_n$.
\begin{enumerate}
\item The last scaling assumption in~\eqref{eq:MainRes:illposed:scaling1} implies that $\eps_n$ is much greater than the connectivity radius which scales as $\l\frac{\log(n)}{n}\r^{\frac1d}$ (for all $d\geq 1$)~\cite{penrose03}.
The bound enters the analysis through the $\infty$-Wasserstein distance between the empirical measure $\mu_n$ and the data generating measure $\mu$ (where $x_i\iid\mu$), in particular, $\eps_n\gg\dWinfty(\mu_n,\mu)$.
Although there is gap between the scaling of the $\infty$-Wasserstein distance and the graph connectivity when $d=2$ in this setting we can close the gap and treat $\dWinfty(\mu_n,\mu)\geq \eps_n\gg \sqrt{\frac{\log n}{n}}$ by introducing an intermediate measure between $\mu_n$ and $\mu$.
\item The requirement that $\frac{n\eps_n^p}{\log n} \gg 1$ is upto a logarithmic factor the same upper bound on $\eps_n$ that was needed in~\cite{slepcev19} for the ill-posedness result with finite number of labels.
In particular,~\cite[Theorem 2.1(i)]{slepcev19} showed minimizers of $\cEpn$ subject to finite constraints (i.e. $|Z_n|=N$) are well-posed when $n\eps_n^p\ll 1$.
Hence, if $n\eps_n^p\ll 1$ we expect to achieve a well-posed limit regardless of the rate at which $|Z_n|\to+\infty$.
We note that if $\frac{n\eps_n^d}{\log n}\gg 1$ then $\frac{n\eps_n^p}{\log n} \gg 1$ for all $p\leq d$, hence the condition $\frac{n\eps_n^p}{\log n} \gg 1$ only becomes relevant when $p>d$.
\end{enumerate}
\end{remark}

For {\bf Model 2} we have the analogous result.

\begin{theorem}[{\bf Model 2, Ill-posed}]
\label{thm:MainRes:illposed:illposed2}
Assume $p>1$, {\bf (A1-3)} and consider {\bf Model 2}. Let $u_n:\Omega_n\to \bbR$ be the solution of~\eqref{eq:MainRes:illposed:Lp}.
Further, we assume that $\beta=\beta_n\to 0^+$, $\eps=\eps_n\to 0^+$ and $\delta=\delta_n\to 0^+$ satisfy
\[ \beta_n\delta_n\ll \eps_n^p, \quad n\eps_n^p\gg\log(n) \quad \text{and} \quad n \eps_n^d \gg \log(n) \]
Then, with probability one, the set $\{u_n\}_{n\in\bbN}$ is pre-compact and any convergent subsequence converges to a constant.
\end{theorem}

\begin{remark}
Theorems~\ref{thm:MainRes:illposed:illposed1} and~\ref{thm:MainRes:illposed:illposed2} remains true if $g$ is continuous and $\Omega$ has a Lipschitz boundary (i.e. we do not need $g\in \Ck{3}(\bar{\Omega})$ or $\Omega$, $\wO$ to have $\Ck{3}$ boundaries).
\end{remark}

The proof of both theorems are given in Section~\ref{sec:IllPosed}, here we quickly sketch out the strategy.
We define
\[ \cEpconneps(w) := \lb \begin{array}{ll} \frac{1}{n^2\eps^2}\sum_{x,y\in \Omega_n} \eta_{\eps}(|x-y|)|w(x) - w(y)|^p & \text{if } w(x) = g(x) \text{ for all } x\in \Gamma_n \\ +\infty & \text{otherwise} \end{array} \rd \]
and
\begin{equation} \label{eq:MainRes:illposed:cEpinfty}
\cEpinfty(w) = \sigma_\eta \int_\Omega |\nabla w(x)|^p \rho^2(x) \, \dd x.
\end{equation}
 We show that $\Glim_{n\to\infty} \cEpconnepsn = \cEpinfty$ in $TL^p$ topology, that is that the graph functional with constraints $\Gamma$ converges to the unconstrained continuum functional. 
This is the mathematical representation of the fact that labels are ``forgotten'' as $n \to \infty$. 
We then apply the fundamental theorem of $\Gamma$-convergence (see Theorem~\ref{thm:Prelim:TLp:Conmin}) which states that if minimizers $u_n$ of $\cEpconnepsn$ are precompact and $\Glim_{n\to\infty} \cEpconnepsn = \cEpinfty$ then any cluster point of $\{u_n\}_{n\in\bbN}$ is a minimizer of $\cEpinfty$.
Since $\cEpinfty$ is minimized by constant functions then we are done. 
The topology here is the $\TLp$ topology (see Section~\ref{sec:app:TLpConv}) and the $\Gamma$-convergence and compactness property are consequences of the main result in~\cite{garciatrillos16} where the unconstrained energies were considered with $p=1$ (restated for convenience in Proposition~\ref{prop:Prelim:TLp:GamConvEn}).

\begin{remark}
As in Remark~\ref{rem:MainRes:wellposed:Soft} one can adapt the proof to the soft constraint problem.
We define 
\[ \min_{w:\Omega_n\to \R} \cEppenneps(w):= \frac{1}{n^2\eps^p}\sum_{x,y\in \Omega_n} \eta_{\eps}(|x-y|)|w(x) - w(y)|^p + \frac{\lambda}{|\Gamma_n|}\sum_{y\in \Gamma_n} |w(y) - g(y)|^2. \]
In order to follow the proof strategy of Theorems~\ref{thm:MainRes:illposed:illposed1} and~\ref{thm:MainRes:illposed:illposed2} we are required to show that the compactness property holds and $\Glim_{n\to\infty}\cEppennepsn = \cEpinfty$.
Using the $\Gamma$-convergence and compactness property result for the constrained functionals $\cEpconnepsn$ and the unconstrained functionals $\cEpneps$, together with the bound $\cEpneps(w)\leq\cEppenneps(w)\leq \cEpconneps(w)$, implies the result.
\end{remark}

\section{The Well-Posed Case} \label{sec:WellPosed}

We give here the proofs of the well-posedness results, Theorems \ref{thm:wellposed1} and \ref{thm:wellposed2}. In Section \ref{subsec:WellPosed:Pointwise} we prove pointwise consistency estimates for graph Laplacians. Section \ref{sec:graph} proves basic estimates on random geometric graphs.
In Section \ref{subsec:WellPosed:RWBounds} we use random walks on random geometric graphs and martingale methods to prove boundary estimates for the Laplacian learning problems, showing the boundary conditions are attained continuously. Finally, in  Section~\ref{sec:wellposedproofs} we use the maximum principle to extend the boundary estimates to the entire domain, completing the proof of our main results. 

In this section, $C$ and $c$ denote constants that can change from line to line, and depend on $d$, $\rho$, and $\Omega$. We always take $C\geq 1$ to be large constant and $0 < c <1$ to represent small constants.

\subsection{Pointwise Convergence of the Graph Laplacian} \label{subsec:WellPosed:Pointwise}

In this section we derive the pointwise consistency of the graph Laplacian, $\L_{n,\eps}$ defined by~\eqref{eq:GL}, to the continuum Laplacian $\cL$, defined below. 
While this problem has been studied before~\cite{hein05,singer06,calder18a}, the novelty here is that we need and derive precise error estimates  near the boundary. Namely when the boundary is within the bandwidth of the kernel used to define graph Laplacian, we identify the correction terms needed. 
We prove the result by using a non-local continuum intermediary functional $\cL_{\eps}$. We define
\begin{align}
\cL_{\eps}u(x) & = \frac{2}{\eps^2} \int_{\Omega} \eta_{\eps}(|x-y|) \l u(x)-u(y)\r \rho(y) \, \dd y \label{eq:WellPosed:Pointwise:NLL} \\
\cL u(x) & = -\frac{\sigma_\eta}{\rho(x)} \Div \l \rho^2 \nabla u\r(x) \label{eq:WellPosed:Pointwise:ContL}
\end{align}
where $\sigma_\eta$ is the constant defined in Assumption {\bf (A3)}. We start by recalling the convergence of the graph Laplacian to the non-local Laplacian. 
\begin{theorem}
\label{thm:WellPosed:Pointwise:GraphtoNLLap}
\cite[Theorem 5]{calder18a}
Assume {\bf (A1-3)}. Then, there exists $C>c>0$ such that for any $\eps\leq \vartheta\leq \frac{1}{\eps}$ we have
\[ \bbP\l \lda \L_{n,\eps}\varphi - \cL_\eps\varphi\lfloor_{\Omega_n}\rda_{\Linfty(\mu_n)} \leq C\vartheta\|\varphi\|_{\Ck{3}(\overline{\Omega})},\, \forall \varphi\in \Ck{3}(\overline{\Omega})\r \geq 1 - C ne^{-cn\eps^{d+2}\vartheta^2}. \]
\end{theorem}

We let $\bm{n}(x)$ be the outward unit normal vector to $\Omega$ at $x\in\partial\Omega$.  Since $\partial \Omega$ is $\Ck{2}$, there exists a $\Ck{2}$ extension of $\bm{n}$ to $\Omega$ such that for some $r>0$,  $\bm{n}(x) = \bm{n}(y_x)$ whenever $\delta_x:=\dist(x,\partial \Omega)\leq r$, where $y_x=\argmin_{y\in \partial \Omega}|x-y|$.
For convenience we write $\frac{\partial}{\partial \bm{n}}:= \bm{n} \cdot \nabla$ for the partial derivative in the direction $\bm{n}$.
We define
\begin{equation} \label{eq:WellPosed:Pointwise:gamma}
\gamma_\eps(x) = \frac{1}{\eps} \int_{B(x,2\eps)\cap\Omega} \eta_\eps(|x-y|) (x-y)\cdot \bm{n}(x) \, \dd y.
\end{equation}
Note that $\gamma_\eps(x)=0$ for $\delta_x\geq 2\eps$. Define $\sigma_1$ and $\sigma_2$ by, for $t\geq 0$,
\begin{align}
\sigma_1(t) & = \int_{B(0,2)\cap\{z_d>-t\}} \eta(|z|) z_1^2 \, \dd z \label{eq:WellPosed:Pointwise:sigma1} \\
\sigma_2(t) & = \int_{B(0,2)\cap\{z_d>-t\}} \eta(|z|) (z_d^2 - z_1^2) \, \dd z. \label{eq:WellPosed:Pointwise:sigma2}
\end{align}
Notice that $\sigma_i$ are Lipschitz continuous, and $\sigma_1(t) =\sigma_\eta$, $\sigma_2(t) = 0$ for all $t\geq 2$.

\begin{theorem}
\label{thm:WellPosed:Pointwise:NLLaptoLLap}
Assume {\bf (A1,3)}. Then there exists $C>0$ such that for any $\varphi\in \Ck{3}(\overline{\Omega})$
\begin{align*}
& \sup_{x\in \Omega} \la \cL_\eps \varphi(x) - \frac{\sigma_1\l\frac{\delta_x}{\eps}\r}{\sigma_\eta} \cL \varphi(x) - \frac{2\rho(x)\gamma_\eps(x)}{\eps} \frac{\partial \varphi}{\partial \bm{n}}(x) + \frac{\sigma_2\l\frac{\delta_x}{\eps}\r}{\rho(x)} \frac{\partial}{\partial \bm{n}} \l\rho^2 \frac{\partial \varphi}{\partial \bm{n}}\r(x) \ra  \leq C \eps \|\varphi\|_{\Ck{3}(\overline{\Omega})}.
\end{align*}
\end{theorem}

\begin{remark}
\label{rem:WellPosed:Pointwise:AwayFromBoundary}
Since $\sigma_1(t) =\sigma_\eta$, $\sigma_2(t) = 0$ for all $t\geq 2$ then for $\delta_x\geq 2\eps$ we have $\gamma_\eps(x) = 0$, $\sigma_1(\delta_x/\eps) = \sigma_\eta$ and $\sigma_2(\delta_x/\eps)=0$ hence the theorem implies
\[ \la \cL_\eps \varphi(x) - \cL \varphi(x) \ra \leq C \eps \|\varphi\|_{\Ck{3}(\overline{\Omega})}. \]
which coincides with pointwise consistency results away from the boundary~\cite{hein05,singer06,calder18a}.
\end{remark}

\begin{proof}[Proof of Theorem~\ref{thm:WellPosed:Pointwise:NLLaptoLLap}.]
By Remark~\ref{rem:WellPosed:Pointwise:AwayFromBoundary} if $\delta_x\geq 2\eps$ then the boundary terms disappear and the proof simplifies, as in \cite[Theorem 5]{calder18a}; hence we only prove the case where $\delta_x\leq 2\eps$.
Fix $x\in\Omega$ with $\delta_x\leq 2\eps$.
After making an orthogonal change of coordinates and a translation, we may assume that $x=x_de_d$ (where $e_d=(0,0,\dots,1)$ is the $d$th standard basis vector), $y_x=0\in \partial \Omega$ and
\begin{equation} \label{eq:WellPosed:Pointwise:boundary}
B(x,2\eps)\cap \Omega = \lb y\in B(x,2\eps)\, : \, y_d > F(y_1,\dots,y_{d-1})\rb,
\end{equation}
where $F:\bbR^{d-1} \to  \bbR$ is $\Ck{3}$, with $\|F\|_{\Ck{3}}$ bounded uniformly over $\partial\Omega$, 
\begin{equation} \label{eq:WellPosed:Pointwise:F}
F(0)=0,\ \ \text{ and }\ \ \nabla F(0)=0.
\end{equation}
Here $x_d=\delta_x$, and $\bm{n}(0)=\bm{n}(x)=-e_d$.

For $y\in \bbR^d$ write $y=(\ty,y_d)$ where $\ty\in \bbR^{d-1}$.
First, note that
\begin{equation}\label{eq:WellPosed:Pointwise:gamma_bound}
\frac{1}{2}\ty^\top  A\ty- c\eps^3\leq F(\ty) \leq \frac{1}{2}\ty^\top  A\ty+ c\eps^3,
\end{equation}
where $A= \nabla ^2 F(0)$ and $c$ depends on $\|F\|_{\Ck{3}}$.
Let us write 
\begin{equation}\label{eq:WellPosed:Pointwise:A}
A:=\left\{ y\in B(x,2\eps) \, : \, y_d >\frac{1}{2}\ty^T  A\ty+ c\eps^3  \right\}
\end{equation}
and
\begin{equation}\label{eq:WellPosed:Pointwise:B}
B:=\left\{ y\in B(x,2\eps) \, : \, \frac{1}{2}\ty^T  A\ty- c\eps^3  < y_d \leq \frac{1}{2}\ty^T  A\ty+ c\eps^3  \right\}.
\end{equation}
Defining
\begin{equation}\label{eq:WellPosed:Pointwise:L2}
L\varphi(x) = \frac{2}{\eps^2} \int_{A}\eta_\eps\l |x-y|\r(\varphi(x) - \varphi(y)) \rho(y) \, \dd y,
\end{equation}
(note that $A\subset\Omega$) we have
\begin{align*}
\la \cL_\eps \varphi(x) - L\varphi(x)\ra & \leq \frac{2}{\eps^2} \int_{B(x,2\eps)\cap \Omega\setminus A}\eta_\eps\l |x-y|\r |\varphi(y) - \varphi(x)| \rho(y) \, \dd y \\
&\leq \frac{c K}{\eps^{d+1}} \int_{B(x,2\eps)\cap \Omega\setminus A}\, \dd y \\
&\leq \frac{c K}{\eps^{d+1}} \int_{B}\, \dd y = \frac{cK}{\eps^{d+1}}\Vol(B),
\end{align*}
where $K=\|\varphi\|_{\Ck{3}(\overline{\Omega})}$.
We easily compute $\Vol(B)\leq c\eps^{d+2}$ and so 
\begin{equation}\label{eq:WellPosed:Pointwise:bound1}
\la \cL_\eps \varphi(x) - L\varphi(x)\ra \leq cK\eps.
\end{equation}
We now Taylor expand to obtain
\begin{align*}
L\varphi(x)&= -\frac{2}{\eps^2} \int_{A}\eta_\eps\l |x-y| \r\Big(\nabla \varphi(x)\cdot (y-x) + \frac{1}{2}(y-x)^\top \nabla^2 \varphi(x)(y-x) + O(K\eps^3)\Big) \\
&\hspace{2cm} \l \rho(x) + \nabla \rho(x)\cdot (y-x)+O(\eps^2) \r \, \dd y\\
&=-\frac{2\rho(x)}{\eps^2} \int_A\eta_\eps\l |x-y| \r\nabla \varphi(x)\cdot (y-x)\, \dd y \\
&\hspace{2cm}-\frac{\rho(x)}{\eps^2}\int_A\eta_\eps\l |x-y| \r (y-x)^\top\nabla^2 \varphi(x)(y-x) \, \dd y \\
&\hspace{2cm}-\frac{2}{\eps^2}\int_A\eta_\eps\l |x-y| \r\nabla \varphi(x)\cdot (y-x) \nabla \rho(x)\cdot (y-x)\, \dd y + O(K\eps)\\
&=:E_1+E_2 +E_3 +O(K\eps).
\end{align*}

Note that $A$ is symmetric in $\ty$, i.e., $(\ty,y_d)\in A$ if and only if $(-\ty,y_d)\in A$.
Since the integrand in $E_1$ is odd in $\ty$, we have
\begin{align}
E_1&=-\frac{\rho(x)}{\eps^2} \frac{\partial \varphi}{\partial x_d}(x)\int_A\eta_\eps\l |x-y| \r(y_d-x_d)\, \dd y \notag \\
&= -\frac{2\rho(x)}{\eps^2} \frac{\partial \varphi}{\partial x_d}(x)\int_{B(x,2\eps)\cap \Omega}\eta_\eps\l |x-y| \r (y-x)\cdot e_d\, \dd y + O(K\eps) \notag \\
&= \frac{2\rho(x)}{\eps} \frac{\partial \varphi}{\partial x_d}(x)\gamma_\eps(x)+ O(K\eps). \label{eq:WellPosed:Pointwise:E1}
\end{align}

For $E_2$ we have
\begin{align*}
E_2 &= -\frac{\rho(x)}{\eps^2}\int_{B(x,2\eps)\cap\{y_d>0\}}\eta_\eps\l |x-y| \r(y-x)^\top\nabla^2 \varphi(x)(y-x) \, \dd y + O(K\eps) \\
&=-\rho(x)\int_{B(0,2)\cap\{z_d>-x_d/\eps\}}\eta\l |z| \r z^\top\nabla^2 \varphi(x)z \, \dd z + O(K\eps) \\
&=-\rho(x)\sum_{i,j=1}^d \frac{\partial^2 \varphi}{\partial x_i \partial x_j}(x)\int_{B(0,2)\cap\{z_d>-x_d/\eps\}}\eta\l |z| \r z_iz_j \, \dd z + O(K\eps).
\end{align*}
Any term with $i\neq j$ is odd and vanishes, so we obtain
\[ E_2= -\rho(x)\sum_{i=1}^d \frac{\partial^2 \varphi}{\partial x_i^2}(x)\int_{B(0,2)\cap\{z_d>-x_d/\eps\}}\eta\l |z| \r z_i^2 \, \dd z + O(K\eps), \]
and so we have
\begin{equation}\label{eq:WellPosed:Pointwise:E2}
E_2= -\rho(x) \l \sigma_1\l \tfrac{x_d}{\eps}\r  \Delta \varphi + \sigma_2\l \tfrac{x_d}{\eps}\r \frac{\partial^2 \varphi}{\partial x_d^2}\r+ O(K\eps).
\end{equation}

Finally, for $E_3$ we have
\begin{align}
E_3 &= -\frac{2}{\eps^2} \int_{B(x,2\eps)\cap\{y_d>0\}}\eta_\eps\l |x-y| \r \nabla \varphi(x)\cdot (y-x) \nabla \rho(x)\cdot (y-x) \, \dd y + O(K\eps) \notag \\
&=-2\int_{B(0,2)\cap\{z_d>-x_d/\eps\}}\eta\l |z| \r\nabla \varphi(x)\cdot z \nabla \rho(x)\cdot z \, \dd z + O(K\eps) \notag \\
&=-2\sum_{i,j=1}^d \frac{\partial \varphi}{\partial x_i}(x) \frac{\partial \rho}{\partial x_j}(x) \int_{B(0,2)\cap\{z_d>-x_d/\eps\}}\eta\l |z| \r z_iz_j \, \dd z + O(K\eps) \notag \\
&=-2\sum_{i=1}^d \frac{\partial \varphi}{\partial x_i}(x) \frac{\partial \rho}{\partial x_i}(x) \int_{B(0,2)\cap\{z_d>-x_d/\eps\}}\eta\l |z| \r z_i^2 \, \dd z + O(K\eps) \notag \\
&=-2\l \sigma_1\l \tfrac{x_d}{\eps} \r \nabla \varphi(x)\cdot \nabla \rho(x) + \sigma_2\l \tfrac{x_d}{\eps} \r \frac{\partial \varphi}{\partial x_d}(x) \frac{\partial \rho}{\partial x_d}(x)  \r + O(K\eps). \label{eq:WellPosed:Pointwise:E3}
\end{align}

Noting that $\bm{n}(x)=-e_d$ and $x_d=\delta_x$, we combine \eqref{eq:WellPosed:Pointwise:E1}, \eqref{eq:WellPosed:Pointwise:E2}, and \eqref{eq:WellPosed:Pointwise:E3} with \eqref{eq:WellPosed:Pointwise:bound1} to complete the proof.
\end{proof}

We can combine Theorems~\ref{thm:WellPosed:Pointwise:GraphtoNLLap} and~\ref{thm:WellPosed:Pointwise:NLLaptoLLap} to derive the pointwise consistency of graph Laplacians for sufficiently smooth functions on $\Omega$.

\begin{corollary}
\label{cor:WellPosed:Pointwise:GraphtoLLap}
Under the conditions of Theorem~\ref{thm:WellPosed:Pointwise:GraphtoNLLap} and Theorem~\ref{thm:WellPosed:Pointwise:NLLaptoLLap} there exists $C>c>0$ such that for any $\eps\leq\vartheta\leq\frac{1}{\eps}$
\begin{align*}
& \sup_{x\in\Omega_n} \la \L_{n,\eps}\varphi(x) - \frac{\sigma_1\l\frac{\delta_{x}}{\eps}\r}{\sigma_\eta} \cL \varphi(x) - \frac{2\rho(x)\gamma_\eps(x)}{\eps} \frac{\partial \varphi}{\partial \bm{n}}(x) + \frac{\sigma_2\l\frac{\delta_{x}}{\eps}\r}{\rho(x)} \frac{\partial}{\partial \bm{n}} \l\rho^2 \frac{\partial \varphi}{\partial \bm{n}}\r(x) \ra \leq C\|\varphi\|_{\Ck{3}(\overline{\Omega})}\vartheta
\end{align*}
for all $\varphi\in \Ck{3}(\overline{\Omega})$ with probability at least $1-Cne^{-cn\eps^{d+2}\vartheta^2}$.
\end{corollary}

\subsection{Estimates on random graphs} \label{sec:graph}

We recall $x_1,\dots,x_n$ are \emph{i.i.d.}~random variables on $\Omega$ with $\Ck{2}$ density $\rho$, and $\Omega_n=\{x_1,\dots,x_n\}$. Let $z_1,z_2,\dots,z_n$ be \emph{i.i.d.}~Bernoulli random variables with parameter $\beta\in (0,1)$. For simplicity we write $z(x_i)=z_i$. For $x\in \Omega$ and $\wO \subset \Omega$, define 
\begin{equation}\label{eq:degree}
d_{n,\eps}(x) = \sum_{y \in \Omega_n}\eta_\eps(|y-x|),
\end{equation}
\begin{equation}\label{eq:prob}
p_{n,\eps}(x;\wO) = \sum_{y\in \Omega_n}\eta_\eps(|y-x|)\one_{z(y)=1}\one_{y\in \wO},
\end{equation}
and
\begin{equation}\label{eq:drift}
q_{n,\eps}(x) = \sum_{y\in \Omega_n}\eta_\eps(|y-x|)\one_{\delta_y \leq \delta_x-\tfrac{\eps}{2}},
\end{equation}
where we recall that $\delta_x=\dist(x,\partial \Omega)$.

The weights $w_{xy}=\eta_\eps(|y-x|)$ endow $\Omega_n$ with the structure of a \emph{graph}. We say the graph is \emph{connected} if for each $x,y\in \Omega_n$ there is a path $x=y_1,y_2,\dots,y_m=y$ with $y_i \in \Omega_n$ such that $\eta_\eps(|y_i-y_{i-1}|)>0$ for all $i$.

We now establish some basic estimates for $d_{n,\eps}$, $p_{n,\eps}$ and $q_{n,\eps}$ that will allow us to control random walks on the graph.
\begin{proposition}\label{prop:concentration1}
Let $0 \leq \vartheta \leq 1$ and assume {\bf (A1,3)} hold.  The event that 
\begin{equation}\label{eq:dest}
\left|\frac{1}{n}d_{n,\eps}(x) - \int_{\Omega\cap B(x,2\eps)}\eta_\eps(|y-x|) \rho(y) \, \dd y\right|\leq C\vartheta,
\end{equation}
holds for all $x\in \Omega_n$ and $\Omega_n$ is a connected graph, has probability at least $1-Cn\exp(-cn\eps^d\vartheta^2)$.
\end{proposition}
\begin{proof}
It is a standard result that the graph is connected with probability at least $1 - Cn\exp(-cn\eps^d)$; we refer the reader to \cite{penrose03} for details.

To prove the estimate \eqref{eq:dest}, fix $x\in \Omega$. By Lemma~\ref{lem:WellPosed:RWBounds:ConcIneq} applied to $\psi(y) = \eta_\eps(|y-x|)$ we have
\begin{equation}\label{eq:WellPosed:RWBounds:Deg}
\P\l|d_{n,\eps}(x) - \E[d_{n,\eps}(x)]| \geq Cn \vartheta\r \leq 2\exp(-cn\eps^d\vartheta^2)
\end{equation}
for all $0 \leq \vartheta \leq 1$, where
\[\E[d_{n,\eps}(x)] = n\int_{\Omega\cap B(x,2\eps)}\eta_\eps(|y-x|)\rho(y)\, \dd y.\]

While we proved the result for a fixed $x\in \Omega$, we can bound the event in question, where $x$ is any point in $\Omega_n$, by conditioning $x\in \Omega_n$, applying the result above, and union bounding over all $x\in \Omega_n$. 
\end{proof}

\begin{proposition}\label{prop:concentration2}
Let $0 \leq \vartheta \leq 1$, $\wO\subset \wO' \subset \Omega$, and assume {\bf (A1,3)} hold.
Assume $\bbP\l z(y)=1\,|\, y\in\tilde{\Omega}\r = \beta$.
The event that 
\begin{equation}\label{eq:pest}
\left|\frac{1}{n}p_{n,\eps}(x;\wO) - \beta\int_{\wO\cap B(x,2\eps)}\eta_\eps(|y-x|) \rho(y) \, \dd y\right|\leq C\beta \eps^{-d}\Vol\l\wO\cap B(x,2\eps)\r \vartheta
\end{equation}
holds for all $x\in \Omega_n\cap \wO'$ has probability at least $1-Cn\exp\l-c\wc n\beta\eps^d\vartheta^2\r$, where
\begin{equation}\label{eq:wc}
\wc := \min_{x\in \wO'}\eps^{-d}\Vol\l\wO\cap B(x,2\eps)\r.
\end{equation}
\end{proposition}

\begin{proof}
We use Bernstein's inequality with $Y_i = \eta_\eps(|x_i-x|)\one_{z_i=1}\one_{x_i\in \wO}$. Here we have
\[ \E[Y_i] = \E[Y_i \, | \, z_i=1 \, x_i\in\wO]\P(z_i=1\,|\,x_i\in\wO) \P(x_i\in\wO) = \beta\int_{\wO\cap B(x,2\eps)}\eta_\eps(|y-x|)\rho(y)\, \dd y \]
and
\[ \sigma^2 \leq \E[Y_i^2] =\beta \int_{\wO\cap B(x,2\eps)}\eta_\eps(|y-x|)^2\rho(y)\, \dd y\leq C\beta \eps^{-2d} \Vol\l\wO\cap B(x,2\eps)\r. \]
We also have $|Y_i|\leq C\eps^{-d}$. Therefore Bernstein's inequality \eqref{eq:bernstein} yields
\[ \P\l|p_{n,\eps}(x;\wO) - n\E[Y_i]| \geq nt \r \leq 2\exp\l-\frac{cn\eps^dt^2}{\beta \eps^{-d} \Vol\l\wO\cap B(x,2\eps)\r+ t}\r \]
for $t>0$. Set $\vartheta=\beta^{-1} \eps^{d} \Vol\l\wO\cap B(x,2\eps)\r^{-1}t$ to find that
\[\P\l|p_{n,\eps}(x;\wO) - n\E[Y_i]| \geq n\beta\eps^{-d}\Vol\l\wO\cap B(x,2\eps)\r\vartheta \r \leq 2\exp\l\frac{-cn\beta\Vol\l\wO\cap B(x,2\eps)\r\vartheta^2}{1+ \vartheta}\r\]
Restricting $0 \leq \vartheta \leq 1$ and bounding $\Vol\l\wO\cap B(x,2\eps)\r\geq \wc\, \eps^d$ completes the proof. As in Proposition  \ref{prop:concentration1}, we complete the proof by union bounding over $x\in \Omega_n$.
\end{proof}

\begin{lemma}\label{lem:qdrift}
Assume {\bf (A1), (A3)}. Then there exist constants $c_1,c_2>0$ such that 
\begin{equation}\label{eq:driftq}
q_{n,\eps}(x) \geq c_1 n \ \ \ \text{ for all }x\in \Omega_n\setminus \partial_{\eps} \Omega
\end{equation}
holds with probability at least  $1 - 2n\exp\left( -c_2n\eps^d \right)$.
\end{lemma}
\begin{proof}
Since $\partial \Omega$ is $\Ck{2}$, the distance function $x\mapsto \delta_x=\dist(x,\partial\Omega)$ is semiconcave, that is, there exists $C>0$ such that $\delta_x - C|x|^2$ is concave. For clarity, let us write $u(x)=\delta_x$ throuhout the proof. Since $x_1,x_2,\dots,x_n$ are \emph{i.i.d.}~with Lebesgue density and $u$ is differentiable almost everywhere, we have that $u$ is differentiable at every $x_i$ almost surely, and $|\nabla u(x_i)|=1$.

Fix $x\in \Omega$ such that $u$ is differentiable at $x$. Due to the semiconcavity of $u$, the function $y\mapsto u(y) - C|x-y|^2$ is concave, and so
\begin{equation}\label{eq:semiconcave}
u(y) - C|x-y|^2 \leq u(x) + \nabla u(x)\cdot (y-x)
\end{equation}
for all $y\in \Omega$.  Let
\[A = \left\{y\in B(x,\eps)\cap \Omega \, :\, \nabla u(x) \cdot (y-x)\leq -C\eps^2 -\frac{\eps}{2}\right\}.\]
By \eqref{eq:semiconcave} we have that
\[A \subset \left\{ y\in B(x,\eps)\cap \Omega \, : \, u(y) \leq u(x) - \frac{\eps}{2} \right\},\]
and so
\[q_{n,\eps}(x) \geq \sum_{y\in \Omega_n \cap A}\eta_\eps(|y-x|).\]
By the concentration inequality \eqref{lem:WellPosed:RWBounds:ConcIneq} we have that
\[q_{n,\eps}(x) \geq n\int_{B(x,\eps)\cap A}\eta_\eps(|y-x|)\rho(y)\, \dd y - Cn\vt\]
holds with probability at least $1-2\exp\left( -cn\eps^d \vt^2 \right)$ for any $0 \leq \vt \leq 1$.
By {\bf (A1), (A3)} we have
\begin{align*}
\int_{A}\eta_\eps(|y-x|)\rho(y)\, \dd y &\geq \rho_{\min}\eps^{-d}\int_{ A}\, \dd y\geq c\,\rho_{\min},
\end{align*}
for $\eps>0$ sufficiently small, where $c$ depends on $d$ and $\Omega$.  Therefore 
\[q_{n,\eps}(x) \geq n\left(c\, \rho_{\min}- C\vt\right).\]
holds with probability at least $1-2\exp\left( -cn\eps^d \vt^2 \right)$. Choosing $\vt>0$ sufficiently small, and union bounding over $x_1,x_2,\dots,x_n$ completes the proof.
\end{proof}

\begin{proposition}\label{prop:asymp}
Let $x\in \Omega$ such that $B(x,2\eps)\subset \Omega$ and assume $\rho \in \Ck{2}(\bar{\Omega})$. Then 
\begin{equation}\label{eq:degasymp}
\int_{B(x,2\eps)}\eta_\eps(|y-x|)\rho(y) \, \dd y  = C_\eta \rho(x) + O(\eps^2).
\end{equation}
where $C_\eta = \int_{B(0,2)}\eta(|z|)\, dz$.
\end{proposition}
\begin{proof}
We compute
\begin{align*}
\int_{B(x,2\eps)}\eta_\eps(|y-x|)\rho(y) \, \dd y &= \int_{B(0,2)}\eta(|z|)\l\rho(x) + \eps\nabla \rho(x)\cdot z + O(\eps^2)\r \, \dd z\\
&=C_\eta\rho(x) + O(\eps^2).\qedhere
\end{align*}
\end{proof}

\subsection{Random walk bounds} \label{subsec:WellPosed:RWBounds}

In this section, we study random walks on geometric graphs and prove basic estimates using Martingale techniques. 
We always assume $\eps \leq 1$ and $\beta\geq \eps^2$.
Here, we assume $z_1,\dots,z_n\in \{0,1\}$ and $x_1,\dots,x_n\in \Omega$ are given deterministic points, we set $\Omega_n=\{x_1,\dots,x_n\}$. We are also given sets $\wO \subset \wO'\subset \Omega$, which will change between the two models, and we define $\wc$ as in \eqref{eq:wc}.

In this section, we assume the estimates derived in Sections~\ref{subsec:WellPosed:Pointwise} and~\ref{sec:graph} hold.  That is, we assume there exists $C,c>0$ and $\vartheta\in [\eps^2,1]$ such that
\begin{equation}\label{eq:degestimate}
|d_{n,\eps}(x) - C_\eta n \rho(x)| \leq Cn\vartheta \ \ \ \text{for all }x\in \Omega_n\setminus \partial_{2\eps}\Omega,
\end{equation}
\begin{equation}\label{eq:pestimate}
cn \leq d_{n,\eps}(x)\leq Cn \ \ \text{ and }\ \ p_{n,\eps}(x;\wO) \geq c \wc n \beta \ \ \ \text{for all }x\in \Omega_n\cap \wO',
\end{equation}
\begin{equation}\label{eq:qestimate}
q_{n,\eps}(x)\geq cn \ \ \ \text{for all }x\in \Omega_n\setminus \partial_\eps\Omega,
\end{equation}
and 
\begin{equation}\label{eq:Lestimate}
|\L_{n,\eps}\varphi(x) - \L \varphi(x)| \leq C\|\varphi\|_{\Ck{3}(\bar{\Omega})}\vartheta \eps^{-1} \ \ \ \text{for all }\varphi\in \Ck{3}(\bar{\Omega}) \text{ and }x\in \Omega_n\setminus \partial_{2\eps}\Omega,
\end{equation}
where $d_{n,\eps}$, $p_{n,\eps}$, $q_{n,\eps}$, $\L_{n,\eps}$ and $\L$ are defined in \eqref{eq:degree}, \eqref{eq:prob}, \eqref{eq:drift}, \eqref{eq:GL} and~\eqref{eq:WellPosed:Pointwise:ContL} respectively.  When the graph is random, as in Section \ref{sec:graph}, this holds with probability at least $1 - Cn\exp(-cn\eps^d\vartheta^2) - Cn\exp(-c\wc n\beta\eps^d)$, due to Propositions \ref{prop:concentration1}, \ref{prop:concentration2} and \ref{prop:asymp}, Lemma~\ref{lem:qdrift}, and Corollary \ref{cor:WellPosed:Pointwise:GraphtoLLap}. Some results require only a subset of the estimates above, and we will indicate this in the results below.

Let $x\in \Omega_n$. We consider a random walk on the graph $\Omega_n$ with transition probabilities 
\begin{equation}\label{eq:transition}
\P(X_k = y \, | \, X_{k-1}=x) = \frac{\eta_{\eps}(|y-x|)}{d_{n,\eps}(x)}
\end{equation}
for $k\geq 1$.  We first recall and precisely state   the well-known  connection between the random walk on the graph and the Laplacian. We note that the difference between the random walk Laplacian and the variational graph Laplacian considered in \eqref{eq:GL} is the weighted degree in the denominator, which results in the extra $\rho$ in the denominator of the limiting weighted Laplacian.
\begin{proposition}\label{prop:rwlap}
Let $\vartheta\in (0,1)$ and $\eps\in (0,1)$.
Assume {\bf (A3)}, \eqref{eq:degestimate} and \eqref{eq:Lestimate} hold. Let $(X_k)_{k\geq 0}$ be a random walk on $\Omega_n$ with transition probabilities \eqref{eq:transition} starting at $X_0=x\in \Omega_n$, and let $\F_k$ denote the $\sigma$-algebra generated by $X_0,X_1,\dots,X_k$. Then for every $\varphi\in \Ck{3}(\bar{\Omega})$
\begin{equation}\label{eq:rwdrift}
\E[\phi(X_{k})-\phi(X_{k-1}) \, | \, \F_{k-1}] =\frac{\sigma_\eta}{C_\eta} \rho^{-2}\Div\left( \rho^2 \nabla \phi \right)\Big\vert_{X_{k-1}}\eps^2 + O\l\eps\vartheta\|\varphi\|_{\Ck{3}(\overline{\Omega})}\r,
\end{equation}
whenever $\dist(X_{k-1},\partial \Omega)\geq 2\eps$.
\end{proposition}

\begin{proof}
Due to \eqref{eq:degestimate} and \eqref{eq:Lestimate} we compute
\begin{align*}
\E[\phi(X_k) - \phi(X_{k-1}) | \F_{k-1}]&=\sum_{y\in \Omega_n}\frac{\eta_\eps(|y-X_{k-1}|)}{d_{n,\eps}(X_{k-1})}(\phi(y)-\phi(X_{k-1}))\\
&= - \eps^2\L_{n,\eps}\phi (X_{k-1})\left(\frac{d_{n,\eps}(X_{k-1})}{n}\right)^{-1}\\
&= - \eps^2\left(\L \phi(X_{k-1}) + O\left( \|\varphi\|_{\Ck{3}(\bar{\Omega})}\vartheta \eps^{-1} \right) \right)\left(C_\eta\rho(X_{k-1}) + O(\vartheta)\right)^{-1}\\
&=\frac{\sigma_\eta}{C_\eta} \rho^{-2}\div\left( \rho^2 \nabla \phi \right)\Big\vert_{X_{k-1}}\eps^2 + O\l\vartheta\eps\|\varphi\|_{\Ck{3}(\overline{\Omega})}\r,
\end{align*}
which completes the proof.
\end{proof}
   
We now use Azuma's inequality to bound how far the random walk travels in $k$ steps. The random walk has a small drift, due to irregularities in the graph $\Omega_n$, that must be accounted for. For completeness, we give a short proof of Azuma's inequality in Appendix \ref{sec:azuma}.

\begin{lemma}\label{lem:azuma}
Let $\eps \in (0,1)$, $\vartheta\in (\eps,1)$ and $x\in \Omega_n$.
Assume {\bf (A3)},  \eqref{eq:degestimate} and \eqref{eq:Lestimate} hold.
Let $(X_k)_{k\geq 0}$ be a random walk on $\Omega_n$ with transition probabilities \eqref{eq:transition} starting at $X_0=x$. For any $r>0$ with $r\leq \dist(x,\partial \Omega) - 2\eps$ we have
\begin{equation} \label{eq:rwdist}
\P\left(\max_{1\leq i\le k}|X_i  - X_0 | > r\right)\leq 2d\exp\l-\frac{c(r-Ck\eps\vartheta)_+^2}{k\eps^2}\r
\end{equation}
for all $k\in \bbN$.
\end{lemma}

\begin{proof}
Let $\F_k$ denote the $\sigma$-algebra generated by $X_0,X_1,\dots,X_k$. Define the stopping time
\begin{equation}\label{eq:stop}
\tau = \inf\{ k\geq 0 \, : \, |X_k - x| > r\},
\end{equation} 
and set $Y_k= X_{k\wedge \tau}$. By Proposition \ref{prop:rwlap} we have
\[\left|\E[Y_k - Y_{k-1} | \F_{k-1}]\right| = \one_{\tau > k-1}|\E[X_k - X_{k-1} | \F_{k-1}]|\leq C(\eps\vartheta + \eps^2) \leq C\eps\vartheta\]
since $\vartheta\geq \eps$. It follows that $Z_k = Y_k\cdot \e_i - Ck\eps\vartheta$ is a supermartingale with respect to $\F_k$ for each $i$, with increments bounded by $C\eps$, since $\vartheta\leq 1$. Azuma's inequality  yields
\[\P(Y_k\cdot \e_i - Y_0\cdot \e_i \geq t + Ck\eps\vartheta)=\P(Z_k-Z_0 \geq t) \leq \exp\l-\frac{ct^2}{k\eps^2}\r.\]
for all $t>0$. Applying the same argument to the submartingale $Z_k = Y_k\cdot \e_i + Ck\eps\vartheta$ and summing over $i$ yields
\[\P(|Y_k  - Y_0 | \geq t + Ck\eps\vartheta)\leq 2d\exp\l-\frac{ct^2}{k\eps^2}\r\]
for any $t>0$. Choose $t = r - Ck\eps\vartheta$ and assume $r > Ck\eps\vartheta$. If $\max_{1\leq i \leq k}|X_i-X_0|> r$ then $\tau \leq k$ and so $|Y_k-Y_0|>r$. It follows that
\[\P\left(\max_{1\leq i\le k}|X_i  - X_0 | > r\right)\leq 2d\exp\l-\frac{c(r-Ck\eps\vartheta)^2}{k\eps^2}\r,\]
which completes the proof.
\end{proof}   

Let $\wO\subset \Omega$ be open and set
\begin{equation}\label{eq:Gamma}
\Gamma_n = \{x_i\in \Omega_n \, : \, z_i=1 \text{ and }x_i\in \wO\}.
\end{equation}
We consider the boundary value problem
\begin{equation}\label{eq:bvp}
\left\{\begin{aligned}
\sum_{y\in \Omega_n}\eta_\eps(|y-x|)(u(y)-u(x)) &= 0,&&\text{if }x\in\Omega_n\setminus \Gamma_n\\ 
u(x) &=g(x),&&\text{if }x\in \Gamma_n,
\end{aligned}\right.
\end{equation}
which includes both {\bf Model 1} and {\bf Model 2} for appropriate choices of $\wO$. 

We first prove a boundary estimate for {\bf Model 1}. Here, $\wO=\wO'\subset \Omega$ are fixed, and so $\wc>0$ is bounded away from zero, independently of the parameters $n,\eps,\beta$.

\begin{theorem}\label{thm:boundary}
Let $\beta\in [\eps^2,1]$ and $\eps\in (0,1)$.
Assume {\bf (A3)} and that \eqref{eq:degestimate}, \eqref{eq:pestimate} and \eqref{eq:Lestimate} hold with $\vartheta=\sqrt{\beta}$ and $\wO^\prime=\wO$.
Let $g\in \Ck{0,1}(\bar{\Omega})$ and let $u_n$ solve \eqref{eq:bvp}. Then there exists $C,K>0$ such that for every $x\in \Omega_n\cap\wO$ with $\dist(x,\partial\wO)\geq K\beta^{-\frac12}\eps\log(\sqrt{\beta}\eps^{-1})$ we have
\begin{equation}\label{eq:boundaryestimate}
|u_n(x) - g(x)| \leq \frac{C\eps}{\sqrt{\beta}}\log\l\frac{\sqrt{\beta}}{\eps}\r.
\end{equation}
\end{theorem}
\begin{proof}
First, we note that if $\beta < 2\eps^2$, then the maximum principle yields
\[|u_n(x)-g(x)| \leq 2\max_{y\in \Omega_n}|g(y)| \leq 2\sqrt{2}\max_{y\in \Omega_n}|g(y)|\frac{\eps}{\sqrt{\beta}}.\]
Hence, we may assume $\beta\geq 2\eps^2$ in the remainder of the proof.

Let $(X_k)_{k\geq 0}$ be a random walk on $\Omega_n$ with transition probabilities \eqref{eq:transition} starting at $X_0=x$, and let $\F_k$ denote the $\sigma$-algebra generated by $X_0,X_1,\dots,X_k$. We can assume $x\not\in \Gamma_n$. Define the stopping times
\[\tau_1 = \inf\{k\geq 0 \, : \, X_k\in \Gamma_n\},\]
and
\[\tau_2 = \inf\{k\geq 0\, : \, |X_k-x|>r\},\]
for $0 < r \leq \dist(x,\partial \wO)$ to be determined.

Let $Y_k = u_n(X_{k\wedge \tau_1})$. We claim that $Y_k$ is a martingale with respect to $\F_k$. Indeed, we compute
\begin{align*}
\E[Y_k - Y_{k-1} | \F_{k-1}]&=\one_{\tau_1 > k-1}\E[u_n(X_{k})-u_n(X_{k-1}) | \F_{k-1}] + \one_{\tau_1 \leq k-1}\E[u_n(X_{\tau_1}) - u_n(X_{\tau_1}) | \F_{k-1}]\\
&=\one_{\tau_1 > k-1}\left(\frac{1}{d_{n,\eps}(X_{k-1})}\sum_{y\in \Omega_n}\eta_\eps(|y-X_{k-1}|)\l u_n(y) - u_n(X_{k-1})\r\right) = 0,
\end{align*}
due to \eqref{eq:bvp}, which establishes the claim.
By Doob's optional stopping theorem we have
\begin{align*}
u_n(x) - g(x)&=u_n(X_0) - g(x)\\
&=\E[u_n(X_{\tau_1\wedge\tau_2}) - g(x)]\notag \\
&=\E[g(X_{\tau_1}) - g(x) \, | \, \tau_1 \leq \tau_2] \P(\tau_1 \leq \tau_2) + \E[u_n(X_{\tau_2}) - g(x) \, | \, \tau_1 > \tau_2] \P(\tau_1>\tau_2)\notag \\
&\leq \|g\|_{\Ck{0,1}(\bar{\Omega})}(r+2\eps) + 2\max_{y\in \Omega_n}|g(y)|\P(\tau_1 > \tau_2).\notag 
\end{align*}
The opposite inequality is proved similarly, yielding
\begin{equation}\label{eq:keyug}
|u_n(x) - g(x)| \leq\|g\|_{\Ck{0,1}(\bar{\Omega})}(r+2\eps) + 2\max_{y\in \Omega_n}|g(y)|\P(\tau_1 > \tau_2).
\end{equation}

We now bound $\P(\tau_1 > \tau_2)$. By Lemma \ref{lem:azuma}
\[\P(\tau_2 \leq k) \leq 2d\exp\l-\frac{c(r-Ck\eps\sqrt{\beta})_+^2}{k\eps^2}\r.\]
Note that by \eqref{eq:pestimate} we have 
\[\P(X_k \in \Gamma_n \, | \, X_{k-1}\in \wO) = \frac{p_{n,\eps}(X_{k-1})}{d_{n,\eps}(X_{k-1})} \geq \frac{cn\beta}{Cn}\geq c\beta.\]
Since $\tau_2>k$ implies that $X_i \in \wO$ for $i=1,\dots,k$ we have
\[\P(\tau_1 > k \, |\,  \tau_2 > k) \leq (1-c\beta)^k.\]
It follows that
\begin{align*}
\P(\tau_1 > k)&=\P(\tau_1 > k \, |\,  \tau_2 > k)\P(\tau_2 > k) + \P(\tau_1 > k\,  |\,  \tau_2 \leq k)\P(\tau_2 \leq k)  \\
&\leq (1-c\beta)^k + 2d\exp\l-\frac{c(r-Ck\eps\sqrt{\beta})_+^2}{k\eps^2}\r.
\end{align*}
For any $k$, if $\tau_1 > \tau_2$ then either $\tau_1 > k$ or $\tau_2 \leq k$. Therefore
\[\P(\tau_1 > \tau_2)\leq \P(\tau_1 > k) + \P(\tau_2 \leq k) \leq \exp\left( -c\beta k \right) + 4d\exp\l-\frac{c(r-Ck\eps\sqrt{\beta})_+^2}{k\eps^2}\r.\]
Choose $k\in \Z$ so that
\begin{equation}\label{eq:kchoice}
1 + \log\left( \sqrt{\beta}\eps^{-1} \right)\leq c\beta k \leq \log\left(\sqrt{\beta}\eps^{-1} \right) + 1 + c\beta.
\end{equation}
Then $e^{-c\beta k} \leq \frac{\eps}{\sqrt{\beta}}$, and since $\beta \geq 2\eps^2$ and $c,\beta\in (0,1]$ we see that $k\geq 1$. Therefore
\[\P(\tau_1 > \tau_2)\leq \frac{\eps}{\sqrt{\beta}} + 4d\exp\l-\frac{c(r-Ck\eps\sqrt{\beta})_+^2}{k\eps^2}\r.\]
Since $\beta\geq 2\eps^2$ and $k$ satisfies \eqref{eq:kchoice} we have
\[\frac{(r-Ck\eps\sqrt{\beta})^2_+}{k\eps^2} \geq \frac{\widetilde{c}\beta\left( r - \widetilde{C}\beta^{-\frac12}\eps\log\left(\sqrt{\beta}\eps^{-1}\right)\right)^2_+}{\eps^2\log\left(\sqrt{\beta}\eps^{-1}\right)},\]
for constants $\widetilde{C}\geq 1$ and $0\leq \widetilde{c}\leq 1$. Setting $r = K\beta^{-\frac12}\eps\log(\sqrt{\beta}\eps^{-1})$ for sufficiently large $K$ and assuming $\dist(x,\partial\wO)\geq K\beta^{-\frac12}\eps\log(\sqrt{\beta}\eps^{-1})$, we have
\[\P(\tau_1 > \tau_2)\leq \frac{\eps}{\sqrt{\beta}} + 2d\exp\l-\log\left(\sqrt{\beta}\eps^{-1}\right)\r = (1+2d)\frac{\eps}{\sqrt{\beta}}.\]
The proof is completed by inserting this into \eqref{eq:keyug}.
\end{proof}

We now turn to {\bf Model 2}, and prove an error estimate from which Theorem \ref{thm:wellposed2} immediately follows. We recall in {\bf Model 2} we take
\begin{equation}\label{eq:wOtwo}
\wO=\partial_{\delta}\Omega,
\end{equation}
where $0 < \delta \leq \eps$. In Proposition \ref{prop:concentration2} and \eqref{eq:pestimate} we take $\wO' = \partial_\eps \Omega$.  In this case we have
\begin{equation}\label{eq:wclower}
\wc = \min_{x\in \partial_\eps \Omega}\eps^{-d}\Vol\l\partial_\delta \Omega\cap B(x,\eps)\r \geq c\delta\eps^{-1}.
\end{equation}

\begin{theorem}\label{thm:deltaboundary}
Let $0 < \delta \leq \eps$ and set $\wO=\partial_\delta \Omega$ and $\wO' = \partial_\eps \Omega$. Assume \eqref{eq:pestimate} and \eqref{eq:qestimate} hold. Let $g\in \Ck{0,1}(\bar{\Omega})$, let $u_n$ solve \eqref{eq:bvp}, and let $u$ be the solution of \eqref{eq:model2pde}. There exists $C,c>0$ such that if \eqref{eq:degestimate}  and \eqref{eq:Lestimate} hold with $\eps^2\leq \vt\leq c\,\eps$ then   
\begin{equation}\label{eq:boundaryestimate2}
\max_{x\in \Omega_n}|u_n(x)-u(x)|\leq \frac{C\vt}{\beta\delta}.
\end{equation}
\end{theorem}

\begin{proof}
The proof is split into 3 parts.

1. We first show that
\begin{equation}\label{eq:interiorbound}
\max_{x\in \Omega_n}|u_n(x) - u(x)|\leq \max_{x\in \Omega_n\cap \partial_{2\eps}\Omega}|u_n(x)-u(x)| + C\vt\eps^{-1}.
\end{equation}
This estimate can be proved with the maximum principle or a martingale argument. We give the martingale argument here. Let $x\in \Omega_n$ with $\dist(x,\partial \Omega) > 2\eps$.  Let $(X_k)_{k\geq 0}$ be a random walk on $\Omega_n$ with transition probabilities \eqref{eq:transition} starting at $X_0=x$, and let $\F_k$ denote the $\sigma$-algebra generated by $X_0,X_1,\dots,X_k$.  Define the stopping time
\[\tau = \inf\{k\geq 0 \, : \, X_k\in \partial_{2\eps}\Omega\},\]
and set $Y_k= X_{k\wedge \tau}$.    We claim that
\begin{equation}\label{eq:meanest}
\E[\tau] \leq \frac{C}{\eps^2}.
\end{equation}
To see this, let $\varphi\in \Ck{3}(\bar{\Omega})$ be the solution of
\begin{equation}\label{eq:phidrift}
\left\{\begin{aligned}
-\div(\rho^2\nabla \phi) &= \rho^2&&\text{in }\Omega\\ 
\phi &=0&&\text{on }\partial \Omega.
\end{aligned}\right.
\end{equation}
By Proposition \ref{prop:rwlap} we have
\begin{align*}
\E[\phi(Y_k) - \phi(Y_{k-1})\, | \,\F_{k-1}]&=\one_{\tau > k-1}\E[\phi(X_k) - \phi(X_{k-1}) | \F_{k-1}]\\
&\leq\one_{\tau > k-1}\left( \frac{\sigma_\eta}{C_\eta}\rho^{-2}\div(\rho^2\nabla \varphi )\Big\vert_{X_{k-1}}\eps^2 + C\vt\eps\right)\\
&=\one_{\tau > k-1}C\eps(\vt-c\eps).
\end{align*}
Assume $\vt\leq \tfrac{c\,\eps}{2}$ so that $\E[\phi(Y_k) - \phi(Y_{k-1}) | \F_{k-1}] \leq -\one_{\tau>k-1}c\,\eps^2$. Then we have that 
\[Z_k := \phi(Y_k) + c\,\eps^2 k\one_{\tau>k-1}\]
is a supermartingale. Let $m\geq 0$. By Doob's optional stopping theorem  $\E[Z_{\tau\wedge m}]\leq \E[Z_0]$,  which yields
\[\E[\phi(Y_{\tau\wedge m})] + c\, \eps^2 \E[\tau\wedge m ] \leq \phi(X_0) = \phi(x).\]
By the maximum principle $\varphi\geq 0$, and so we obtain
\[\E[\tau \wedge m] \leq C\frac{\phi(x)}{\eps^2}\leq \frac{C}{\eps^2},\]
for all $m\geq 0$. Sending $m\to \infty$ establishes the claim \eqref{eq:meanest}.

By Proposition \ref{prop:rwlap} we have
\begin{align*}
\E[u(Y_k) - u(Y_{k-1})\, | \,\F_{k-1}]&=\one_{\tau_1 > k-1}\E[u(X_k) - u(X_{k-1}) | \F_{k-1}]\\
&= \one_{\tau_1 > k-1}\left( \frac{\sigma_\eta}{C_\eta}\rho^{-2}\div(\rho^2\nabla u)\Big\vert_{X_{k-1}}\eps^2 + O(\eps^2)\right)=O(\vt\eps).
\end{align*}
As in the proof of Theorem \ref{thm:boundary}, $u_n(Y_k)$ is a martingale, and so $u_n(Y_k) - u(Y_k) + C\vt\eps k$ is a submartingale. By Doob's optional stopping theorem 
\[u_n(x) - u(x)=u_n(X_0) - u(X_0)\leq \E[u_n(X_{\tau}) - u(X_\tau) + C\vt\eps\tau]\leq \E[|u_n(X_\tau)-u(X_\tau)|] + C\vt\eps\E[\tau].\]
Similarly, $u(Y_k) - u_n(Y_k) + C\vt\eps k$ is a submartingale, and so we also obtain
\[u(x) - u_n(x)\leq \E[|u_n(X_\tau)-u(X_\tau)|] + C\vt\eps\E[\tau].\]
Since $X_\tau \in \partial_{2\eps}\Omega$ and $\E[\tau]\leq C/\eps^2$ we have
\[|u_n(x) - u(x)|\leq\E[|u_n(X_\tau)-u(X_\tau)|] + C\vt\eps\E[\tau]\leq \max_{y\in \Omega_n\cap \partial_{2\eps}\Omega}|u_n(y)-u(y)| + C\vt\eps^{-1}\]
for $x\in \Omega_n$ with $\dist(x,\partial \Omega) > 2\eps$. For $x\in \Omega_n\cap \partial_{2\eps}\Omega$ the estimate above holds trivially, so this establishes \eqref{eq:interiorbound}.

2. We now estimate the boundary term in \eqref{eq:interiorbound}. In particular, we claim that 
\begin{equation}\label{eq:iterate}
\max_{x\in \Omega_n\cap \partial_{2\eps}\Omega}|u_n(x) -u(x)| \leq C\eps + (1- c\beta\delta\eps^{-1})\max_{x\in \Omega_n\cap \partial_{4\eps}\Omega}|u_n(x)-u(x)|.
\end{equation}
To see this, let $(X_k)_{k\geq 0}$ be a random walk on $\Omega_n$ with transition probabilities \eqref{eq:transition} starting at $X_0=x$. We first assume $x\in \Omega_n \cap \partial_{\eps}\Omega$.
Then by~\eqref{eq:pestimate} we have
\[\P(X_1\in \Gamma_n) = \frac{p_{n,\eps}(x;\partial_\delta\Omega)}{d_{n,\eps}(x)} \geq \frac{c\wc n\beta}{Cn}\geq c\,\beta\delta\eps^{-1}.\]
Therefore, we have
\begin{align*}
|u_n(x) - u(x)| &= |\E[u_n(X_1) - u(x)]|\\
&\leq |\E[u_n(X_1) - u(x)\, | \, X_1 \in \Gamma_n]| \P(X_1\in \Gamma_n)+|\E[u_n(X_1) - u(x) \, | \, X_1\not\in \Gamma_n]| \P(X_1\not\in \Gamma_n)\\
&\leq \E[|g(X_1) - u(X_1)|\,  | \, X_1 \in \Gamma_n]+\E[|u_n(X_1) - u(X_1)| \, | \, X_1\not\in \Gamma_n] (1-c\beta\delta\eps^{-1}) + C\eps\\
&\leq C\eps + (1-c\beta\delta\eps^{-1})\max_{y\in \Omega_n\cap \partial_{3\eps}\Omega } |u_n(y)-u(y)|.
\end{align*}
We now consider $x\in \Omega_n \cap (\partial_{3\eps/2}\Omega\setminus \partial_{\eps}\Omega)$.
By \eqref{eq:pestimate} and \eqref{eq:qestimate} 
\[\P(X_1 \in \partial_{\eps}\Omega) \geq \frac{q_{n,\eps}(x)}{d_{n,\eps}(x)}\geq c_1>0.\]
Using the bound above we have
\begin{align*}
|u_n(x) - u(x)| &= |\E[u_n(X_1) - u(x)]|\\
&\leq \E[|u_n(X_1) - u(X_1)|\, | \, X_1\in\partial_{\eps}\Omega]\P(X_1\in \partial_{\eps}\Omega)\\
&\hspace{1in}+\E[|u_n(X_1) - u(X_1)| \, | \, X_1\not\in \partial_{\eps}\Omega] \P(X_1\not\in \partial_{\eps}\Omega) + C\eps\\
&\leq \left(C\eps + (1-c\beta\delta\eps^{-1})\max_{y\in \Omega_n\cap \partial_{3\eps}\Omega } |u_n(y)-u(y)|\right)\P(X_1\in \partial_{\eps}\Omega) \\
&\hspace{1.5in}+\max_{y\in \Omega_n\cap \partial_{7\eps/2} \Omega} |u_n(y)-u(y)|(1-\P(X_1\in \partial_{\eps}\Omega)) + C\eps \\ 
&\leq C\eps + (1-c_1c\beta\delta\eps^{-1})\max_{y\in \Omega_n\cap \partial_{7\eps/2} \Omega} |u_n(y)-u(y)|.
\end{align*}
Repeating the same argument for $x\in\Omega_n\cap( \partial_{2\eps}\Omega\setminus \partial_{3\eps/2}\Omega)$ completes the proof of the claim.

3. Combining \eqref{eq:interiorbound} with \eqref{eq:iterate}, and noting $\vt\geq \eps^2$, we obtain
\begin{equation}\label{eq:usefulbound}
\max_{x\in \Omega_n}|u_n(x) - u(x)| \leq (1- c\beta\delta\eps^{-1})\max_{y\in \Omega_n\cap \partial_{4\eps}\Omega}|u_n(y)-u(y)|+ C\vt\eps^{-1}.
\end{equation}
Since
\[\max_{y\in \Omega_n\cap \partial_{4\eps}\Omega}|u_n(y)-u(y)|\leq \max_{x\in \Omega_n}|u_n(x) - u(x)|\]
we can rearrange \eqref{eq:usefulbound} to obtain
\[\max_{x\in \Omega_n}|u_n(x) - u(x)| \leq C\beta^{-1}\delta^{-1}\vt,\]
which completes the proof.
\end{proof}

\subsection{Proofs of well-posedness results}
\label{sec:wellposedproofs}

We now complete the proofs of the well-posedness results. Our proof uses the maximum principle and a barrier argument. We first define the barrier and establish some technical results.

\begin{lemma}
\label{lem:barrier}
Assume {\bf (A1-3)} and {\bf Model 1} hold.
For $\alpha>0$ define
\begin{equation}\label{eq:enlarged}
\wO_\alpha = \{x\in \tilde{\Omega} \, : \, \dist(x,\partial\wO) > \alpha\}.
\end{equation}
Let $w_\alpha$ solve 
\begin{equation}\label{eq:WellPosed:Subset:Hard:ELEq}
\left\{\begin{aligned}
\L w_\alpha &= 0&&\text{in }\Omega\setminus \wO_{\alpha+2\eps}\\ 
w_\alpha &=g&&\text{in } \wO_{\alpha+2\eps}\\
\frac{\partial w_\alpha}{\partial \bm{n}} &=0&&\text{on }\partial \Omega,
\end{aligned}\right.
\end{equation} 
and $\varphi_\alpha$ solve
\begin{equation}\label{eq:phi}
\left\{\begin{aligned}
\L \phi_\alpha &= 1&&\text{in }\Omega\setminus \wO_{\alpha+2\eps}\\ 
\phi_\alpha &=g&&\text{in } \wO_{\alpha+2\eps}\\
\frac{\partial \varphi_\alpha}{\partial \bm{n}} &=1&&\text{on }\partial \Omega.
\end{aligned}\right.
\end{equation}
Then there exists $C>c>0$ such that for all $0 \leq \alpha \leq c$ and $\eps<\vartheta<\frac{1}{\eps}$, with probability at least $1-Cn\exp(-cn\eps^{d+2}\vartheta^2)$ we have $\|\varphi_\alpha\|_{\Linfty(\Omega)}\leq C$, $\|w_\alpha - u\|_{\Linfty(\Omega)} \leq C(\alpha+\eps)$,
\begin{align}
\la \L_{n,\eps}w_\alpha(x)\ra & \leq C\l\gamma_\eps(x) +\vartheta\r \label{eq:wellposed:wellposedproofs:Lapw} \\
\la \L_{n,\eps}\varphi_\alpha(x) - \frac{\sigma_1\l\frac{\delta_{x}}{\eps}\r}{\sigma_\eta} - \frac{2\rho(x)\gamma_{\eps}(x)}{\eps} \ra & \leq C\l\gamma_\eps(x) +\vartheta\r \label{eq:wellposed:wellposedproofs:Lapphi}
\end{align}
for all $x\in \Omega_n\setminus \wO_{\alpha}$, where $u$ is the solution of \eqref{eq:model1pde}.
\end{lemma}

\begin{proof}
By elliptic regularity~\cite{gilbarg2015elliptic} we have that $w_\alpha$ and $\varphi_\alpha$ are $\Ck{3}(\bar{\Omega}\setminus\tilde{\Omega})$ and moreover
\[ \sup_{0\leq\alpha\leq c} \max\{\|w_\alpha\|_{\Ck{3}(\bar{\Omega}\setminus\wO)},\|\varphi_\alpha\|_{\Ck{3}(\bar{\Omega}\setminus\wO)}\}<+\infty. \]
By Corollary~\ref{cor:WellPosed:Pointwise:GraphtoLLap}, for $x\in \Omega_n\setminus\wO_{\alpha}$ we have, with probability at least $1-Cne^{-cn\eps^{d+2}\vartheta^2}$,
\[ \L_{n,\eps}w_\alpha(x) = \frac{2\rho(x)\gamma_\eps(x)}{\eps} \frac{\partial w_\alpha}{\partial\bm{n}}(x) - \frac{\sigma_2\l\tfrac{\delta_{x}}{\eps}\r}{\rho(x)} \frac{\partial}{\partial\bm{n}} \l\rho^2\frac{\partial w_\alpha}{\partial \bm{n}}\r(x) + O(\vartheta) \]
for any $\eps\leq\vartheta\leq \frac{1}{\eps}$.
So, since $\frac{\gamma_\eps(x)}{\eps} \frac{\partial w_\alpha}{\partial \bm{n}}(x) = O\l\frac{\delta_{x}}{\eps}\mathds{1}_{\delta_{x}\leq 2\eps}\r = O(\gamma_\eps(x))$,
\[ \L_{n,\eps}w_\alpha(x) = O\l \gamma_{\eps_n}(x) + \la \sigma_2\l\tfrac{\delta_{x}}{\eps}\r \ra + \vartheta\r. \]
By the lemma below $\la \sigma_2\l\tfrac{\delta_{x}}{\eps}\r \ra = O\l\gamma_\eps(x)+\eps\r$ which completes the proof of~\eqref{eq:wellposed:wellposedproofs:Lapw}.

For~\eqref{eq:wellposed:wellposedproofs:Lapphi}, we again apply Corollary~\ref{cor:WellPosed:Pointwise:GraphtoLLap} to imply that, with probability at least $1-Cne^{-cn\eps^{d+2}\vartheta^2}$,
\[ \L_{n,\eps}\varphi_\alpha(x) = \frac{\sigma_1\l\tfrac{\delta_{x}}{\eps}\r}{\sigma_\eta} + \frac{2\rho(x)\gamma_\eps(x)}{\eps} \frac{\partial \varphi_\alpha}{\partial \bm{n}}(x) - \frac{\sigma_2\l\tfrac{\delta_{x}}{\eps}\r}{\rho(x)} \frac{\partial}{\partial\bm{n}} \l\rho^2\frac{\partial \varphi_\alpha}{\partial \bm{n}}\r(x) + O(\vartheta) \]
for any $\eps\leq\vartheta\leq \frac{1}{\eps}$.
Hence,
\begin{align*}
\L_{n,\eps}\varphi_\alpha(x) - \frac{\sigma_1\l\tfrac{\delta_{x}}{\eps}\r}{\sigma_\eta} - \frac{2\rho(x)\gamma_\eps(x)}{\eps} & = \frac{2\rho(x)\gamma_\eps(x)}{\eps}\l \frac{\partial \varphi_\alpha}{\partial \bm{n}}(x) - 1\r + O\l \la \sigma_2\l\tfrac{\delta_{x}}{\eps}\r\ra + \eps+\vartheta \r \\
 & = O\l\gamma_{\eps}(x) + \la \sigma_2\l\tfrac{\delta_{x}}{\eps}\r\ra + \eps+\vartheta \r
\end{align*}
since $\frac{\partial \varphi_\alpha}{\partial \bm{n}}(x) = 1+O(\delta_{x})$ for $\delta_{x}\leq 2\eps$.
We conclude~\eqref{eq:wellposed:wellposedproofs:Lapphi} by applying the lemma below.

To show $\|w_\alpha - u\|_{\Linfty(\Omega)} \leq C(\alpha+\eps)$ we let $v_\alpha = w_\alpha-u$ and note that $\cL v_\alpha = 0$ in $\Omega\setminus \tilde{\Omega}$.
Hence the maximum of $v_\alpha$ is achieved on $\tilde{\Omega}\setminus \tilde{\Omega}_{\alpha+2\eps}$.
Since $v_\alpha = 0$ on $\partial \tilde{\Omega}_{\alpha+2\eps}$ and is Lipschitz continuous (with Lipschitz constant independent of $\alpha$) then $|v_\alpha(x)|\leq C(\alpha+\eps)$ on $\tilde{\Omega}\setminus \tilde{\Omega}_{\alpha+2\eps}$.
A similar argument with the minimum implies that the minimum of $v_\alpha$ is achieved on $\tilde{\Omega}\setminus \tilde{\Omega}_{\alpha+2\eps}$.
Hence $\|v_\alpha\|_{\Linfty(\Omega)}\leq C(\alpha+\eps)$ as required.
\end{proof}

\begin{lemma}
\label{lem:WellPosed:Subset:Hard:gamma}
Assume {\bf (A1,3)}. Then, there exists $C>0$ such that
\[ \la\sigma_2\l \tfrac{\delta_x}{\eps} \r\ra \leq C\l\gamma_\eps(x) + \eps\r \]
for all $x\in \Omega$.
\end{lemma}

\begin{proof}
Note that
\begin{align*}
\gamma_\eps(x) &= \frac{1}{\eps}\int_{\Omega\cap B(x,2\eps)}\eta_\eps\l |x-y|\r (x-y)\cdot \bm{n}(x) \, \dd y \\
&=-\int_{\eps^{-1}\l \Omega-x \r\cap B(0,2)}\eta(|z|) z\cdot \bm{n}(x)\, \dd z \\
&=\int_{B(0,2)\cap\{z_d \geq -\frac{\delta_x}{\eps}\}}\eta(|z|) z_d\, \dd z + O(\eps)  = \sigma_3\l\tfrac{\delta_x}{\eps}\r + O(\eps)
\end{align*}
where $\sigma_3(t) = \int_{B(0,2)\cap\{z_d \geq -t\}}\eta(|z|) z_d\, \dd z$.
We show that there exists $C>0$ such that $|\sigma_2(t)|\leq C \sigma_3(t)$, when combined with the above this completes the proof.

Assume $\eta(t) = 0$ for all $t>T$ and $\eta(t) > 0$ for $t<T$, from our assumptions $T\in [1,2]$.
Since $\sigma_3(t) = - \int_{B(0,2)\cap \{z_d\leq -t\}} \eta(|z|) z_d\,\dd z$, it is straightforward to show that $\sigma_3(t) = 0$ if and only if $t>T$.
It is also straightforward to check that $\sigma_2(t) = 0$ if $t\geq T$.
Hence, $\sigma_3(t) = 0$ implies $\sigma_2(t)=0$.
By L'H\^{o}pital's rule (and the Fundamental Theorem of Calculus),
\begin{align*}
\lim_{t\to T^-} \frac{|\sigma_2(t)|}{\sigma_3(t)} & = -\lim_{t\to T^-} \frac{\sigma_2(t)}{\sigma_3(t)} = -\lim_{t\to T^-} \frac{\dot{\sigma}_2(t)}{\dot{\sigma}_3(t)} \\
 & = \lim_{t\to T^-} \frac{\int_{[-2,2]^{d-1}} \eta\l \l |\tilde{z}|^2+t^2\r^\frac12\r (t^2-z_1^2) \, \dd \tilde{z}}{t\int_{[-2,2]^{d-1}} \eta\l \l |\tilde{z}|^2+t^2\r^\frac12\r \, \dd \tilde{z}} \\
 & = \lim_{t\to T^-} \l t - \frac{\int_{[-2,2]^{d-1}} \eta\l \l |\tilde{z}|^2+t^2\r^\frac12\r z_1^2 \, \dd \tilde{z}}{t\int_{[-2,2]^{d-1}} \eta\l \l |\tilde{z}|^2+t^2\r^\frac12\r \, \dd \tilde{z}} \r
\end{align*}
where we again adopt the notation $z=(\tilde{z},z_d)\in \bbR^{d-1}\times \bbR$.
We notice that $\eta\l \l |\tilde{z}|^2+t^2\r^\frac12\r = 0$ for all $|\tilde{z}|^2+t^2\geq T^2$, hence we can assume $z_1^2\leq T^2-t^2$.
Now we have,
\[ 0\leq \frac{\int_{[-2,2]^{d-1}} \eta\l \l |\tilde{z}|^2+t^2\r^\frac12\r z_1^2 \, \dd \tilde{z}}{t\int_{[-2,2]^{d-1}} \eta\l \l |\tilde{z}|^2+t^2\r^\frac12\r \, \dd \tilde{z}} \leq \frac{T^2-t^2}{t} \to 0. \]
Hence, $\lim_{t\to T^-} \frac{|\sigma_2(t)|}{\sigma_3(t)} = T$.
In particular, there exists $0<\xi<T$ such that for all $t\in [\xi,T]$ we have $\frac{|\sigma_2(t)|}{\sigma_3(t)}\leq 2T$.
And since $\frac{|\sigma_2(t)|}{\sigma_3(t)}$ is uniformly continuous on $[0,\xi]$ it is bounded, therefore there exists some $C\in [2T,+\infty)$ such that $\frac{|\sigma_2(t)|}{\sigma_3(t)}\leq C$.
This completes the proof.
\end{proof}

We now have the proof of Theorem \ref{thm:wellposed1}(i).
\begin{proof}[Proof of Theorem \ref{thm:wellposed1}(i)]
By Theorem \ref{thm:boundary}, Propositions \ref{prop:concentration1}, \ref{prop:concentration2}, and \ref{prop:asymp}, we have
\begin{equation}\label{eq:boundary}
\max_{x\in \Omega_n\cap \wO_\alpha} |u_n(x) - g(x)| \leq \frac{C\eps}{\sqrt{\beta}}\log\l\frac{\sqrt{\beta}}{\eps}\r,
\end{equation}
with probability at least $1-Cn\exp\left( -cn\beta\eps^d \right)$, where $\alpha:=K\beta^{-\frac12}\eps\log(\sqrt{\beta}\eps^{-1})$ and $\wO_\alpha$ is defined in \eqref{eq:enlarged}. We define
\[v_\alpha(x)=\begin{cases}
w_\alpha(x)+M\vartheta\varphi_\alpha(x),&\text{if }x\in \Omega\setminus \wO_{\alpha+2\eps}\\
g(x),&\text{otherwise,}\end{cases}\]
where $w_\alpha$ and $\varphi_\alpha$ are defined in Lemma~\ref{lem:barrier}, and $M$ will be chosen shortly. By  Lemma~\ref{lem:barrier} with probability at least $1-Cn\exp(-cn\eps^{d+2}\vartheta^2)$ we have
\begin{align*}
\L_{n,\eps}v_\alpha(x) & \geq \l \frac{2M\rho(x)\vartheta}{\eps} - C(1+M\vartheta) \r \gamma_\eps(x) + \l \frac{M\sigma_1\l\frac{\delta_{x}}{\eps}\r}{\sigma_\eta} - C(1+M\vartheta) \r \vartheta \\
 & \geq \l 2M\rho(x) - C(1+M\vartheta) \r \gamma_\eps(x) + \l \frac{M}{2} - C(1+M\vartheta) \r \vartheta
\end{align*}
for all $x\in \Omega_n\setminus \wO_\alpha$,  since $\sigma_1\l\frac{\delta_{x}}{\eps}\r\geq \frac{\sigma_\eta}{2}$. There exists $c>0$ such that for $M$ sufficiently large, and all $\eps<\vartheta<c$ we have $\L_{n,\eps}v_\alpha(x) > 0$ for all $x\in \Omega_n\setminus \wO_\alpha$. 
   
Since $\L_{n,\eps}u_n(x)=0$ for $x\in \Omega_n\setminus \Gamma_n$, we have that $\L_{n,\eps}(u_n - v_\alpha)(x) < 0$ for all $x\in \Omega_n\setminus (\Gamma_n\cup \wO_\alpha)$. If $u_n - v_\alpha$ attained its maximum value over $\Omega_n$ at any point $x\in \Omega_n\setminus (\Gamma_n\cup \wO_\alpha)$, then we would have $\L_{n,\eps}(u_n - v_\alpha)(x) \geq 0$, which is a contradicition. Therefore $u_n-v_\alpha$ attains its maximum value over $\Omega_n$ at some $x_*\in \Gamma_n\cup \wO_\alpha$. If $x_*\in \Gamma_n\subset \wO$ and $x_*\not\in \wO_{\alpha+2\eps}$, then 
\[u_n(x_*)-v_\alpha(x_*) =g(x_*) - w_\alpha(x_*) - M\vartheta\varphi_\alpha(x_*) \leq C(\alpha  +\vartheta),\]
due to Lemma \ref{lem:barrier} and the definition of $w_\alpha$. 
If $x_*\in \Gamma_n\cap \wO_{\alpha+2\eps}$, then $u_n(x_*)-v_\alpha(x_*) = 0$.
If $x_*\in \Omega_n\cap \wO_\alpha$ then by \eqref{eq:boundary} we have
\[u_n(x_*)-v_\alpha(x_*) \leq g(x_*) + \frac{C\eps}{\sqrt{\beta}}\log\l\frac{\sqrt{\beta}}{\eps}\r - v_\alpha(x). \]
Now,
\[ g(x_*) - v_\alpha(x_*) = \lb \begin{array}{ll} g(x_*) - w_\alpha(x_*) - M\vartheta\varphi_\alpha(x_*) & \text{if } x_*\in \wO_\alpha\setminus \wO_{\alpha+2\eps} \\ 0 & \text{if } x_*\in \wO_{\alpha+2\eps} \end{array} \rd \leq C(\vartheta +\eps). \]
It follows that
\[\max_{x\in\Omega_n}(u_n(x)-v_\alpha(x))\leq  C(\alpha  +\vartheta)\]
with probability at least $1-Cn\exp(-cn\eps^{d+2}\vartheta^2) - Cn\exp\left( -cn\beta\eps^d \right)$. Invoking Lemma \ref{lem:barrier} again, we have
\[|v_\alpha - u| \leq |w_\alpha - u| + M\vartheta |\varphi_\alpha| \leq C(\alpha + \vartheta),\]
and so
\[\max_{x\in\Omega_n}(u_n(x)-u(x))\leq  C(\alpha  +\vartheta).\]
We obtain the opposite inequality similarly, and complete the proof by noting that $\beta\geq \eps^2$ $\vartheta\leq c$ implies that $n\beta \eps^d \geq n\eps^{d+2}\vartheta^2$.
\end{proof}

To extend the result to the regime $\eps_n\leq  \l\frac{\log n}{n}\r^{\frac{1}{d+2}}$ and therefore beyond the regime where the Laplacian is pointwise consistent we will rely on $\Gamma$-convergence and $\TL^2$ compactness.
We start with $\TL^2$ compactness of minimizers.

\begin{proposition}
\label{prop:WellPosed:Subset:Hard:Compact}
Under the assumptions of Theorem~\ref{thm:wellposed1}(ii) we have that, with probability one, $\{u_n\}_{n\in\bbN}$ is precompact in $\TL^2$ and furthermore, if $u$ is any cluster point of $\{u_n\}_{n\in\bbN}$, then $u=g$ on $\tilde{\Omega}$.
\end{proposition}

\begin{proof}
It is straightforward to show that $\sup_{n\in\bbN} \|u_n\|_{\Ltwo(\mu_n)}<+\infty$ and $\sup_{n\in \bbN} \cE_n^{(2)}(u_n)<+\infty$ therefore by Proposition~\ref{prop:Prelim:TLp:GamConvEn} we have that $\{u_n\}_{n\in \bbN}$ is pre-compact in $\TL^2$.
Assume that $u_n\to u$ in $\TL^2$.
Choose $\alpha_n = \frac{C\eps_n}{\sqrt{\beta_n}}\log\l\frac{\sqrt{\beta_n}}{\eps_n}\r$.
By Theorem~\ref{thm:boundary} and the Borel-Cantelli lemma,
\[ \|u_n-g\|_{\Linfty(\Omega_n\cap\tilde{\Omega}_{\alpha_n})} = o(1) \]
with probability one.

Let $\tilde{\Omega}_n = \tilde{\Omega}_{\alpha_n+\|T_n-\Id\|_{\Linfty}}$ where $T_n$ is a transport map between $\mu_n$ and $\mu$ with $\|T_n-\Id\|_{\Linfty}\leq \gamma_n\to 0$.
Note that $\tilde{\Omega}_n \subseteq \tilde{\Omega}_{\alpha_n} \subseteq \tilde{\Omega}$ and $T_n(\tilde{\Omega}_n) \subseteq \tilde{\Omega}_{\alpha_n}$.
Hence,
\begin{align*}
\| u-g\|_{\Ltwo(\tilde{\Omega}_n)} & \leq \|u - u_n\circ T_n\|_{\Ltwo(\tilde{\Omega}_n)} + \|u_n\circ T_n - g\circ T_n\|_{\Ltwo(\tilde{\Omega}_n)} + \| g\circ T_n - g \|_{\Ltwo(\tilde{\Omega}_n)} \\
 & \leq \| u-u_n\circ T_n\|_{\Ltwo(\Omega)} + \l \|u_n - g\|_{\Linfty(\tilde{\Omega}_{\alpha_n}\cap \Omega_n)}  + \Lip(g) \gamma_n \r \Vol(\tilde{\Omega}).
\end{align*}
Taking the limit as $n\to \infty$ in the above we have $\| u-g\|_{\Ltwo(\tilde{\Omega})} = 0$ and hence $u=g$ on $\tilde{\Omega}$.
\end{proof}

We now prove the existence of a recovery sequence for the constrained energy $\cEpncon$ to $\cEpinftycon$ defined by
\begin{align*}
\cEpncon(w_n) & = \lb \begin{array}{ll} \cEpn(w_n) & \text{if } w_n(x_i) = g(x_i) \,\forall x_i\in \Gamma_n \\ +\infty & \text{else.} \end{array} \rd \\
\cEpinftycon(w) & = \lb \begin{array}{ll} \cEpinfty(w) & \text{if } w(x) = g(x) \,\forall x\in \wO \\ +\infty & \text{else.} \end{array} \rd
\end{align*}
where $\cEpn$ and $\cEpinfty$ are defined by~\eqref{eq:MainRes:illposed:Lp} and~\eqref{eq:MainRes:illposed:cEpinfty} respectively.

\begin{lemma}
\label{lem:WellPosed:Subset:Hard:RecoverySeq}
Under the assumptions of Theorem~\ref{thm:wellposed2}(ii) we have, with probability one, that for any $v\in \Lp(\mu)$ there exists a sequence $v_n\in \Ltwo(\mu_n)$ such that $v_n\to v$ in $\TLp$ and
\[ \limsup_{n\to \infty} \cEpncon(v_n) \leq \cEpinftycon(v). \] 
\end{lemma}

\begin{proof}
First consider any $v\in \Lip(\Omega)$ with $\cEpinftycon(v)<+\infty$.
Define $v_n\in \Lp(\mu_n)$ by
\[ v_n(x) = v(x) \quad \text{for } x\in \Omega_n. \]
It is straightforward to show that $v_n\to v$ in $\TLp$ and
\[ \limsup_{n\to\infty} \cEpncon(v_n) = \limsup_{n\to\infty} \cEpn(v_n) \leq \cEpinfty(v) = \cEpinftycon(v). \]
By the density of Lipschitz functions in $\Lp$ we can extend the result to any $v\in \Lp$ via a diagonalisation argument.
\end{proof}

We now prove Theorem~\ref{thm:wellposed1}(ii).

\begin{proof}[Proof of Theorem~\ref{thm:wellposed1}(ii)]
We start by showing that the liminf inequality holds along the minimising sequence.
We know by Proposition~\ref{prop:WellPosed:Subset:Hard:Compact} that the set $\{u_n\}_{n\in \bbN}$ of minimizers of $\cEncon^{(2)}$ is, with probability one, pre-compact and any limit $u$ satisfies $u=g$ on $\tilde{\Omega}$.
Assume $u_{n_k}\to u$ in $\TL^2$.
Then,
\[ \liminf_{k\to\infty} \cE^{(2)}_{n_k,\mathrm{con}}(u_{n_k}) = \liminf_{k\to\infty} \cE_{n_k}^{(2)}(u_{n_k}) \geq \cE_\infty^{(2)}(u) = \cEinftycon^{(2)}(u) \]
by Proposition~\ref{prop:Prelim:TLp:GamConvEn}.
Now let $v\in \Ltwo(\Omega)$ be any other function in $\Ltwo(\Omega)$ and $v_n$ it's recovery sequence (which exists by Lemma~\ref{lem:WellPosed:Subset:Hard:RecoverySeq}).
We have,
\begin{align*}
\cEinftycon^{(2)}(v) & \geq \limsup_{n\to\infty} \cEncon^{(2)}(v_n) \\
 & \geq \limsup_{k\to\infty} \cE^{(2)}_{n_k,\mathrm{con}}(v_{n_k}) \\
 & \geq \limsup_{k\to\infty} \cE^{(2)}_{n_k,\mathrm{con}}(u_{n_k}) \\
 & \geq \liminf_{k\to\infty} \cE^{(2)}_{n_k,\mathrm{con}}(u_{n_k}) \\
 & \geq \cEinftycon^{(2)}(u).
\end{align*}
By the above argument $\cEinftycon^{(2)}(v)\geq \cEinftycon^{(2)}(u)$ for all $v\in \Lp(\Omega)$, hence $u$ is a minimizer of $\cEinftycon^{(2)}$.
Since the minimizer of $\cEinftycon^{(2)}$ is the unique function satisfying~\eqref{eq:model1pde} it follows that the whole sequence converges.

To obtain $\Ltwo$ convergence we note
\[ \|u_n - u\lfloor_{\Omega_n}\|_{\Ltwo(\mu_n)} = \|u_n\circ T_n - u\lfloor_{\Omega_n}\circ T_n \|_{\Ltwo(\mu)} \leq \| u_n\circ T_n - u\|_{\Ltwo(\mu)} + \| u -u\lfloor_{\Omega_n}\circ T_n\|_{\Ltwo(\mu)}. \]
Since $u$ is Lipschitz and $T_n$ can be chosen to converge to the identity uniformly then we are done.
\end{proof}

\section{The Ill-Posed Case} \label{sec:IllPosed}

The aim of this section is to prove Theorems~\ref{thm:MainRes:illposed:illposed1} and~\ref{thm:MainRes:illposed:illposed2}.
Approximately the theorems state that if the number of labeled data points is less than $n\eps_n^p$, and $\eps_n$ is sufficiently large, then solutions $u_n$ of~\eqref{eq:BVP} are converging to constants.
We note that if $p\leq d$ then "sufficiently large" is satisfied by $\eps_n$ being greater than connectivity of the graph.

In this section we use variational methods, rather than the relationship between minimizers and random walks which is valid only for $p=2$.
In particular, we are able to treat any $p>1$ and we are able to prove the results for each model in a unified framework.
We therefore assume that for some subset $\Omega_n^\prime$ we have that if $x_i\in\Omega_n^\prime$ then $x_i\in\Gamma_n$ with probability $\beta$.
In {\bf Model 1} $\Omega_n^\prime = \wO$ and in {\bf Model 2} $\Omega_n^\prime = \partial_\delta\Omega$. 

We start by proving the $\Gamma$-convergence of the constrained functionals, this is the corresponding result to Proposition~\ref{prop:Prelim:TLp:GamConvEn} which concerns the unconstrained case.
The proof is analogous to the case when the constraint set $Z_n$ is fixed (in particular when $Z_n=\{1,2,\dots,N\}$ for $N$ fixed), see~\cite{slepcev19}.

\begin{proposition}
\label{prop:IllPosed:Hard:GamConv}
Let $p>1$, $\Omega_n^\prime\subset\Omega$ is open and each $x_i\in\Omega_n$ is in $\Gamma_n$ with probability $\beta\one_{\Omega_n^\prime}(x_i)$.
Assume {\bf (A1-3)} hold and $\delta_n=\int_{\Omega_n^\prime}\rho(x)\,\dd x$.
Further, we assume that $\beta_n\to 0^+$, $\eps_n\to 0^+$ and $\delta_n$ satisfy
\[ \frac{\beta_n\delta_n}{\eps_n^p} \ll 1, \quad \frac{n\eps_n^p}{\log n} \gg 1\quad \text{and} \quad n\eps_n^d\gg \log(n). \]
Then, with probability one,
\[ \Glim_{n\to\infty} \cEpconnepsn = \cEpinfty \]
on the set $\{(\nu,f)\,:\, \nu\in \cP(\Omega), \|f\|_{\Linfty(\nu)}\leq \|g\|_{\Linfty(\Omega)}\}$.
Furthermore, if $\{u_n\}_{n=1}^\infty$ is a sequence satisfying $\sup_{n\in \bbN} \|u_n\|_{\Lp(\mu_n)}<\infty$ and $\sup_{n\in\bbN}\cEpconnepsn(u_n)<\infty$ then $\{u_n\}_{n=1}^\infty$ is pre-compact in $\TLp$.
\end{proposition}

\begin{proof}
We divide the proof into the liminf and limsup inequalities, and the compactness property.

\emph{Compactness:} If $\sup_{n\in\bbN} \cEpconnepsn(u_n)<\infty$ then we have $\cEpconnepsn(u_n) = \cEpnepsn(u_n)$ hence compactness follows from Proposition~\ref{prop:Prelim:TLp:GamConvEn}.

\emph{Liminf:} Let $u_n\to u$ in $\TLp$ and assume $\liminf_{n\to \infty} \cEpconnepsn(u_n) < \infty$ else the liminf inequality is trivial.
Hence, we may assume that $u_n$ satisfies the constraints.
By Proposition~\ref{prop:Prelim:TLp:GamConvEn},
\[ \liminf_{n\to\infty} \cEpconnepsn(u_n) = \liminf_{n\to\infty} \cEpnepsn(u_n) \geq \cEpinfty(u) \]
as required.

\emph{Limsup:} Pick $u\in \Lp(\Omega)$ and suppose $\cEpinfty(u)<\infty$ else the limsup inequality is trivial.
Let $u_n$ be a recovery sequence with respect to the $\Gamma$-convergence of $\cEpn$ in Proposition~\ref{prop:Prelim:TLp:GamConvEn}, i.e. (with probability one)
\[ u_n \to u \text{ in } \TLp \qquad \text{and} \qquad \limsup_{n\to\infty} \cEpn(u_n) \leq\cEpinfty(u). \]
Define
\[ \hat{u}_n(x_i) = \lb \begin{array}{ll} g(x_i) & \text{if } i\in Z_n \\ u_n(x_i) & \text{else.} \end{array} \rd \]
By construction $\hat{u}_n$ satisfies the constraints.

We claim that (on a set of probability one)
\begin{align}
& \hat{u}_n \to u \qquad \text{in } \TLp \label{eq:IllPosed:Hard:GamConvClaim1} \\
& \lim_{n\to\infty} \l \cEpconnepsn(\hat{u}_n) - \cEpn(u_n) \r = 0. \label{eq:IllPosed:Hard:GamConvClaim2}
\end{align}

We first show that there exists $N$ and $a_n\to0^+$ such that if $n\geq N$ then $|Z_n|\leq a_n n\eps_n^p$ where $N<+\infty$ with probability one.
Let $z_i$ be the iid random variables satisfying $z_i=1$ if $i\in Z_n$ and $z_i=0$ if $i\not\in Z_n$.
Note $\bbE|z_i|^k = \beta_n\delta_n$ for all $k>0$.
By Bernstein's inequality, for any $t>0$,
\[ \bbP\ls |Z_n| \geq t + n\beta_n\delta_n \rs = \bbP\ls \sum_{i=1}^n (z_i - \beta_n\delta_n) \geq t \rs \leq e^{-\frac{\frac12 t^2}{n\beta_n\delta_n+\frac13 t}}. \]
Let us find $a_n$ such that
\[ 1\gg a_n\gg \max\lb \frac{\beta_n\delta_n}{\eps_n^p}, \frac{\log n}{n\eps_n^p} \rb. \]
In particular, for $n$ sufficiently large we can assume $a_n>\frac{2\beta_n\delta_n}{\eps_n^p}$ then, choosing $t = \frac12 n a_n\eps_n^p$, we have
\[ \bbP\ls \frac{|Z_n|}{n\eps_n^p} \geq a_n \rs \leq e^{-\frac{3}{16} n a_n\eps_n^p}. \]
Moreover, for $n$ sufficiently large $\frac{3}{16} n a_n\eps_n^p\geq 2 \log(n)$, hence
\[ \bbP\ls \frac{|Z_n|}{n\eps_n^p} \geq a_n \rs \leq n^{-2} \]
which is summable, i.e. $\sum_{n=N}^\infty \bbP\ls \frac{|Z_n|}{n\eps_n^p} \geq a_n \rs \leq \sum_{n=N}^\infty n^{-2}<+\infty$ so by the Borel Cantelli lemma the event $\frac{|Z_n|}{n\eps_n^p}\geq a_n$ occurs a finite number of times with probability one.
In particular, there exists $N<+\infty$ such that $\frac{|Z_n|}{n\eps_n^p}\leq a_n$ for all $n\geq N$.

To show~\eqref{eq:IllPosed:Hard:GamConvClaim1} it is enough to show that $\|\hat{u}_n-u_n\|_{\Lp(\mu_n)}\to 0$.
Now,
\[ \|\hat{u}_n-u_n\|_{\Lp(\mu_n)}^p = \frac{1}{n} \sum_{i=1}^n |\hat{u}_n(x_i) - u_n(x_i)|^p \leq \frac{(2\|g\|_{\Linfty(\Omega)})^p|Z_n|}{n}. \] Since, with probability one, $|Z_n|\leq a_n n\eps_n^p$ for $n$ sufficiently large then we have $\|\hat{u}_n-u_n\|_{\Lp(\mu_n)}\to 0$.

For~\eqref{eq:IllPosed:Hard:GamConvClaim2} we compute
\begin{align*}
\la \cEpconnepsn(\hat{u}_n) - \cEpnepsn(u_n) \ra & \leq \frac{2}{n^2\eps_n^p} \sum_{i\in Z_n} \sum_{j\not\in Z_n} W_{ij} \la \la g(x_i) - u_n(x_j)\ra^p - \la u_n(x_i) - u_n(x_j)\ra^p \ra \\
 & \hspace{1cm} + \frac{1}{n^2\eps_n^p} \sum_{i\in Z_n} \sum_{j\in Z_n} W_{ij} \la \la g(x_i) - g(x_j)\ra^p - \la u_n(x_i) - u_n(x_j)\ra^p \ra \\
 & \leq \frac{4(2\|g\|_{\Linfty(\Omega)})^p}{n^2\eps_n^p} \sum_{i\in Z_n} \sum_{j=1}^n W_{ij}.
\end{align*}
By the argument in \cite[Lemma 4.9]{slepcev19} we can find a $C$ such that
\[ \frac{1}{n} \sum_{j=1}^n W_{ij} \leq C \]
for all $n$ sufficiently large (with probability one).
Hence, (after redefining $C$)
\[ \la \cEpconnepsn(\hat{u}_n) - \cEpn(u_n) \ra \leq \frac{C|Z_n|}{n\eps_n^p}. \]
Since $\frac{|Z_n|}{n\eps_n^p} \leq a_n\to 0$ this proves~\eqref{eq:IllPosed:Hard:GamConvClaim2} with probability one.
\end{proof}

By the above lemma we can now prove Theorems~\ref{thm:MainRes:illposed:illposed1} and~\ref{thm:MainRes:illposed:illposed2}.

\begin{proof}[Proof of Theorems~\ref{thm:MainRes:illposed:illposed1} and~\ref{thm:MainRes:illposed:illposed2}]
We apply Proposition~\ref{prop:IllPosed:Hard:GamConv}; for {\bf Model 1} we have $\delta_n \sim 1$, for {\bf Model 2} we have $\delta_n\to 0^+$.
In all cases the assumptions in Theorems~\ref{thm:MainRes:illposed:illposed1} and~\ref{thm:MainRes:illposed:illposed2} imply that $\frac{\beta_n\gamma_n}{\eps_n^p}\to 0^+$.
Hence, the following statements all hold with probability one.
As in the proof of \cite[Theorem~2.1]{slepcev19} we can argue that minimizers $u_n$ of $\cEpncon(\cdot;Z_n)$ satisfy $\|u_n\|_{\Linfty(\mu_n)}\leq \|g\|_{\Linfty}$ for all $n$ sufficiently large, so by Proposition~\ref{prop:IllPosed:Hard:GamConv} we can infer compactness of $u_n$.
Since $\Gamma$-convergence (also given by Proposition~\ref{prop:IllPosed:Hard:GamConv}) plus compactness implies the convergence of minimizers (along subsequences), see Theorem~\ref{thm:Prelim:TLp:Conmin}, then we have that any converging subsequence of $u_n$ is converging to a minimizer of $\cEpinfty$.
Since the minimizers of $\cEpinfty$ are the constant functions then we are done.
\end{proof}

\section{Numerical Experiments} \label{sec:NumExp}

We now present the results of numerical experiments to support the convergence rates established in Theorems~\ref{thm:wellposed1} and \ref{thm:wellposed2}.

\subsection{Synthetic data}
\label{sec:synthetic}

\begin{figure}
\centering
\subfloat[{\bf Model 1}]{\includegraphics[width=0.5\textwidth,clip=true]{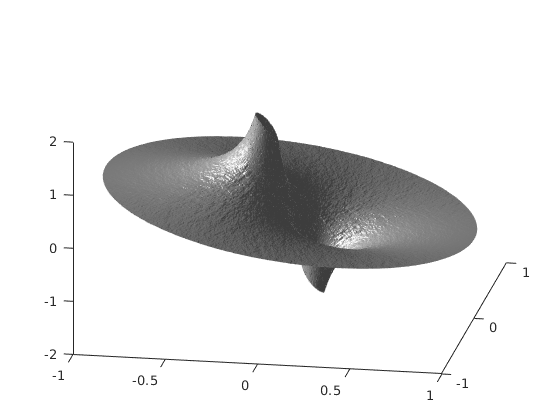}}
\subfloat[{\bf Model 2}]{\includegraphics[width=0.5\textwidth,clip=true]{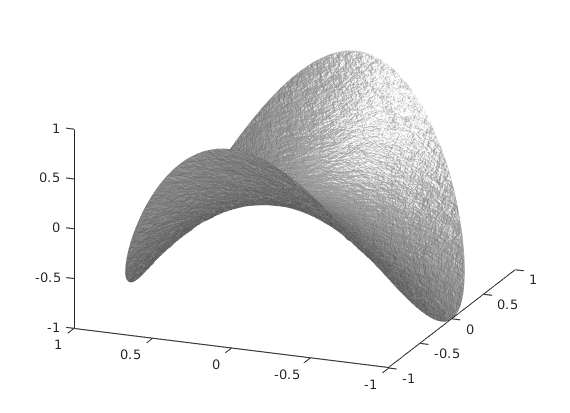}}
\caption{Example of the solution $u_n$ for $d=2$ for each model.}
\label{fig:models}
\end{figure}

We first consider synthetic experiments with {\bf Model 1} and {\bf Model 2}. We take our domain to be the unit ball $\Omega = B(0,1)$ in $\R^d$, for $d=2,3$. For {\bf Model 1} we take the label domain  $\wO$ to be the ball $\wO=B(0,\tfrac12)$ and examine how the rate changes as $\beta$ is varied.
For {\bf Model 2} the label domain is the set $\partial_\delta\Omega = B(0,1)\setminus B(0,1-\delta)$, here we fix $\beta=1$ and examine how the rate changes as $\delta$ is varied.
 
For each model, we choose the label function $g:\Omega\to \R$ to be an explicit solution of the corresponding boundary value problem (\eqref{eq:model1pde} or \eqref{eq:model2pde}). For {\bf Model 1} with $d=2$ we choose
\[g(x) = \log |x - z^*||z| - \log|x+z^*||z| + \log|x-z| - \log|x+z|,\]
where $z^* = z/|z|^2$, which is the solution of the two-point source problem
\[\Delta g = 2\pi(\delta_z - \delta_{-z})\]
with Neumann condition on $\partial \Omega$. We choose $z = (1/8,0)$ to ensure the singularities are strictly inside the label domain, and we truncate the function $g$ near the singularity to ensure it is Lipschitz. The corresponding function for $d=3$ is given by
\begin{align*}
g(x) &= \frac{1}{|z||x-z^*|} - \frac{1}{|z||x+z^*|} + \frac{1}{|x-z|}-\frac{1}{|x+z|}\\
&\hspace{1in}+ \log\left( \frac{z}{|z|}(z^*-x) + |z^*-x| \right) - \log\left( \frac{z}{|z|}(z^*+x) + |z^*+x| \right).
\end{align*}
and it solves the two-point problem $-\Delta g = 4\pi (\delta_z - \delta_{-z})$. For {\bf Model 2} we use the harmonic label function
\[g(x) = \sum_{i=1}^d (-1)^{i-1}(x \cdot \e_i)^2 - \frac{1}{2d}( (-1)^{d-1} +1)|x|^2\]
where $\e_i\in\bbR^d$ are the standard basis vectors, i.e., when $d=2$, $\e_1=(1,0)$ and $\e_2=(0,1)$.

The graph is constructed from $n$ independent and uniformly distributed random variables $x_1,x_2,\dots,x_n$, and the kernel $\eta_\eps$ for defining the weights is a Gaussian kernel with standard deviation $\sigma = \eps/2$. We choose the length scale
\begin{equation}\label{eq:epsscale}
\eps = \left( \frac{\log(n)}{n} \right)^{\frac{1}{d+2}},
\end{equation}
which is the lower limit of the pointwise consistency results for graph Laplacians, and the lower limit allowable for the convergence rates in Theorems \ref{thm:wellposed1} and \ref{thm:wellposed2}. We remark that neither theorem establishes a convergence rate for this small value of $\eps$.  

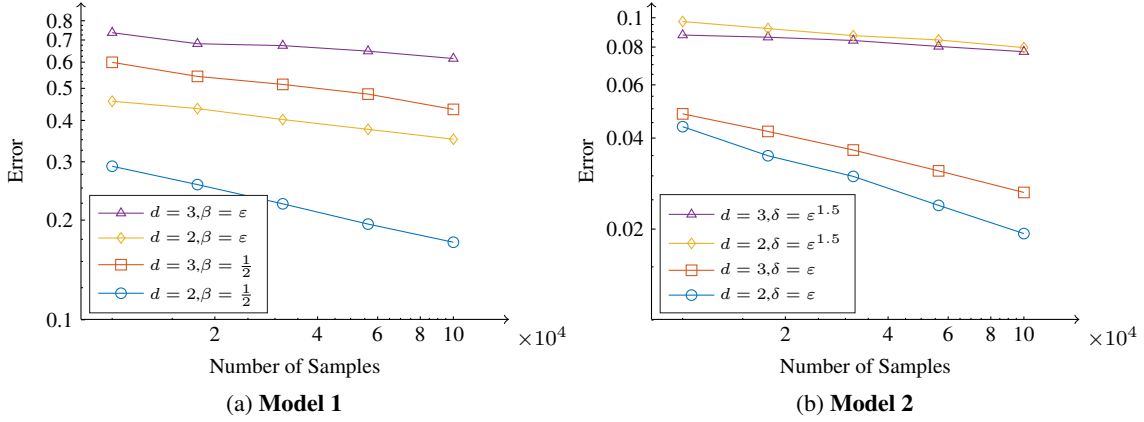
\begin{figure}
\setlength\figureheight{0.3\textwidth}
\setlength\figurewidth{0.45\textwidth}
\centering
\scriptsize
\subfloat[{\bf Model 1}]{% This file was created by matlab2tikz.
%
%The latest updates can be retrieved from
%  http://www.mathworks.com/matlabcentral/fileexchange/22022-matlab2tikz-matlab2tikz
%where you can also make suggestions and rate matlab2tikz.
%
\definecolor{mycolor1}{rgb}{0.49400,0.18400,0.55600}%
\definecolor{mycolor2}{rgb}{0.92900,0.69400,0.12500}%
\definecolor{mycolor3}{rgb}{0.85000,0.32500,0.09800}%
\definecolor{mycolor4}{rgb}{0.00000,0.44700,0.74100}%
\begin{tikzpicture}
\def\xl{9}
\def\xu{11.75}
\def\yl{-2.3}
\def\yu{-0.2}
\def\bigxticks{(\yu-\yl)/100}
\def\smallxticks{(\yu-\yl)/200}
\def\bigyticks{(\xu-\xl)/100}
\def\smallyticks{(\xu-\xl)/200}

\begin{axis}[%
width=\figurewidth,
height=\figureheight,
scale only axis,
xmin={\xl-(\xu-\xl)/5},
xmax={\xu+(\xu-\xl)/5},
ymin={\yl-(\yu-\yl)/5},
ymax={\yu+(\yu-\yl)*0.05},
hide axis,
axis background/.style={fill=white!100},
legend style={at={(axis cs: {\xl+2*\bigyticks},{\yl+2*\bigxticks})}, anchor=south west, legend cell align=left, align=left, draw=white!15!black, font=\tiny, text = black}
]

\draw [black,->] (axis cs: \xl,\yl) -- (axis cs: {\xu+(\xu-\xl)*0.05},\yl) node[below right] {$\times 10^4$};
\draw [black,->] (axis cs: \xl,\yl) -- (axis cs: \xl,{\yu+(\yu-\yl)*0.05});
\node[black] at (axis cs: {(\xu+(\xu-\xl)*0.05+\xl)/2},{\yl-0.16*(\yu-\yl)}) {Number of Samples};
\node[black,rotate=90] at (axis cs: {\xl-0.16*(\xu-\xl)},{(\yu+(\yu-\yl)*0.05+\yl)/2}) {Error};
%\foreach \yValue in {-2.2,-2,-1.8,-1.6,-1.4,-1.2,-1.0,-0.8,-0.6,-0.4,-0.2} {
%    \edef\temp{\noexpand\draw [-] ({\xl+(\xu-\xl)/100},\yValue) -- (\xl,\yValue) node[black,left] {\yValue};} 
%    \temp
%}
\foreach \yValue in {0.1,0.2,0.3,0.4,0.5,0.6,0.7,0.8} {
    \edef\temp{\noexpand\draw [black,-] ({\xl+\bigyticks},{ln(\yValue)}) -- (\xl,{ln(\yValue)}) node[black,left] {\yValue};} 
    \temp
}
\foreach \yValue in {0.125,0.15,0.175,0.225,0.25,0.275,0.325,0.35,0.375,0.425,0.45,0.475,0.525,0.55,0.575,0.625,0.65,0.675,0.725,0.75,0.775} {
    \edef\temp{\noexpand\draw [black,-] ({\xl+\smallyticks},{ln(\yValue)}) -- (\xl,{ln(\yValue)}) node[black,left] {};} 
    \temp
}
%\foreach \xValue in {9,9.5,10,10.5,11,11.5} {
%	\edef\temp{\noexpand\draw [-] (\xValue,{\yl+(\yu-\yl)/100}) -- (\xValue,\yl) node[black,below] {\xValue};}    
%    \temp
%}
\foreach \xValue in {2,4,6,8,10} {
	\edef\temp{\noexpand\draw [black,-] ({ln(\xValue*10^4)},{\yl+\bigxticks}) -- ({ln(\xValue*10^4)},\yl) node[black,below] {\xValue};}    
    \temp
}
\foreach \xValue in {1,1.5,2.5,3,3.5,4.5,5,5.5,6.5,7,7.5,8.5,9,9.5} {
	\edef\temp{\noexpand\draw [black,-] ({ln(\xValue*10^4)},{\yl+\smallxticks}) -- ({ln(\xValue*10^4)},\yl) node[black,below] {};}    
    \temp
}

\addplot [color=mycolor1, mark=triangle, mark options={solid, mycolor1}]
  table[row sep=crcr]{%
9.21034037197618	-0.306095126315034\\
9.78599822374366	-0.382971826890459\\
10.3616399829233	-0.395812287541014\\
10.9372768351602	-0.434365166655835\\
11.5129254649702	-0.48574776423695\\
};
\addlegendentry{$d=3$,$\beta=\eps$}

\addplot [color=mycolor2, mark=diamond, mark options={solid, mycolor2}]
  table[row sep=crcr]{%
9.21034037197618	-0.782728416069306\\
9.78599822374366	-0.833889488387306\\
10.3616399829233	-0.910086041710568\\
10.9372768351602	-0.978212435684259\\
11.5129254649702	-1.04681546600413\\
};
\addlegendentry{$d=2$,$\beta=\eps$}

\addplot [color=mycolor3, mark=square, mark options={solid, mycolor3}]
  table[row sep=crcr]{%
9.21034037197618	-0.512116845444259\\
9.78599822374366	-0.610102245680416\\
10.3616399829233	-0.665745979610753\\
10.9372768351602	-0.733090237089661\\
11.5129254649702	-0.839476593115873\\
};
\addlegendentry{$d=3$,$\beta=\frac12$}

\addplot [color=mycolor4, mark=o, mark options={solid, mycolor4}]
  table[row sep=crcr]{%
9.21034037197618	-1.23495323464603\\
9.78599822374366	-1.36243657115546\\
10.3616399829233	-1.49618262040931\\
10.9372768351602	-1.63720089576289\\
11.5129254649702	-1.7633349709977\\
};
\addlegendentry{$d=2$,$\beta=\frac12$}

\end{axis}
\end{tikzpicture}%
}
\subfloat[{\bf Model 2}]{% This file was created by matlab2tikz.
%
%The latest updates can be retrieved from
%  http://www.mathworks.com/matlabcentral/fileexchange/22022-matlab2tikz-matlab2tikz
%where you can also make suggestions and rate matlab2tikz.
%
\definecolor{mycolor1}{rgb}{0.49400,0.18400,0.55600}%
\definecolor{mycolor2}{rgb}{0.92900,0.69400,0.12500}%
\definecolor{mycolor3}{rgb}{0.85000,0.32500,0.09800}%
\definecolor{mycolor4}{rgb}{0.00000,0.44700,0.74100}%
\begin{tikzpicture}
\def\xl{9}
\def\xu{11.75}
\def\yl{-4.6}
\def\yu{-2.3}
\def\bigxticks{(\yu-\yl)/100}
\def\smallxticks{(\yu-\yl)/200}
\def\bigyticks{(\xu-\xl)/100}
\def\smallyticks{(\xu-\xl)/200}

\begin{axis}[%
width=\figurewidth,
height=\figureheight,
scale only axis,
xmin={\xl-(\xu-\xl)/5},
xmax={\xu+(\xu-\xl)/5},
ymin={\yl-(\yu-\yl)/5},
ymax={\yu+(\yu-\yl)*0.05},
hide axis,
axis background/.style={fill=white!100},
legend style={at={(axis cs: {\xl+2*\bigyticks},{\yl+2*\bigxticks})}, anchor=south west, legend cell align=left, align=left, draw=white!15!black, font=\tiny, text = black}
]

\draw [black,->] (axis cs: \xl,\yl) -- (axis cs: {\xu+(\xu-\xl)*0.05},\yl) node[below right] {$\times 10^4$};
\draw [black,->] (axis cs: \xl,\yl) -- (axis cs: \xl,{\yu+(\yu-\yl)*0.05});
\node[black] at (axis cs: {(\xu+(\xu-\xl)*0.05+\xl)/2},{\yl-0.16*(\yu-\yl)}) {Number of Samples};
\node[black,rotate=90] at (axis cs: {\xl-0.16*(\xu-\xl)},{(\yu+(\yu-\yl)*0.05+\yl)/2}) {Error};
\foreach \yValue in {0.02,0.04,0.06,0.08,0.1} {
    \edef\temp{\noexpand\draw [black,-] ({\xl+\bigyticks},{ln(\yValue)}) -- (\xl,{ln(\yValue)}) node[black,left] {\yValue};} 
    \temp
}
\foreach \yValue in {0.01,0.015,0.025,0.03,0.035,0.045,0.05,0.055,0.065,0.07,0.075,0.085,0.09,0.095} {
    \edef\temp{\noexpand\draw [black,-] ({\xl+\smallyticks},{ln(\yValue)}) -- (\xl,{ln(\yValue)}) node[black,left] {};} 
    \temp
}
\foreach \xValue in {2,4,6,8,10} {
	\edef\temp{\noexpand\draw [black,-] ({ln(\xValue*10^4)},{\yl+\bigxticks}) -- ({ln(\xValue*10^4)},\yl) node[black,below] {\xValue};}    
    \temp
}
\foreach \xValue in {1,1.5,2.5,3,3.5,4.5,5,5.5,6.5,7,7.5,8.5,9,9.5} {
	\edef\temp{\noexpand\draw [black,-] ({ln(\xValue*10^4)},{\yl+\smallxticks}) -- ({ln(\xValue*10^4)},\yl) node[black,below] {};}    
    \temp
}

\addplot [color=mycolor1, mark=triangle, mark options={solid, mycolor1}]
  table[row sep=crcr]{%
9.21034037197618	-2.4342863233776\\
9.78599822374366	-2.45062359136304\\
10.3616399829233	-2.47556016955313\\
10.9372768351602	-2.5208494975108\\
11.5129254649702	-2.56079142511633\\
};
\addlegendentry{$d=3$,$\delta=\eps^{1.5}$}

\addplot [color=mycolor2, mark=diamond, mark options={solid, mycolor2}]
  table[row sep=crcr]{%
9.21034037197618	-2.33132126878168\\
9.78599822374366	-2.3847350959704\\
10.3616399829233	-2.43839328508241\\
10.9372768351602	-2.47127925655626\\
11.5129254649702	-2.52990285391911\\
};
\addlegendentry{$d=2$,$\delta=\eps^{1.5}$}

\addplot [color=mycolor3, mark=square, mark options={solid, mycolor3}]
  table[row sep=crcr]{%
9.21034037197618	-3.03457059147863\\
9.78599822374366	-3.16829726837463\\
10.3616399829233	-3.31001362856298\\
10.9372768351602	-3.4686155310425\\
11.5129254649702	-3.63300845098558\\
};
\addlegendentry{$d=3$,$\delta=\eps$}

\addplot [color=mycolor4, mark=o, mark options={solid, mycolor4}]
  table[row sep=crcr]{%
9.21034037197618	-3.13208741296459\\
9.78599822374366	-3.35340798407757\\
10.3616399829233	-3.51084004902097\\
10.9372768351602	-3.73149250167759\\
11.5129254649702	-3.94624723915059\\
};
\addlegendentry{$d=2$,$\delta=\eps$}

\end{axis}
\end{tikzpicture}%
}
\caption{Plots of errors versus number of samples $n$ for each model.}
\label{fig:model_errs}
\end{figure}

\begin{table}[!t]
\centering
\begin{tabular}{|c|c|c|c|c|}
 \hline
 &\multicolumn{2}{c|}{\textbf{Model 1}}&\multicolumn{2}{c|}{{\bf Model 2}} \\
\hline
Dimension& $\beta=1/2$ & $\beta=\eps$ & $\beta=1,\delta =\eps$ & $\beta=1,\delta=\eps^{1.5}$ \\
\hline
$d=2$     &1.02 &0.52   &1.54   &0.37         \\
$d=3$     &0.75 &0.39   &1.44   &0.31       \\
\hline
\end{tabular}
\caption{Convergence rates $\alpha$ from fitting the error to a power law $\eps^{\alpha}$.}
\label{tab:rates}
\end{table}

Plots of the errors are shown in Figure \ref{fig:model_errs}.  For both models, we report the error
\[\max_{x\in \Omega_n}|u_n(x) - g(x)|,\]
averaged over 100 trials. For {\bf Model 1} we consider both $\beta=1/2$ and $\beta=\eps$, and for {\bf Model 2} we consider $\beta=1$, and both $\delta=\eps$ and $\delta=\eps^{1.5}$. We also ran experiments with $\beta=\eps^2$ and $\delta=\eps^2$ in each model, and observed non-convergence as expected from our ill-posedness results (Theorems \ref{thm:MainRes:illposed:illposed1} and \ref{thm:MainRes:illposed:illposed2}). Table \ref{tab:rates} shows the convergence rates $\alpha$ from fitting the error to a power law $\eps^\alpha$.
For the larger length scale $\eps = \left( \frac{\log(n)}{n} \right)^{\frac{1}{d+4}}$, we would expect to get rates of $\alpha=1$ and $\alpha=1/2$ for $\beta=1/2,\delta=\eps$, and $\beta=\eps,\delta=\eps^{1.5}$, respectively, due to the rates of convergence in Theorems \ref{thm:wellposed1} and \ref{thm:wellposed2}. Even though we have selected a much smaller length scale \eqref{eq:epsscale}, for which our theorems do not guarantee rates, we see that many of the rates still hold experimentally and agree with our the expected rates, or are even better. We expect there is some amount of stochastic homogenization taking place, which allows the rates to be pushed beyond the pointwise consistency regime of graph Laplacians. A careful analysis of this phenomenon is beyond the scope of our work.

\subsection{Comparison on MNIST dataset}
\label{sec:MNIST}

\begin{figure}
\centering
\includegraphics[width=0.8\textwidth,clip=true,trim = 150 250 120 200]{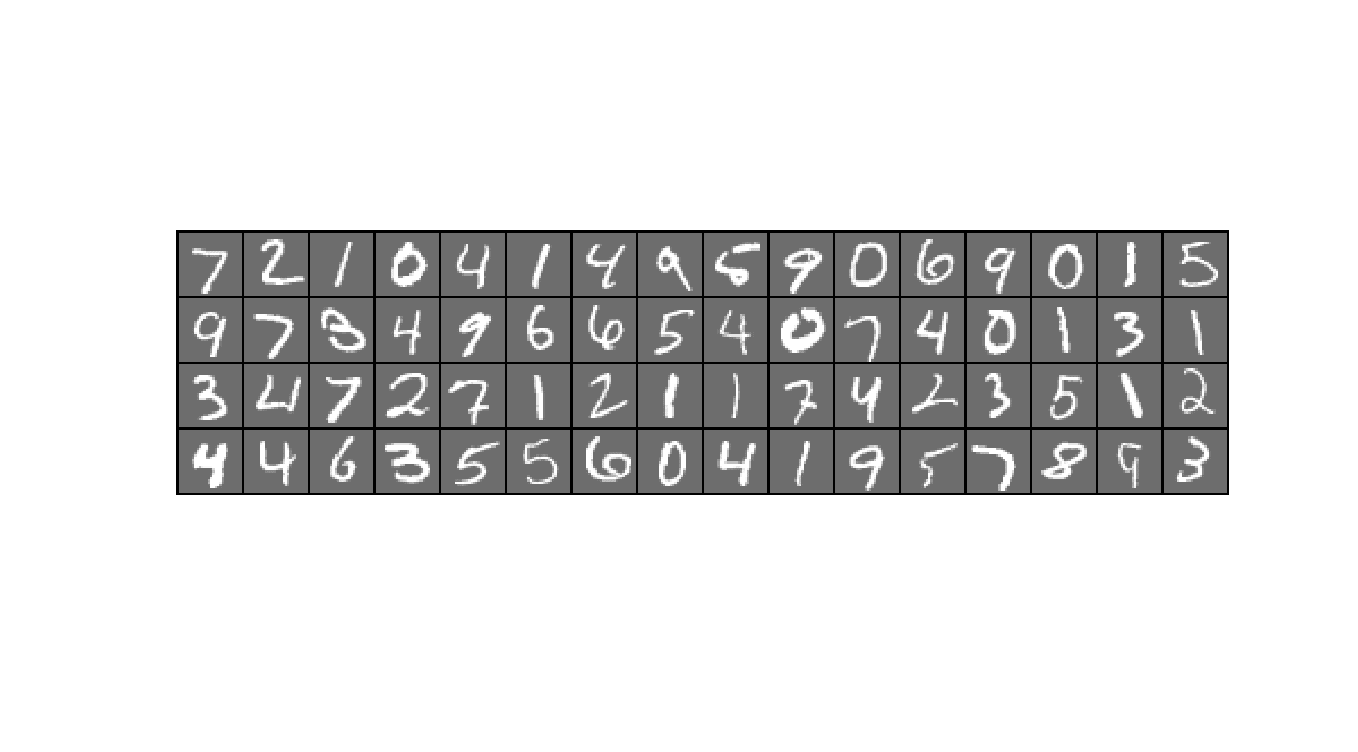}
\caption{Example of some of the handwritten digits from the MNIST dataset \cite{lecun1998gradient}.}
\label{fig:MNIST}
\end{figure}
We now consider an experiment with the MNIST dataset, which consists of 70,000 grayscale $28\times28$ pixel images of handwritten digits $0$--$9$ \cite{lecun1998gradient}. See Figure \ref{fig:MNIST} for some examples of MNIST digits. The experiment was conducted using the GraphLearning Python package and the code is available online~\cite{GraphLearning}. We construct the graph using all 70,000 MNIST images by connecting each data point to its nearest $k$ neighbors (in Euclidean distance) using Gaussian weights with variance $\sigma^2 = d_k^2/8$, where $d_k$ is the distance to the $k^{\rm th}$ nearest neighbor. The graph is then symmetrized by replacing the weight matrix $W$ with $\tfrac{1}{2}(W^T + W)$. Given some portion of the graph is labeled, the semi-supervised learning is conducted by a one-vs-rest approach, which computes solutions of  $10$ binary classifications problems, each  solving the Laplacian learning problem \eqref{eq:BVP}, and selects the most likely label for each digit.

For the experiment, we used 10 different labeling rates, labeling $m=1,2,4,8,16,32,64,128,256,512$ images \emph{per class}, which correspond to labeling rates of $0.014\%$ up to $7.3\%$. For each label rate, we ran 100 trials randomly selecting which images to provide as labels, and averaged the test error over all the trials. Figure  \ref{fig:MNISTerror} shows the error plots for this experiment. We ran the experiments on graphs with different connectivity length scales, choosing $k=10,20,30$ nearest neighbors. While our main Theorems (Theorems \ref{thm:wellposed1} and \ref{thm:wellposed2}) hold for $\eps$-graph constructions, we expect the results will hold for symmetrized $k$-NN graphs as well, using recently established pointwise consistency results with linear rates for $k$-NN graph Laplacians \cite{calder2019improved}.
\begin{figure}
\setlength\figureheight{0.3\textwidth}
\setlength\figurewidth{0.45\textwidth}
\centering
\scriptsize
% This file was created by matlab2tikz.
%
%The latest updates can be retrieved from
%  http://www.mathworks.com/matlabcentral/fileexchange/22022-matlab2tikz-matlab2tikz
%where you can also make suggestions and rate matlab2tikz.
%
\begin{tikzpicture}
\def\xl{-0.1}
\def\xu{6.3}
\def\yl{1.5}
\def\yu{4.5}
\def\bigxticks{(\yu-\yl)/100}
\def\smallxticks{(\yu-\yl)/200}
\def\bigyticks{(\xu-\xl)/100}
\def\smallyticks{(\xu-\xl)/200}

\begin{axis}[%
width=\figurewidth,
height=\figureheight,
scale only axis,
xmin={\xl-(\xu-\xl)/5},
xmax={\xu+(\xu-\xl)/5},
ymin={\yl-(\yu-\yl)/5},
ymax={\yu+(\yu-\yl)*0.05},
hide axis,
axis background/.style={fill=white!100},
legend style={at={(axis cs: {\xl+2*\bigyticks},{\yl+2*\bigxticks})}, anchor=south west, legend cell align=left, align=left, draw=white!15!black, font=\tiny, text = black}
]

\draw [black,->] (axis cs: \xl,\yl) -- (axis cs: {\xu+(\xu-\xl)*0.05},\yl) node[below right] {};
\draw [black,->] (axis cs: \xl,\yl) -- (axis cs: \xl,{\yu+(\yu-\yl)*0.05});
\node[black] at (axis cs: {(\xu+(\xu-\xl)*0.05+\xl)/2},{\yl-0.16*(\yu-\yl)}) {Number of Labels per Class};
\node[black,rotate=90] at (axis cs: {\xl-0.16*(\xu-\xl)},{(\yu+(\yu-\yl)*0.05+\yl)/2}) {Test Error \%};
\foreach \yValue in {10,30,50,70,90} {
    \edef\temp{\noexpand\draw [black,-] ({\xl+\bigyticks},{ln(\yValue)}) -- (\xl,{ln(\yValue)}) node[black,left] {\yValue};} 
    \temp
}
\foreach \yValue in {5,15,20,25,35,40,45,55,60,65,75,80,85} {
    \edef\temp{\noexpand\draw [black,-] ({\xl+\smallyticks},{ln(\yValue)}) -- (\xl,{ln(\yValue)}) node[black,left] {};} 
    \temp
}
\foreach \xValue in {1,2,4,8,16,32,64,128,256,512} {
	\edef\temp{\noexpand\draw [black,-] ({ln(\xValue)},{\yl+\bigxticks}) -- ({ln(\xValue)},\yl) node[black,below] {\xValue};}    
    \temp
}
\foreach \xValue in {} {
	\edef\temp{\noexpand\draw [black,-] ({ln(\xValue*10^4)},{\yl+\smallxticks}) -- ({ln(\xValue*10^4)},\yl) node[black,below] {};}    
    \temp
}

%\foreach \yValue in {1.5,2,2.5,3,3.5,4,4.5} {
%    \edef\temp{\noexpand\draw [-] ({\xl+(\xu-\xl)/100},\yValue) -- (\xl,\yValue) node[black,left] {\yValue};} 
%    \temp
%}
%\foreach \xValue in {2,3,4,5,6,7,8,9} {
%	\edef\temp{\noexpand\draw [-] (\xValue,{\yl+(\yu-\yl)/100}) -- (\xValue,\yl) node[black,below] {\xValue};}    
%    \temp
%}

\addplot [color=blue, mark=triangle, mark options={solid, blue}]
  table[row sep=crcr]{%
0	4.42858454731795\\
0.6931	4.31110646070351\\
1.3863	3.88645490971424\\
2.0794	3.0270080475575\\
2.7726	2.22353106308421\\
3.4657	1.9503573882683\\
4.1589	1.79443920884981\\
4.8520	1.67976589226713\\
5.5452	1.59143686132455\\
6.2383	1.51076608863073\\
};
\addlegendentry{k=10}

\addplot [color=red, mark=o, mark options={solid, red}]
  table[row sep=crcr]{%
0	4.45172992042352\\
0.6931	4.34045175263827\\
1.3863	3.99395855924528\\
2.0794	3.2703253070578\\
2.7726	2.46381919745752\\
3.4657	2.14412848930165\\
4.1589	1.97341413703431\\
4.8520	1.84827981964142\\
5.5452	1.75426526480048\\
6.2383	1.67094684901131\\
};
\addlegendentry{k=20}

\addplot [color=green, mark=diamond, mark options={solid, green}]
  table[row sep=crcr]{%
0	4.45595994871862\\
0.6931	4.35863261604527\\
1.3863	4.0701748129149\\
2.0794	3.40778506283046\\
2.7726	2.61497456012206\\
3.4657	2.26425368085537\\
4.1589	2.08536148419083\\
4.8520	1.95532285309342\\
5.5452	1.85775006060117\\
6.2383	1.77272954332125\\
};
\addlegendentry{k=30}

\end{axis}
\end{tikzpicture}%
\caption{Error plots for MNIST experiment showing testing error versus number of labels, averaged over 100 trials.}
\label{fig:MNISTerror}
\end{figure}
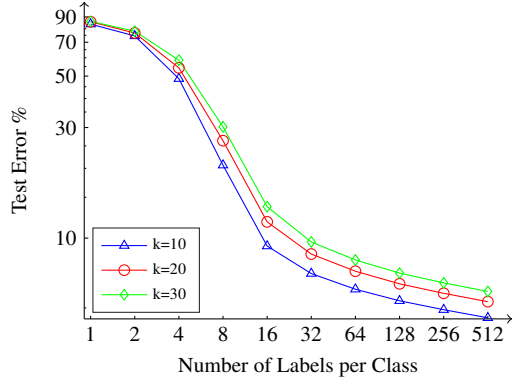

There are two important observations to make about the error plots in Figure \ref{fig:MNISTerror}. First, we see that the error decreases as the length scale of the graph decreases, as expected in the convergence rates in Theorems \ref{thm:wellposed1} and \ref{thm:wellposed2}. Second, if we fit the error to a power law of the form
\[\text{Error }\sim m^{-\alpha},\]
where $m$ is the number of labeled data points, we find that $\alpha=0.53$ for $k=10$, $\alpha=0.51$ for $k=20$, and $\alpha=0.49$ for $k=30$. Since $m$ is proportional to the label rate $\beta$ (more precisely, $\beta = m/n$), this aligns very closely with the dependence on $\beta$ in the convergence rate in Theorem \ref{thm:wellposed1}. However, we note in Figure \ref{fig:MNISTerror} that the rate $\alpha$ is not constant, and there appear to be several regimes. The low label rate regime of $m=1,2$ has a rate of roughly $\alpha=0.17$, while the moderate regime from $m=4$ to $m=16$ has a rate of approximately $\alpha=1$, and the high label rate regime from $m=64$ to $m=512$ has a rate of approximately $\alpha=0.15$.

\section{Conclusions} \label{sec:Conc}

Our results (stated for the hard constraint model) showed that: (1) when $p=2$ and $\eps_n\gg \l\frac{\log n}{n}\r^{\frac{1}{d+4}}$ that
\begin{enumerate}
\item if $\beta_n\gg \eps_n^2$ the limit is well-posed, and
\item if $\beta_n\ll \eps_n^2$ the limit is ill-posed;
\end{enumerate}
and (2) when $p=2$ and $\eps_n\lesssim \l\frac{\log n}{n}\r^{\frac{1}{d+4}}$ that
\begin{enumerate}
\item if $\beta_n\gg \eps_n^2$ and $\beta_n\gg \frac{\log n}{n\eps_n^d}$ the limit is well-posed, and
\item if $\beta_n\ll \eps_n^2$ the limit is ill-posed.
\end{enumerate}
Hence, there is a gap in the regime $\eps_n^2\ll\beta_n\lesssim \frac{\log n}{n\eps_n^d}$ which our results do not cover.
We conjecture that when $\eps_n^2\ll\beta_n\lesssim \frac{\log n}{n\eps_n^d}$ we remain in the well-posed regime.

Our results do not establish a well-posed regime for $p\neq 2$. We also conjecture that the ill-posed regime is sharp (upto perhaps logarithms), and therefore if $\eps_n^p\gg \beta_n$ or $n\eps_n^p\ll \log n$ then the limit is asymptotically well-posed.
Indeed, whenever $n\eps_n^p\ll 1$ (which can only happen when $p>d$) uniform convergence has already been established, see~\cite[Lemma 4.5]{slepcev19}, and hence the problem is already well-posed with finitely many constraints (formally corresponding to $\beta_n\sim \frac{1}{n}$).
Since we used the random walk interpretation of minimizers, which is specific to $p=2$, the techniques in this paper do not immediately generalize to prove well-posedness for the variational $p$-Laplacian for $p\neq 2$.

\section*{Acknowledgements}

JC was supported by NSF DMS Grant 1713691, and is grateful for the hospitality of the Center for Nonlinear Analyis at Carnegie Mellon University, and to Marta Lewicka for helpful discussions.  
DS is grateful to NSF for support via grant DMS-1814991. 
MT is grateful for the hospitality of the Center for Nonlinear Analysis at Carnegie Mellon University and the School of Mathematics at the University of Minnesota, for the support of the Cantab Capital Institute for the Mathematics of Information and Cambridge Image Analysis at the University of Cambridge, and has received funding from the European Research Council under the European Union's Horizon 2020 research and innovation programme grant agreement No 777826 (NoMADS) and grant agreement No 647812.

\appendixpage
\appendix

\section{Concentration Inequalities}\label{sec:azuma}
For completeness, we include some inequalities from probability theory.
We start with Azuma's inequality for supermartingales.

\begin{theorem}[Azuma's inequality]\label{thm:Azuma}
Let $X_0,X_1,X_2,X_3,\dots$ be a supermartingale with respect to a filtration $\F_1,\F_2,\F_3,\dots$ (i.e., $\E[X_{k}-X_{k-1}\left| \F_{k-1}\right.] \leq 0$). Assume that conditioned on $\F_{k-1}$ we have  $|X_{k}-X_{k-1}|\leq r$ almost surely for all $k$. Then for any $\vartheta>0$
\begin{equation}\label{eq:Azuma}
\P(X_k-X_0 \geq \vartheta) \leq \exp\left(-\frac{\vartheta^2}{2kr^2}\right).
\end{equation}
\end{theorem}
\begin{proof}
We use the usual Chernoff bounding method to obtain
\[\P(X_k-X_0 \geq \vartheta)=\P\left(e^{s(X_k-X_0)}\geq e^{s\vartheta}\right)\leq e^{-s\vartheta}\E\left[e^{s(X_k-X_0)}\right]=e^{-s\vartheta}\E\left[ e^{s\sum_{i=1}^k(X_i-X_{i-1} )}\right],\]
for $s>0$ to be determined. Since $|X_k-X_{k-1}|\leq r$ conditioned on $\F_{k-1}$, we use convexity of $x\mapsto e^{sx}$ to obtain
\begin{align*}
\E\left[ \left. e^{s(X_k-X_{k-1} )} \right| \F_{k-1} \right] &\leq \E\left[ \left. e^{-sr} + \left( \frac{X_k-X_{k-1}+r}{r} \right)\sinh(sr) \right| \F_{k-1} \right]\\
&= e^{-sr} + \left( \frac{\E[X_k-X_{k-1}\left| \F_{k-1}\right.]+r}{r} \right)\sinh(sr)\\
&\leq e^{-sr} + \sinh(sr) = \cosh(sr)\leq e^{\frac{s^2r^2}{2}}.
\end{align*}
Therefore
\[\E\left[ e^{s\sum_{i=1}^k(X_i-X_{i-1} )}\right]=\E\left[e^{s\sum_{i=1}^{k-1}(X_i-X_{i-1} )} \E\left[ \left. e^{s(X_k-X_{k-1} )} \right| \F_{k-1} \right] \right]\leq e^{\frac{s^2r^2}{2}}\E\left[e^{s\sum_{i=1}^{k-1}(X_i-X_{i-1} )}  \right].\]
Continuing by induction we find that
\[\P(X_k-X_0 \geq \vartheta) \leq \exp\left(-s\vartheta + \frac{ks^2r^2}{2}\right).\]
Choosing $s=\vartheta/kr^2$ completes the proof.
\end{proof}

Next, we recall a concentration inequality from \cite{calder18a}.

\begin{lemma}[{\cite[Remark 7]{calder18AAA}}]
\label{lem:WellPosed:RWBounds:ConcIneq}
Let $Y_1,Y_2,Y_3,\dots,Y_n$ be a sequence of \emph{i.i.d}~random variables on $\bbR^d$ with Lebesgue density $\rho:\bbR^d\to \bbR$, let $\psi:\bbR^d \to \bbR$ be bounded and Borel measurable with compact support in a ball $B(x,r)$ for some $r>0$, and define
\[ Y = \sum_{i=1}^n \psi(Y_i).\]
Then, for any $0 \leq \vartheta \leq 1$,
\[ \bbP\l |Y-\bbE(Y)|\geq c\|\psi\|_{\Linfty(B(x,r))} nr^d\vartheta\r \leq 2\exp(-Cnr^d\vartheta^2), \]
where $c>0$, $C>0$ are constants depending only on $\|\rho\|_{\Linfty}$ and $d$. 
\end{lemma}

We also recall Bernstein's inequality \cite{boucheron2013concentration}. For $Y_1,\dots,Y_n$ \emph{i.i.d.}~with variance $\sigma^2 = \E((Y_i-\E[Y_i])^2)$, if $|Y_i|\leq M$ almost surely for all $i$ then Bernstein's inequality states that for any $\vartheta>0$
\begin{equation}\label{eq:bernstein}
\P\left( \left| \sum_{i=1}^n Y_i - \E[Y_i] \right|> n\vartheta\right)\leq 2\exp\left( -\frac{n\vartheta^2}{2\sigma^2 + 4M\vartheta/3} \right).
\end{equation}

\section{\texorpdfstring{$\TLp$}{TLp} Convergence of Minimizers} \label{sec:app:TLpConv}

The $\TLp$ topology was introduced in~\cite{garciatrillos16} to define a discrete-to-continuum convergence for variational problems on graphs (as is the setting in this paper).
The idea is to consider discrete, and continuum, functions as pairs: $(\mu,u)$ where $\mu\in\cP(\Omega)$ and $u\in \Lp(\mu)$.
For example, in the discrete setting we choose $\mu_n = \frac{1}{n} \sum_{i=1}^n \delta_{x_i}$, where $x_i\iid\mu\in\cP(\Omega)$, to be the \emph{empirical measure} then $u_n\in \Lp(\mu_n)$ implies that $u_n:\Omega_n\to \bbR$.
To define a metric we work on the space:
\[ \TLp(\Omega) := \lb (\mu,u) \,:\, \mu\in\cP_p(\Omega) \text{ and } u\in \Lp(\mu) \rb. \]
This space is a metric with
\begin{equation} \label{eq:Prelim:TLp:TLpKant}
\dTLp((\mu,u),(\nu,v)) := \inf_{\pi\in \Pi(\mu,\nu)} \sqrt[p]{\int_{\Omega\times \Omega} |x-y|^p + |u(x) - v(y)|^p \, \dd \pi(x,y)}
\end{equation}
where $\Pi(\mu,\nu)$ is the subset of probability measures on $\Omega\times\Omega$ such that the first marginal is $\mu$ and the second marginal is $\nu$.
We call any $\pi\in\Pi(\mu,\nu)$ a \emph{transport plan}.
The proof that $(\TLp,\dTLp)$ is a metric space follows from its connection to optimal transport, we refer to~\cite[Remark 3.4]{garciatrillos16} for more details.

In the setting of this paper we can characterize $\TLp$ convergence as follows (the following holds due to existence of a density $\rho$ of $\mu$).
A function $T:\Omega\to\Omega$ is a \emph{transport map} between $\mu$ and $\nu$ if $T_{\#}\mu=\nu$, where the pushforward of a measure is defined by
\[ T_{\#} \mu(A) = \mu(T^{-1}(A)) = \mu\l\lb x\in \Omega\,:\, T(x) \in A\rb\r \qquad \text{for any measurable } A\subset \Omega. \]
In the notation of transport maps the $\TLp$ distance can be written 
\begin{equation} \label{eq:Prelim:TLp:TLpMong}
\dTLp((\mu,u),(\nu,v)) = \inf_{T_{\#}\mu = \nu} \sqrt[p]{\int_\Omega |x-T(x)|^p + |u(x) - v(T(y))|^p \, \dd \mu(x)}.
\end{equation}
(In general~\eqref{eq:Prelim:TLp:TLpKant} and~\eqref{eq:Prelim:TLp:TLpMong} are not equivalent but in special cases -- such as in the setting of this paper -- the two formulations coincide, in optimal transport~\eqref{eq:Prelim:TLp:TLpKant} would be called the Kantorovich formulation and~\eqref{eq:Prelim:TLp:TLpMong} the Monge formulation.)
The following result can be found in~\cite[Proposition 3.12]{garciatrillos16}.

\begin{proposition}
\label{prop:app:TLpEquiv}
Let $\Omega\subset\bbR^d$ be open, $(\mu,u),(\mu_n,u_n)\in\TLp(\Omega)$ for all $n\in \bbN$ and assume $\mu$ is absolutely continuous with respect to the Lebesgue measure.
Then, $(\mu_n,u_n)\toTLp (\mu,u)$ if and only if $\mu_n\weakstarto \mu$ and for any sequence of transportation maps $T_n$ satisfying $(T_n)_{\#}\mu = \mu_n$ and $\|\Id - T_n\|_{\Lone(\mu)}\to 0$ we have
\[ \int_\Omega |u(x) - u_n(T_n(x))|^p \, \dd \mu(x) \to 0. \] 
\end{proposition}

In our context the sequence of measures $\mu_n$ are the empirical measure which, with probability one, converge weak$^*$ to the true data generating measure $\mu$ when data points are iid.
Hence, it is enough to find a transportation map converging to the identity.
With an abuse of the definition we will often say $u_n$ converges to $u$ in $\TLp$ when we mean $(\mu_n,u_n)$ converges to $(\mu,u)$ in $\TLp$. 
\vspace{\baselineskip}

With the above notion of convergence we can define a topology in which to study variational limits.
In particular the $\TLp$ space gives us a way to define $\Gamma$-convergence of discrete-to-continuum functionals.
We recall the definition of almost sure $\Gamma$-convergence.

\begin{mydef}[$\Gamma$-convergence]
\label{def:Prelim:TLp:Gamma}
Let $(Z,d)$ be a metric space, $\Lzero(Z;\bbR\cup\{\pm \infty\})$ be the set of measurable functions from $Z$ to $\bbR\cup\{\pm \infty\}$, and $(\cX,\bbP)$ be a probability space.
The function $\cX\ni\omega\mapsto E_n^{(\omega)} \in \Lzero(Z;\bbR\cup\{\pm \infty\})$ is a random variable.
We say $E_n^{(\omega)}$ $\Gamma$-converges almost surely on the domain $Z$ to $E_\infty :Z\to \bbR\cup\{\pm\infty\}$ with respect to $d$, and write $E_\infty = \Glim_{n \to \infty} E_n^{(\omega)}$, if there exists a set $\cX^\prime\subset \cX$ with $\bbP(\cX^\prime) = 1$, such that for all $\omega\in \cX^\prime$ and all $f\in Z$:
\begin{itemize}
\item[(i)] (liminf inequality) for every sequence $\{f_n\}_{n=1}^\infty$ converging to $f$
\[ E_\infty(f) \leq \liminf_{n\to \infty} E_n^{(\omega)}(f_n), \text{ and } \]
\item[(ii)] (recovery sequence) there exists a sequence $\{f_n\}_{n=1}^\infty$ converging to $f$ such that
\[ E_\infty(f) \geq \limsup_{n\to \infty} E_n^{(\omega)}(f_n). \]
\end{itemize}
\end{mydef}

The key property of $\Gamma$-convergence is that, when combined with a compactness result, it implies the convergence of minimizers.
In particular, the following theorem is fundamental in the theory of $\Gamma$-convergence.

\begin{theorem}[Convergence of Minimizers]
\label{thm:Prelim:TLp:Conmin}
Let $(Z,d)$ be a metric space and $(\cX,\bbP)$ be a probability space.
The function $\cX\ni\omega\mapsto E_n^{(\omega)} \in \Lzero(Z;\bbR\cup\{\pm \infty\})$ is a random variable.
Let $f_n^{(\omega)}$ be a minimizing sequence for $E_n^{(\omega)}$.
If, with probability one, the set $\{f_n^{(\omega)}\}_{n=1}^\infty$ is pre-compact and $E_\infty = \Glim_n E_n^{(\omega)}$ where $E_\infty:Z\to[0,\infty]$ is not identically $+\infty$ then, with probability one,
\[ \min_Z E_\infty = \lim_{n\to \infty} \inf_Z E_n^{(\omega)}. \]
Furthermore any cluster point of $\{f_n^{(\omega)}\}_{n=1}^\infty$ is almost surely a minimizer of $E_\infty$.
\end{theorem}

The theorem is also true if we replace minimizers with almost minimizers.

We recall the definition of our discrete unconstrained functional $\cEpneps$, defined by~\eqref{eq:MainRes:illposed:Lp}, and our continuum unconstrained functional $\cEpinfty$, defined by~\eqref{eq:MainRes:illposed:cEpinfty}.
When $p=1$ it was shown in~\cite{garciatrillos16} that, with probability one, $\Glim_{n\to\infty} \cE^{(1)}_{n,\eps_n} = \cE^{(1)}_\infty$ and $\cE^{(1)}_{n,\eps_n}$ satisfies a compactness property where $\cE^{(1)}_\infty$ is a weighted total variation norm.
The proof generalizes almost verbatim for $p>1$ with the additional condition that, if $d=2$, $\eps_n\gg \frac{(\log n)^{\frac34}}{\sqrt{n}}$.
The additional assumption when $d=2$ has already been shown to be unnecessary.
For example, in~\cite{garciatrillos16a} the authors use the $\Gamma$-convergence result (with the more restrictive lower bound for $d=2$) to prove convergence of Cheeger and Ratio  cuts, this lower bound was removed in~\cite{muller18AAA} using a refined grid matching technique within the $\Gamma$-convergence argument.
Later results, i.e.~\cite{caroccia19,calder2019improved}, avoid the additional assumption via comparing the empirical measure measure to an intermediary measure; we follow this argument below.
For the following result we do not need the compact support assumption in {\bf (A3)} and so we restate the third assumption.

\begin{enumerate}
\item[\bf (A3')] The interaction potential $\eta:[0,\infty)\to[0,\infty)$ is non-increasing, positive and continuous at $t=0$.
We define $\eta_\eps = \frac{1}{\eps^d} \eta(\cdot/\eps)$ and assume $\sigma_\eta:= \int_{\bbR^d} \eta(|x|) |x_1|^2 \, \dd x<\infty$.
\end{enumerate}

\begin{proposition}
\label{prop:Prelim:TLp:GamConvEn}
Assume {\bf (A1,A2,A3')}, $\eps_n\gg\sqrt[d]{\frac{\log n}{n}}$ and $p> 1$ we define $\cEpneps$ by~\eqref{eq:MainRes:illposed:Lp} and $\cEpinfty$ by~\eqref{eq:MainRes:illposed:cEpinfty}.
Then, with probability one,
\[ \Glim_{n\to\infty} \cEpnepsn = \cEpinfty. \]
Furthermore, if $\{u_n\}_{n=1}^\infty$ is a sequence satisfying $\sup_{n\in \bbN} \|u_n\|_{\Lp(\mu_n)}<\infty$ and $\sup_{n\in\bbN}\cEpnepsn(u_n)<\infty$ then $\{u_n\}_{n=1}^\infty$ is pre-compact in $\TLp$ and any limit point is in $\Wkp{1}{p}(\Omega)$.
\end{proposition}

\begin{proof}
The proof for $d\geq 3$, or $d=2$ with the additional constraint that $\eps_n\gg \frac{(\log n)^{\frac34}}{\sqrt{n}}$, was stated in~\cite[Theorem 4.7]{slepcev19} for $p>1$ and the proof is a simple adaptation of the $p=1$ case which was given in~\cite{garciatrillos16}.
Hence, we only prove the case for $d=2$ here.

By either~\cite[Lemma 3.1]{caroccia19} or~\cite[Proposition 2.10]{calder2019improved} there exists a probability measure $\tilde{\mu}_n$ with density $\tilde{\rho}_n$ such that, with probability one, there exists $\tilde{T}_n:\Omega\to\Omega_n$ and $\theta_n\to 0$ with the property that $\tilde{T}_{n\#}\tilde{\mu}_n = \mu_n$, $\|\tilde{T}_n-\Id\|_{\Linfty(\Omega)}\ll \l\frac{\log n}{n}\r^{\frac{1}{d}}$ and $\|\rho - \tilde{\rho}_n\|_{\Linfty(\Omega)}\leq \theta_n$.
The proof is divided into three parts corresponding to the compactness property, the liminf inequality and the recovery sequence.

\paragraph{Compactness property.}
Assume $\sup_{n\in \bbN} \|u_n\|_{\Lp(\mu_n)}<\infty$ and $\sup_{n\in\bbN}\cEpnepsn(u_n)<\infty$.
Find $a>0$ and $b>0$ such that $\eta \geq \tilde{\eta}$ where $\tilde{\eta}(t) = a$ for all $|t|\leq b$ and $\tilde{\eta}(t)=0$ for all $|t|>b$.
Let $\tilde{u}_n = u_n\circ \tilde{T}_n$.
Then,
\begin{align*}
\cEpn(u_n) & \geq \frac{1}{\eps_n^p} \int_{\Omega^2} \tilde{\eta}_{\eps_n}(|x-y|) |u_n(x) - u_n(y)|^p \, \dd \mu_n(x) \, \dd \mu_n(y) \\
 & = \frac{1}{\eps_n^p} \int_{\Omega^2} \tilde{\eta}_{\eps_n}\l |\tilde{T}_n(x) - \tilde{T}_n(y)|\r \la \tilde{u}_n(x) - \tilde{u}_n(y)\ra^p \, \dd \tilde{\mu}_n(x) \, \dd \tilde{\mu}_n(y) \\
 & \geq \frac{1}{\eps_n^{d+p}} \int_{\Omega^2} \tilde{\eta}\l\frac{|x-y|}{\tilde{\eps}_n}\r \la \tilde{u}_n(x) - \tilde{u}_n(y)\ra^p \, \dd \tilde{\mu}_n(x) \, \dd \tilde{\mu}_n(y),
\end{align*}
since $\tilde{\eta}\l\frac{|x-y|}{\tilde{\eps}_n}\r \leq \tilde{\eta}\l\frac{|\tilde{T}_n(x)-\tilde{T}_n(y)|}{\eps_n}\r$ where $\tilde{\eps}_n = \eps_n - \frac{2}{b} \|\tilde{T}_n - \Id\|_{\Linfty(\Omega)}$.
Hence,
\[ \cEpn(u_n) \geq \frac{\tilde{\eps}_n^{d+p}}{\eps_n^{d+p}} \l 1 - \frac{\theta_n}{\rho_{\min}}\r^2 \frac{1}{\tilde{\eps}_n^p} \int_{\Omega^2} \tilde{\eta}_{\tilde{\eps}_n}(|x-y|) \la \tilde{u}_n(x)-\tilde{u}_n(y)\ra^p \rho(x) \rho(y) \, \dd x \, \dd y = \alpha_n \cENLp{\tilde{\eps}_n}(\tilde{u}_n) \]
where $\alpha_n = \frac{\tilde{\eps}_n^{d+p}}{\eps_n^{d+p}} \l 1 - \frac{\theta_n}{\rho_{\min}}\r^2 \to 1$ and $\cENLp{\eps}$ is defined in~\eqref{eq:app:cENLp} with $\eta = \tilde{\eta}$.
We also have
\[ \sup_{n\in\bbN} \| \tilde{u}_n\|_{\Lp(\mu)}^p \leq \sup_{n\in\bbN} \frac{\|\tilde{u}_n\|_{\Lp(\tilde{\mu}_n)}^p}{1-\frac{\theta_n}{\rho_{\min}}} = \sup_{n\in\bbN}  \frac{\|u_n\|_{\Lp(\mu_n)}^p}{1-\frac{\theta_n}{\rho_{\min}}} < +\infty. \]
By Theorem~\ref{thm:app:NLConv} below $\{\tilde{u}_n\}_{n\in\bbN}$ is precompact in $\Lp(\mu)$, and hence there exists a subsequence (relabeled) such that $\tilde{u}_n = u_n\circ \tilde{T}_n\to u$ in $\Lp(\mu)$.
Now as $\tilde{\mu}_n\weakstarto \mu$ there exists an invertible transport map $S_n$ such that $\tilde{\mu}_n = S_{n\#}\mu$ and $S_n\to \Id$ in $\Lp(\mu)$.
Now choose $T_n = \tilde{T}_n\circ S_n$ (note that $T_{n\#}\mu = \mu_n$) so, assuming $n$ is sufficiently large such that $\min_{x\in\Omega} \tilde{\rho}_n(x) \geq \frac{\rho_{\min}}{2}$, then
\begin{align*}
\l \int_\Omega \la u_n(T_n(x)) - u(x)\ra^p \, \dd \mu(x) \r^{\frac{1}{p}} & = \l \int_\Omega \la \tilde{u}_n(S_n(x)) - u(x)\ra^p \, \dd [S_n^{-1}]_{\#} \tilde{\mu}_n(x) \r^{\frac{1}{p}} \\
 & = \l \int_\Omega \la \tilde{u}_n(x) - u(S_n^{-1}(x))\ra^p \, \dd \tilde{\mu}_n(x) \r^{\frac{1}{p}} \\
 & \leq \l \int_\Omega \la \tilde{u}_n(x) - u(x)\ra^p \, \dd \tilde{\mu}_n(x) \r^{\frac{1}{p}} \\
 & \hspace{1cm} + \l \int_\Omega \la u(S_n^{-1}(x)) - u(x)\ra^p \, \dd \tilde{\mu}_n(x) \r^{\frac{1}{p}} \\
 & = \| \tilde{u}_n - u\|_{\Lp(\tilde{\mu}_n)} + \| u - u\circ S_n\|_{\Lp(\mu)}.
\end{align*}
The first term above goes to zero since we already established convergence of $\tilde{u}_n$ to $u$ in $\Lp(\mu)$ (which bounds the $\Lp(\tilde{\mu}_n)$ norm), and the second term goes to zero by~\cite[Lemma 3.10]{garciatrillos16} since $S_n\to \Id$ in $\Lp(\mu)$ (see also $\Lp$ convergence of translations).
By Proposition~\ref{prop:app:TLpEquiv} $u_n\to u$ in $\TLp$.

\paragraph{Liminf inequality.}
Let $u_n\to u$ in $TL^p$.
We start by assuming $\eta=\tilde{\eta}$ where $\tilde{\eta}$ is given in the compactness proof.
Following the argument in the compactness proof we have
\[ \liminf_{n\to\infty} \cEpn(u_n) \geq \liminf_{n\in\infty} \alpha_n\cENLp{\tilde{\eps}_n}(\tilde{u}_n) \geq \cEpinfty(u) \]
with the last inequality following from the $\Gamma$-convergence of $\cENLp{\tilde{\eps}_n}$ (Theorem~\ref{thm:app:NLConv}).
The proof continues as in the proof of~\cite[Theorem 1.1]{garciatrillos16} by generalising to piecewise constant $\eta$ with compact support, then to compactly supported $\eta$, and finally to non-compactly supported $\eta$.

\paragraph{Recovery sequence.}
It is enough to prove the recovery sequence for $u\in \Wkp{1,p}(\Omega)\cap \Lip$.
In which case we can define $u_n=u\lfloor_{\Omega_n}$ and it is straightforward to show that $u_n\to u$ in $\TLp$.
Assume that $\eta=\tilde{\eta}$ is again as defined in the compactness proof.
One has $\tilde{\eta}\l\frac{|x-y|}{\tilde{\eps}_n}\r \geq \tilde{\eta}\l\frac{|\tilde{T}_n(x)-\tilde{T}_n(y)|}{\eps_n}\r$ where now we define $\tilde{\eps}_n = \eps_n+\frac{2}{b}\|\tilde{T}_n-\Id\|_{\Linfty(\Omega)}$.
A very similar calculation as in the compactness property implies $\cEpn(u_n) \leq \beta_n \cENLp{\tilde{\eps}_n}(\tilde{u}_n)$ where $\beta_n = \frac{\tilde{\eps}_n^{d+p}}{\eps_n^{d+p}}\l 1 + \frac{\theta_n}{\rho_{\min}}\r^2 \to 1$.
Hence, by Theorem~\ref{thm:app:NLConv}(2) we have $\limsup_{n\to\infty} \cEpn(u_n) \leq \cEpinfty(u)$.
The proof generalizes to any $\eta$ satisfying Assumption {\bf (A3')} as in the liminf inequality.
\end{proof}

The following theorem was stated in~\cite[Theorem 4.1]{garciatrillos16} for $p=1$ and generalizes easily to $p>1$.
Part (1) was also stated in~\cite[Lemma 4.6]{slepcev19}, and (2) is either contained within the proof of~\cite[Theorem 4.1]{garciatrillos16} or can be arrived at easily from the characterisation of $\Wkp{1}{p}$ found, for example, in~\cite[Theorem 10.55]{leoni09}.
We include the result here for convenience.

\begin{theorem}
\label{thm:app:NLConv}
Let $\Omega\subset\bbR^d$ be open, bounded and with Lipschitz boundary, let $\rho:\Omega\to\bbR$ be continuous and bounded from above and below by positive constants, let $\eta$ satisfy {\bf (A3)}, and let $\mu$ be the measure with density $\rho$.
Define $\cENLp{\eps}$ by
\begin{equation} \label{eq:app:cENLp}
\cENLp{\eps}(u) =  \frac{1}{\eps^p} \int_{\Omega^2} \eta_{\eps}(|x-y|) \la u(x)-u(y)\ra^p \rho(x) \rho(y) \, \dd x \, \dd y
\end{equation}
and $\cEpinfty$ by~\eqref{eq:MainRes:illposed:cEpinfty}.
Then,
\begin{enumerate}
\item[(1)] $\Glim_{\eps\to 0} \cENLp{\eps} = \cEpinfty$,
\item[(2)] if $u\in \Wkp{1,p}(\Omega)$ then $u_n = u$ is a recovery sequence, and 
\item[(3)] if $\eps_n\to 0$ and $\{u_n\}_{n\in\bbN}$ satisfies $\sup_{n\in\bbN} \| u_n\|_{\Lp(\mu)}<+\infty$ and $\sup_{n\in\bbN} \cENLp{\eps_n}(u_n)<+\infty$ then $\{u_n\}_{n\in\bbN}$ is precompact in $\Lp(\mu)$.
\end{enumerate}
\end{theorem}

%\nocite{*}
%\bibliographystyle{plain}
\bibliographystyle{abbrv}
\bibliography{ref_NumDataSSL}

\end{document}